\documentclass[11pt,reqno]{amsart}

\usepackage{ifpdf}
\usepackage{amsmath,amsfonts,amsthm,amsopn,amssymb}
\usepackage{verbatim,wasysym,mathrsfs}
\usepackage[margin=2.75cm, heightrounded]{geometry}
\usepackage{cite,marginnote}
\usepackage{color,enumitem,graphicx}
\usepackage[colorlinks=true,urlcolor=blue, citecolor=red,linkcolor=blue,
linktocpage,pdfpagelabels, bookmarksnumbered,bookmarksopen]{hyperref}
\usepackage[english]{babel}
\usepackage[T1]{fontenc}
\usepackage{lmodern}

\newtheorem{Thm}{Theorem}[section]
\newtheorem{Lem}[Thm]{Lemma}
\newtheorem{Assum}[Thm]{Assumption}
\newtheorem{Cor}[Thm]{Corollary}
\newtheorem{Prop}[Thm]{Proposition}

\theoremstyle{definition}
\newtheorem{Def}[Thm]{Definition}
\newtheorem{Rem}{Remark}[section]

\newtheorem*{Def*}{Definition}

\newcommand{\R}{{\mathbb R}}

\allowdisplaybreaks

\numberwithin{equation}{section}

\begin{document}

\title[upper critical Choquard equation in dimension three]{Existence and asymptotics for the upper critical Choquard equation in dimension three}

\author[]{Jinkai Gao}
\address[Jinkai Gao]{\newline\indent School of Mathematical Sciences,
\newline\indent Nankai University,
\newline\indent Tianjin, 300071, China.}
    \email{\href{mailto:jinkaigao@mail.nankai.edu.cn}{jinkaigao@mail.nankai.edu.cn}}

\date{\today}



\begin{abstract}
In this paper, we are interested in the existence and asymptotic behavior of least energy solutions to the upper critical Choquard equation
\begin{equation*}
    \begin{cases}
        -\Delta u+au=\displaystyle\left(\int_{\Omega}\frac{u^{6-\alpha}(y)}{|x-y|^\alpha}dy\right)u^{5-\alpha}&\mbox{in}\ \Omega,\\
         u>0 \ \  &\mbox{in}\ \Omega,\\
 u=0 \ \  &\mbox{on}\ \partial \Omega,
    \end{cases}
\end{equation*}
where $\Omega \subset \R^{3}$ is a bounded domain with a $C^{2}$ boundary, $\alpha \in (0,3)$, $a \in C(\overline{\Omega}) \cap C^{1}(\Omega)$, and the operator $-\Delta + a$ is coercive. We first establish that the following three properties are equivalent: the existence of least energy solutions, the validity of a strict inequality in the associated minimization problem, and the positivity of the Robin function somewhere in the domain. This leads naturally to the definition of a critical function $a$. Under the perturbation $a \mapsto a + \varepsilon V$ with $a$ critical and $V \in L^{\infty}(\Omega)$, we prove that least energy solutions exist. Furthermore, we establish a refined energy estimate and describe their asymptotic profile.
\end{abstract}

\maketitle

\section{Introduction}
In this paper, we are concerned with the following Choquard equation
\begin{equation}\label{main choquard equation with varepsilon}
    \begin{cases}
        -\Delta u+au=\displaystyle\left(\int_{\Omega}\frac{u^{6-\alpha}(y)}{|x-y|^\alpha}dy\right)u^{5-\alpha} &\mbox{in}\ \Omega,\\
  u>0
  \ \  &\mbox{in}\ \Omega,\\
  u=0
  \ \  &\mbox{on}\ \partial \Omega,
    \end{cases}
\end{equation}
where $\Omega\subset\R^{3}$ is a bounded domain with a $C^{2}$ boundary, $a\in C(\bar{\Omega
})\cap C^{1}(\Omega)$, $\alpha\in (0,3)$, and $6-\alpha$ is the upper critical exponent in the sense of the Hardy-Littlewood-Sobolev inequality. Furthermore, the operator $-\Delta+a$ is assumed to be coercive on $H^{1}_{0}(\Omega)$, that is, there exists a constant $c>0$ such that\footnote{This coercivity condition is equivalent to the positivity of the first Dirichlet eigenvalue of the operator $-\Delta+a$.}
\begin{equation*}
    \int_{\Omega}|\nabla u|^2+au^{2}dx\geq c\int_{\Omega}|\nabla u|^{2}dx,\text{~for all~}u\in H^{1}_{0}(\Omega),
\end{equation*}
which is a necessary condition for the solvability of \eqref{main choquard equation with varepsilon} (see \cite[Lemma C.1]{Konig2022FineMA}). In particular, the case $a \equiv -\lambda$ is admissible for any $\lambda \in (0, \lambda_1)$, where $\lambda_1$ is the first Dirichlet eigenvalue of $-\Delta$ on $\Omega$.

The Choquard equation has appeared in various physical contexts. It was first introduced by Fr\"{o}hlich \cite{FrohlichHerbert} and Pekar \cite{Pekar} in the context of polaron modeling, and later by Choquard for modelling a one-component plasma \cite{Lieb1976SAM}. It also arises as the Schr\"{o}dinger--Newton equation in models coupling the Schr\"{o}dinger equation of quantum physics with nonrelativistic Newtonian gravity \cite{Bahrami-2014,Penrose1998QuantumCE}. For further discussion on the physical backgrounds of this equation, we refer to the survey by Moroz and Van Schaftingen \cite{Moroz2017JFPTA} and references therein.

An important equation closely related to \eqref{main choquard equation with varepsilon} is the upper critical Choquard equation in the whole space
\begin{equation}\label{critical choquard equation}
    -\Delta u=\left(\int_{\R^{3}}\frac{u^{6-\alpha}(y)}{|x-y|^{\alpha}}dy\right)u^{5-\alpha}\quad\text{in~}\ \R^3.
\end{equation}
Recall the family of Aubin-Talenti functions
\begin{equation*}
    U_{\xi, \lambda}(x) = \frac{\lambda^{1/2}}{(1 + \lambda^2|x-\xi|^2)^{1/2}},\quad \xi\in\R^{3},\lambda\in\R^{+}
\end{equation*}
and the sharp Sobolev constant
\begin{equation}\label{best sobolev constant}
    S=3 \left( \frac\pi 2 \right)^{4/3}.  
\end{equation}
It is known \cite{Miao-2015,Lei-2018,Du2019,Guo2019,LePhuong-2020} that all positive solutions to equation \eqref{critical choquard equation} are given by
\begin{equation}\label{definition-of-U-bar}
	\bar{U}_{\xi, \lambda}(x)=3^{\frac{1}{4}}S^{-\frac{3-\alpha}{4(5-\alpha)}}C_{\alpha}^{-\frac{1}{2(5-\alpha)}}U_{\xi,\lambda}(x):=\bar{C}_{\alpha}U_{\xi,\lambda}(x),\quad \xi\in\R^{3},\lambda\in\R^{+},
\end{equation}
where $C_{\alpha}$ is the sharp constant for the Hardy–Littlewood–Sobolev inequality defined in \eqref{definition of C alpha}. Furthermore, the upper critical Choquard equation \eqref{critical choquard equation} arises as the Euler--Lagrange equation of the minimization problem 
\begin{equation}\label{optimization problem for shl}
    S_{HL}:=\inf_{u\in\dot{H}^{1}(\R^{3}), u\not\equiv 0}\frac{\int_{\R^3} |\nabla u|^2dx}{\left(\int_{\R^{3}}\int_{\R^{3}}\frac{|u(x)|^{6-\alpha}|u(y)|^{6-\alpha}}{|x-y|^{\alpha}}dxdy\right)^{\frac{1}{6-\alpha}}}.
\end{equation}
This problem is equivalent to the Hardy-Littlewood-Sobolev-type inequality
\begin{equation*}
    \int_{\R^{3}}|\nabla u|^{2}dx\geq S_{HL}\left(\int_{\R^{3}}\int_{\R^{3}}\frac{|u(x)|^{6-\alpha}|u(y)|^{6-\alpha}}{|x-y|^{\alpha}}dxdy\right)^{\frac{1}{6-\alpha}},\quad\forall u\in\dot{H}^{1}(\R^{3}),
\end{equation*}
and the sharp constant is given by
\begin{equation}\label{definition of SHL}
    S_{HL}=SC_{\alpha}^{-\frac{1}{6-\alpha}}.
\end{equation}
In addition, the moving sphere method yields that the function $U_{\xi, \lambda}$ satisfies
\begin{equation}\label{U-x-lamda for choquard equation}
    -\Delta U_{\xi,\lambda}=\bar{C}_{\alpha}^{2(5-\alpha)}\left(\int_{\R^{3}}\frac{U_{\xi,\lambda}^{6-\alpha}(y)}{|x-y|^{\alpha}}dy\right)U_{\xi,\lambda}^{5-\alpha}\quad\text{in}\ \R^3,
\end{equation}
and
\begin{equation}\label{important-identity-1}
    \int_{\R^{3}}\frac{U_{\xi,\lambda}^{6-\alpha}(y)}{|x-y|^{\alpha}}dy=\frac{3}{\bar{C}_{\alpha}^{2(5-\alpha)}}U_{\xi,\lambda}^{\alpha}(x)\quad\text{in}\ \R^3,
\end{equation}
where $\bar{C}_{\alpha}$ is the constant defined in \eqref{definition-of-U-bar}.

To find solutions of \eqref{main choquard equation with varepsilon}, a natural approach is to consider the minimization problem
\begin{equation}\label{minimizing-problem}
    S_{HL}(a) :=\inf_{u\in H^1_0(\Omega), \|u\|_{HL}=1}\int_\Omega |\nabla u|^2+ a u^2dx,
\end{equation}
where \begin{equation}\label{defin-norm-HL}
    \|u\|_{HL}:=\left(\int_{\Omega}\int_{\Omega}\frac{|u(x)|^{6-\alpha}|u(y)|^{6-\alpha}}{|x-y|^{\alpha}}dxdy\right)^{\frac{1}{2(6-\alpha)}}.
\end{equation} 
Indeed, if $v_{a}$ is a minimizer of $S_{HL}(a)$ with $\|v_{a}\|_{HL}=1$, then a standard variational argument shows that $u_{a}:=S_{HL}(a)^{\frac{1}{2(5-\alpha)}}|v_{a}|$ is a solution of \eqref{main choquard equation with varepsilon} and, in fact, a least energy solution. However, when $a \equiv 0$, it is known that $S_{HL}(0) = S_{HL}$ and that $S_{HL}(0)$ is never achieved unless $\Omega = \mathbb{R}^{3}$, as shown in \cite[Lemma 1.3]{Gao2016TheBT}. On the other hand, the Pohozaev identity \eqref{pohozaev-identiy} yields that if
\begin{equation*}
    a+\frac{(x,\nabla a)}{2}\geq 0
\end{equation*}
and $\Omega$ is a strictly star-shaped domain, then \eqref{main choquard equation with varepsilon} has no solutions. {A natural question arises: Under what conditions on the function $a$ is $S_{HL}(a)$ achieved ?} Problems of this type were first studied by Brezis and Nirenberg \cite{Brezis1983CPAM} concerning the following equation
\begin{equation}\label{brezis-nirenberg-problem}
    \begin{cases}
        -\Delta u+au=u^{2^*-1} &\mbox{in}\ \Omega,\\
  u>0
  \ \  &\mbox{in}\ \Omega,\\
  u=0
  \ \  &\mbox{on}\ \partial \Omega,
    \end{cases}
\end{equation}
where $N\geq3$, $\Omega\subset\R^{N}$ is a bounded domain with a $C^{2}$ boundary, $a\in C(\bar{\Omega})\cap C^{1}(\Omega)$, and $2^*:=\frac{2N}{N-2}$ is the Sobolev critical exponent.
Let 
\begin{equation*}
S(a):=\inf_{u\in H^1_0(\Omega), \|u\|_{L^{2^*}(\Omega)}=1}\int_\Omega |\nabla u|^2+ a u^2dx.
\end{equation*}
When $N\geq4$, Brezis and Nirenberg showed that the following properties are equivalent
\begin{enumerate}[label=\upshape(\arabic*)]
    \item There is $x\in\Omega$ such that ${a}(x)<0$.
    \item $S(a)<S$.
    \item $S(a)$ is achieved by some function $u_{a}$.
\end{enumerate}  
When $N=3$, the situation is more subtle. For $a\equiv-\lambda$ with $\lambda$ being a positive constant, Brezis and Nirenberg proved that there exists a constant $\lambda^*\in(0,\lambda_{1})$ such that
\begin{equation*}
    S(-\lambda)=S\text{~~if~~}\lambda\in(0,\lambda^*],\quad  S(-\lambda)<S\text{~~if~~}\lambda\in(\lambda^*,\lambda_{1}),
\end{equation*}
and $S(-\lambda)$ is not achieved for $\lambda\in(0,\lambda^*)$. In the case of a ball, they further established that the threshold is $\lambda^{*}=\frac{1}{4}\lambda_{1}$, and that even at the endpoint, $S(-\frac{1}{4}\lambda_{1})$ is also not achieved. Subsequently, Druet \cite{Druet2002EllipticEW} (see also Esposito \cite{Esposito}) extended the above results to the general case where $a$ is a function and $\Omega$ is a general domain, thereby positively answering two conjectures previously proposed by Brezis \cite{Brezis1986}. Their results show that the condition that the Robin function $\phi_{a}$ (see \eqref{defin-robin-function}) is positive somewhere in $\Omega$ plays the same role as the condition that $a$ is negative somewhere in $\Omega$ in the case $N\geq 4$. Therefore, unlike in higher dimensions, the existence of solutions to \eqref{brezis-nirenberg-problem} in three dimensions is global in nature, depending on the values of $a$ throughout $\Omega$ and the geometry of $\Omega$. Following the definition of Pucci and Serrin \cite{Pucci, Pucci-1}, $N=3$ is referred to as a {critical dimension} for problem \eqref{brezis-nirenberg-problem}. Very recently, Druet's results were extended to the fractional Laplace equation and the $p$-Laplace equation by De Nitti and K\"{o}nig \cite{DeNitti-Konig} and by Angeloni and Esposito \cite{Angeloni-Esposito} respectively. We also refer to \cite{Collion, Druet-indaian} for analogous results on the Riemannian manifolds.

In recent years, there has been considerable interest in the following Choquard equation
\begin{equation}\label{choquard-2}
    \begin{cases}
        -\Delta u-\lambda u=\displaystyle\left(\int_{\Omega}\frac{u^{2^*_{\alpha}}(y)}{|x-y|^\alpha}dy\right)u^{2^*_{\alpha}-1} &\mbox{in}\ \Omega,\\
  u=0
  \ \  &\mbox{on}\ \partial \Omega,
    \end{cases}
\end{equation}
where $N\geq3$, $\Omega\subset\R^{N}$ is a bounded domain with a $C^{2}$ boundary, $\lambda>0$ is a constant, and $2^*_{\alpha}:=\frac{2N-\alpha}{N-2}$ is the upper critical exponent in the sense of the Hardy-Littlewood-Sobolev inequality. The existence of solutions to \eqref{choquard-2} can be traced back to Gao and Yang \cite{Gao2016TheBT}. They proved that if $\lambda$ is not an eigenvalue of $-\Delta$ (and is sufficiently large for $N=3$), then \eqref{choquard-2} admits a nontrivial solution. Furthermore, by employing the reduction method, the authors of \cite{Yang2023ExistenceOC, chen2024blowingupsolutionschoquardtype, chen2025choquard} established the existence of single-bubble solutions for \eqref{choquard-2} in the limit $\lambda \to 0$ when $N \geq 4$, and $\lambda \to \lambda^*>0$ when $N=3$. Sign-changing solutions have also been recently obtained in \cite{LIU2025128726}. However, to the best of our knowledge, the characterization of positive least energy solutions arising from the minimization problem \eqref{minimizing-problem} in three dimensions remains open. This paper addresses this gap by focusing on the three-dimensional case \eqref{main choquard equation with varepsilon}. Before stating our main results, we first recall the definitions of the Green's function and the Robin function.


Since the operator $-\Delta+a$ is assumed to be coercive, it has a Green's function $G_a$ satisfying, in the sense of distributions, for each fixed $y\in\Omega$,
\begin{equation*} 
\left\{
\begin{array}{l@{\quad}l}
-\Delta_x\, G_a(x,y) + a(x)\, G_a(x,y) = \, \delta_y & \quad \text{in} \ \Omega, \\
G_a(x,y) = 0  & \quad \text{on} \  \partial\Omega,
\end{array}
\right.  
\end{equation*}
where $\delta_{y}$ is the Dirac measure in $y$. Note that $G_{a}(x,y)$ is positive for every $x\neq y\in\Omega$ and is symmetric with respect to the two variables. The regular part $H_a$ of $G_a$ is defined by 
\begin{equation*} 
H_a(x,y) :=G_a(x,y)-\frac{1}{4\pi|x-y|}.
\end{equation*}
Then $H_a$ is a distributional solution for
\begin{equation} \label{Ha-pde}
\left\{
\begin{array}{l@{\quad}l}
-\Delta_x\, H_a(x,y) + a(x)\, G_a(x,y)= 0 &\text{in}  \ \Omega, \\
H_a(x,y) =-\frac{1}{4\pi|x-y|}  &\text{on}  \ \partial\Omega.
\end{array}
\right.  
\end{equation}
It is well-known that for each $y\in\Omega$, the function $H_a(\cdot,y)$, which is defined in $\Omega\setminus\{y\}$, admits a continuous extension to $\Omega$. Hence, one can define a function on $\Omega$ by taking its values on the diagonal, which is referred to as the Robin function
\begin{equation}\label{defin-robin-function}
    \phi_a(x) := H_a(x,x),\quad x\in\Omega.
\end{equation}
Further properties of the Green's function are discussed in Section \ref{section-2}.

The first main result of this paper is the following. 

\begin{Thm}\label{thm-00}
Assume $\alpha\in(0,3)$ is sufficiently small. The following properties are equivalent
\begin{enumerate}[label=\upshape(\arabic*)]
    \item There is $x\in\Omega$ such that $\phi_{a}(x)>0$.
    \item $S_{HL}(a)<S_{HL}$.
    \item $S_{HL}(a)$ is achieved by some function $u_{a}$.
\end{enumerate}  
\end{Thm}

Theorem \ref{thm-00} motivates the following definition, which is in the spirit of the work of Hebey and Vaugon \cite{Hebey2001FromBC}.

\begin{Def}
We say that a function $a$ is critical if $S_{HL}(a)=S_{HL}$ and $S_{HL}(\tilde{a})<S_{HL}$ for every $\tilde{a}$ satisfying $\tilde{a}(x) \leq a(x)$ for all $x \in \Omega$ with $\tilde{a} \not\equiv a$.
\end{Def}

Define the zero set of $\phi_a$ by
\begin{equation*}
    \begin{aligned}
        \mathcal{N}_{a}:=\{x\in\Omega:\phi_{a}(x)=0\}.
    \end{aligned}
\end{equation*}

\begin{Thm}\label{thm-a}
The condition that $a$ is critical is equivalent to
\begin{equation*}
    \max_{\Omega}\phi_{a}(x)=0.
\end{equation*}
Consequently, if $a$ is critical, then $\mathcal{N}_{a} \neq \emptyset$ and every point in $\mathcal{N}_{a}$ is a critical point of $\phi_{a}$.
\end{Thm}

\begin{Rem}
\begin{enumerate}
\item The assumption that $\alpha>0$ is sufficiently small is only used to prove the implication $(3) \Rightarrow (2)$; see Proposition \ref{assumption on a}.
\item For the unit ball $\Omega = B_1(0)$, it is known from \cite{Brezis1986} that the constant function $a = -\frac{\pi^{2}}{4}$ is critical, with the corresponding zero set $\mathcal{N}_a = \{0\}$ and the Green's function given by $G_a(0, y) = \frac{1}{|y|} \cos\left(\frac{\pi |y|}{2}\right)$.
\end{enumerate}
\end{Rem}

Note that when $a$ is a critical function, $S_{HL}(a)=S_{HL}$ is not achieved. It is therefore natural to ask whether $S_{HL}(a+\varepsilon V)$ becomes achievable under such a perturbation, and if so, what the asymptotic behavior of the corresponding minimizers is as $\varepsilon \to 0^{+}$. Research in this area dates back at least to \cite{Atkinson-jde,Atkinson-2,Brezispletier1989,Budd} on the Brezis-Nirenberg problem \eqref{brezis-nirenberg-problem} and near-critical problems \footnote{The {near-critical problem} refers to the situation where the critical exponent $2^*-1$ in \eqref{brezis-nirenberg-problem} is replaced by $2^*-1-\varepsilon$, with $\varepsilon > 0$ being sufficiently small.}, and has seen a surge of interest in recent years. To be more precise, when $\Omega$ is the unit ball in $\R^{3}$, Brezis and Peletier \cite{Brezispletier1989} established the asymptotics of solutions to \eqref{brezis-nirenberg-problem} with $a=-(\frac{\pi^{2}}{4}+\varepsilon)$ as $\varepsilon\to 0^{+}$. They also proposed three conjectures on the asymptotic behavior of the energy-minimizing solutions for both the Brezis-Nirenberg problem and the near-critical problem on general domains in $\R^{N}$ with $N\geq3$. Subsequently, the first two of these conjectures—concerning the asymptotics in general domains for $N \geq 4$—were proved independently by Han \cite{Han1991} and Rey \cite{Rey1989ProofOT} (see also \cite{JunchengWei1998, Takahashi-2004,Frank2019Energyhigher-dimensional} and references therein for further related results). More recently, Frank, K\"{o}nig, and Kova\v{r}\'{i}k \cite{Frank2021BlowupOS} proved the third conjecture, which concerns the asymptotics of the near-critical problem on general domains in $\R^{3}$ as $\varepsilon \to 0^{+}$. Additionally, in \cite{Frank2019Energythree-dimensional, Frank2021BlowupOS}, the same authors established the refined energy asymptotics and asymptotic profiles of the energy-minimizing solutions to \eqref{brezis-nirenberg-problem} under the perturbation $a \mapsto a+ \varepsilon V$ for a critical function $a$ and a perturbation $V\in L^{\infty}(\Omega)$. Furthermore, the asymptotics for multi-bubble solutions, which may blow-up and concentrate at several distinct points, have been investigated very recently by Cao, Luo, and Peng \cite{Cao2021Trans} and by K\"{o}nig and Laurain \cite{könig2023multibubbleblowupanalysisbrezisnirenberg, Konig2022FineMA}.

On the other hand, for the upper critical Choquard equation \eqref{choquard-2}, the asymptotic behavior of energy-minimizing solutions as $\lambda \to 0^{+}$ has been studied recently in dimensions $N \geq 5$ by Yang and Zhao \cite{Yang2023BlowUpBO} and by Pan, Wen, and Yang \cite{Pan-2025}. In contrast, the cases $N=3$ and $N=4$ remain open. This paper addresses the three-dimensional case. In the final part, we study the asymptotics of $S_{HL}(a+\varepsilon V)$ for a critical function $a$ and a perturbation $V \in L^{\infty}(\Omega)$, as well as the behavior of the corresponding minimizers.

In what follows, we work under the following assumption
\begin{Assum}\label{assumption-a}
    a is a critical function, and $a(x)<0$ for all $x\in\mathcal{N}_{a}$.
\end{Assum}

We first define 
\begin{equation*}
    Q_{V}(x):=\int_{\Omega}V(y)G_{a}(x,y)^{2}dy,~~\forall x\in\Omega\text{~~and~~} \mathcal{N}_{a}(V):=\{x\in\mathcal{N}_{a}:Q_{V}(x)<0\}.
\end{equation*}

Under the condition $\mathcal{N}_{a}(V) \neq \emptyset$, the asymptotic behavior of the perturbed minimal energy $S_{HL}(a+\varepsilon V)$ is given as follows.

\begin{Thm}\label{thm-3}
If $\mathcal{N}_{a}(V)\neq\emptyset$, then $S_{HL}(a+\varepsilon V)<S_{HL}$ for any $\varepsilon>0$ and
     \begin{equation}\label{eq-energy-asym}
        \lim_{\varepsilon\to0^{+}}\frac{S_{HL}-S_{HL}(a+\varepsilon V)}{\varepsilon^{2}}=\frac{128}{3}S_{HL}\sup_{\xi\in\mathcal{N}_{a}(V)}\frac{Q_{V}(\xi)^{2}}{|a(\xi)|}.
    \end{equation}
\end{Thm}

If $\mathcal{N}_{a}(V)\neq\emptyset$, then Theorems \ref{thm-00} and \ref{thm-3} yield that $S_{HL}(a+\varepsilon V)$ is achieved by a minimizer, denoted $u_{\varepsilon}$. After a suitable scaling, $u_{\varepsilon}$ satisfies

\begin{equation}\label{main choquard equation with varepsilon and V}
    \begin{cases}
        -\Delta u+(a+\varepsilon V)u=\displaystyle\left(\int_{\Omega}\frac{u^{6-\alpha}(y)}{|x-y|^\alpha}dy\right)u^{5-\alpha}&\mbox{in}\ \Omega,\\
  u>0
  \ \  &\mbox{in}\ \Omega,\\
  u=0
  \ \  &\mbox{on}\ \partial \Omega.
    \end{cases}
\end{equation}
We now turn to studying the asymptotic profiles of the sequence of minimizers $\{u_{\varepsilon}\}$. Before stating our results, we need to introduce some more notation. Let $P:H^{1}(\Omega)\to H^{1}_{0}(\Omega)$ be the orthogonal projection defined for every $\varphi\in H^{1}(\Omega)$ by
\begin{equation*}
    \int_{\Omega}\nabla P\varphi\cdot\nabla\psi dx= \int_{\Omega}\nabla \varphi\cdot\nabla\psi dx,\quad\forall\psi\in H^{1}_{0}(\Omega).
\end{equation*}
Then function $P\bar{U}_{\xi,\lambda}$ satisfies the following equation
\begin{equation}\label{eq-pU}
\begin{cases}
    -\Delta P\bar{U}_{\xi,\lambda}= -\Delta \bar{U}_{\xi,\lambda}&\text{~in~}\Omega,\\
    \quad\ \ P\bar{U}_{\xi,\lambda}=0&\text{~on~}\partial\Omega.
\end{cases}
\end{equation}
Moreover, we define the space
\begin{equation*}
    T_{\xi,\lambda}:=\text{~Span~}\left\{P\bar{U}_{\xi,\lambda},\frac{\partial P\bar{U}_{\xi,\lambda}}{\partial\lambda},\frac{\partial P\bar{U}_{\xi,\lambda}}{\partial {\xi}_{i}}, i=1,2,3.\right\}\subset H_{0}^{1}(\Omega)
\end{equation*}
and denote by $T_{\xi,\lambda}^{\bot}$ its orthogonal complement in $H_{0}^{1}(\Omega)$ with respect to the product $(u,v):=\int_{\Omega}\nabla u\cdot \nabla v$. We also denote by $\Pi_{\xi,\lambda}$ and $\Pi_{\xi,\lambda}^{\bot}$ the projections onto $T_{\xi,\lambda}$ and $T_{\xi,\lambda}^{\bot}$, respectively.

\begin{Thm}\label{thm-6}
Assume that $\mathcal{N}_{a}(V)\neq\emptyset$. Let $u_{\varepsilon}$ be a minimizer for $S_{HL}(a+\varepsilon V)$ solving \eqref{main choquard equation with varepsilon and V}. There are sequences $\{\mu_{\varepsilon}\}\subset \R^{+}$, $\{\xi_{\varepsilon}\}\subset\Omega$, $\{\lambda_{\varepsilon}\}\subset \R^{+}$ and $\{r_{\varepsilon}\}\subset T_{\xi,\lambda}^{\bot}$ such that
\begin{equation*}
    \begin{aligned}
        u_{\varepsilon}=\mu_{\varepsilon}\left(P\bar{U}_{\xi_{\varepsilon},\lambda_{\varepsilon}}+4\pi\bar{C}_{\alpha}\lambda_{\varepsilon}^{-\frac{1}{2}}\Pi^{\bot}_{\xi_{\varepsilon},\lambda_{\varepsilon}}(H_{a}(\xi_{\varepsilon},\cdot)-H_{0}(\xi_{\varepsilon},\cdot))+r_{\varepsilon}\right).
    \end{aligned}
\end{equation*}
Moreover, as $\varepsilon\to0$, we have
\begin{equation*}
    \xi_{\varepsilon}\to\xi_{0}\text{~for some~}\xi_{0}\in\mathcal{N}_{a}(V) \text{~such that~}  \frac{Q_{V}(\xi_{0})^{2}}{|a(\xi_{0})|}=\sup_{\xi\in\mathcal{N}_{a}(V)}\frac{Q_{V}(\xi)^{2}}{|a(\xi)|},
\end{equation*}
\begin{equation*}
    \mu_{\varepsilon}=1+\frac{256}{3}\phi_{0}(\xi_{0})\frac{|Q_{V}(\xi_{0})|}{|a(\xi_{0})|}\varepsilon+o(\varepsilon),
\end{equation*}
\begin{equation*}
   \varepsilon\lambda_{\varepsilon}\to\frac{|a(\xi_{0})|}{4|Q_{V}(\xi_{0})|},
\end{equation*}
and
\begin{equation*}
    \phi_{a}(\xi_{\varepsilon})=o(\varepsilon),\quad \|\nabla r_{\varepsilon}\|_{L^{2}(\Omega)}=o(\varepsilon).
\end{equation*}
\end{Thm}

Finally, in the degenerate case $\mathcal{N}_{a}(V)=\emptyset$, we have the following result.
\begin{Thm}\label{thm-4}
If $\mathcal{N}_{a}(V)=\emptyset$, then one of the following holds: 
    \begin{enumerate}[label=\upshape(\arabic*)]
        \item  $S_{HL}(a+\varepsilon V)<S_{HL}$ for all sufficiently small $\varepsilon>0$, and moreover, $S_{HL}(a+\varepsilon V)=S_{HL}+o(\varepsilon^{2})$ as $\varepsilon\to0^{+}$.
        \item  $S(a+\varepsilon V)=S_{HL}$ for all sufficiently small $\varepsilon>0$.
    \end{enumerate}
     If, in addition, $Q_{V}(x)>0$ for all $x\in \mathcal{N}_{a}$, then case $(1)$ cannot occur.
\end{Thm}

\begin{Rem}
    \begin{enumerate}
        \item Note that if $a$ is a constant and is critical, then $a$ must be negative. Moreover, for a general critical function $a$, Corollary \ref{corollary-phi-a} yields $a \leq 0$ in $\mathcal{N}_{a}$. Thus, Assumption~\ref{assumption-a} is not very restrictive.
        \item The situation in dimension three differs from the higher-dimensional case. In higher dimensions, it suffices to consider an initial expansion of the form $u_{\varepsilon} = \mu_{\varepsilon}(P\bar{U}_{\xi_{\varepsilon},\lambda_{\varepsilon}} + w_{\varepsilon})$, with $w_{\varepsilon} \in T_{\xi_{\varepsilon},\lambda_{\varepsilon}}^{\bot}$, which can be obtained via a concentration-compactness argument, see \cite{Yang2023BlowUpBO}. However, in dimension three, as shown in \cite{Brezispletier1989,Frank2019Energythree-dimensional, Frank2021BlowupOS}, a finer analysis is needed. Specifically, to achieve a better approximation for $u_{\varepsilon}$, we need to further decompose the remainder term $w_{\varepsilon}$ as
\begin{equation}
    w_{\varepsilon} = 4\pi\bar{C}_{\alpha}\lambda_{\varepsilon}^{-\frac{1}{2}}\Pi^{\bot}_{\xi_{\varepsilon},\lambda_{\varepsilon}}\big(H_{a}(\xi_{\varepsilon},\cdot)-H_{0}(\xi_{\varepsilon},\cdot)\big) + r_{\varepsilon}.
\end{equation}
This finer decomposition, however, leads to greater technical challenges. Moreover, the nonlocal term introduces further difficulties, which require us to develop new estimates. Here, the symmetry of the double integrals and the application of both the Hardy-Littlewood-Sobolev and the reversed Hardy-Littlewood-Sobolev inequalities are crucial.
\end{enumerate}
\end{Rem}

The paper is organized as follows. Section~\ref{section-2} reviews preliminary results. The proof of Theorem~\ref{thm-00} is presented in Section~\ref{section-3}. In Section~\ref{section-4}, we first establish a sharper upper bound for $S_{HL}(a+\varepsilon V)$, leading to Theorem~\ref{thm-a}. Furthermore, by employing a refined expansion of the minimizers $u_{\varepsilon}$, we derive a lower bound for $S_{HL}(a+\varepsilon V)$, from which Theorems~\ref{thm-3}--\ref{thm-4} follow.

\medskip

\noindent\textbf{Notations.}
Throughout this paper, we adopt the following notations.
\begin{enumerate}
    \item The homogeneous Sobolev space $\dot{H}^{1}(\R^{N})$ is defined as $$\dot{H}^{1}(\R^{N}):=\left\{u\in L^{\frac{2N}{N-2}}(\R^{N}):~\nabla u\in L^{2}(\R^{N})\right\}.$$ For a domain $\Omega$, the norm in $H^{1}_{0}(\Omega)$ is given by $\|u\|_{H^{1}_{0}(\Omega)}:=\left(\int_{\Omega}|\nabla u|^{2}dx\right)^{1/2}$.  
    \item We use $C$ to denote various positive constants whose value may change from line to line. The notation $C_{1}=o(\varepsilon)$ means that $C_{1}/\varepsilon\to0$ as $\varepsilon\to0$, and $C_{2}=O(\varepsilon)$ means that $|C_{2}/\varepsilon|\leq C $ for some constant $C>0$ as $\varepsilon\to0$.
     \item  Let $f,g: X\to \R^{+}$ be two nonnegative functions defined on some set $X$. We write $f\lesssim g$ (or equivalently, $g\gtrsim f$) if there exists a constant $C>0$, independent of $x\in X$, such that $f(x)\leq C g(x)$ for all $x\in X$. Furthermore, we write $f\sim g$ if both $f\lesssim g$ and $g\lesssim f$ hold.
\end{enumerate}

\section{Preliminaries}\label{section-2}
In this section, we give some preliminaries. First, we recall the following critical Lane-Emden-Fowler equation
\begin{equation}\label{Emden-Fowler equation}
   -\Delta u = 3u^5
\quad\text{in}\ \R^3. 
\end{equation}
It is well-known \cite{AUBIN,Talenti} that all positive solutions to \eqref{Emden-Fowler equation} are precisely the Aubin-Talenti functions $\{U_{\xi,\lambda}\}$ with parameters $\xi \in \R^3$ and $\lambda > 0$. This equation arises as the Euler-Lagrange equation of the minimization problem 
\begin{equation}\label{optimization problem for s}
    S=\inf_{u\in\dot{H}^{1}(\R^{3}), \|u\|_{L^{6}(\R^3)}=1}\int_{\R^3} |\nabla u|^2dx,
\end{equation}
which is related to the Sobolev inequality
\begin{equation*}
  \int_{\R^3} |\nabla u|^2 \geq S \left( \int_{\R^3} u^6 \right)^{1/3},\quad\forall u\in\dot{H}^{1}(\R^{3}),  
\end{equation*}
where $S$ is the sharp Sobolev constant.

For the nonlocal problem with convolution, the Hardy-Littlewood-Sobolev (HLS for short) inequality (see \cite{Lieb2001}) and the reversed Hardy-Littlewood-Sobolev (RHLS for short) inequality (see \cite{Dou-2015-imrn,Beckner,Nguyen-2017-Israel}) play an important role.
\begin{Thm}\label{HL}
Let $1<\theta,r<\infty$ and $0<\alpha<3$ with $\frac{1}{\theta}+\frac{1}{r}=2-\frac{\alpha}{3}$. If $f\in L^\theta(\mathbb{R}^3)$ and $g\in L^r(\mathbb{R}^3)$, then there exists a constant $C_{\theta,r,\alpha}>0$, independent of $f,g$, such that
\begin{equation}\label{HLS}
  \left|\displaystyle\int_{\mathbb{R}^3}\int_{\mathbb{R}^3}\frac{f(x)g(y)}{|x-y|^\alpha}dxdy\right|\leq C_{\theta,r,\alpha}\|f\|_{L^\theta(\mathbb{R}^3)}\|g\|_{L^r(\mathbb{R}^3)}.
\end{equation}
If $\theta=r=\frac{6}{6-\alpha}$, then 
\begin{equation}\label{definition of C alpha}
    C_{\theta,r,\alpha}=C_{\alpha}=\pi^{\frac{\alpha}{2}}\frac{\Gamma(\frac{3-\alpha}{2})}{\Gamma(3-\frac{\alpha}{2})}\left(\frac{\Gamma(3)}{\Gamma(\frac{3}{2})}\right)^{\frac{3-\alpha}{3}}.
\end{equation}
In this case, the equality in \eqref{HLS} holds if and only if $f\equiv(const.)g$, where 
\begin{equation*}
  g(x)=A(a^2+|x-b|^2)^{-\frac{6-\alpha}{2}},\text{~for some~}A\in \mathbb{C},0\neq a\in \mathbb{R}\text{~and~}b\in \mathbb{R}^3.
\end{equation*}
\end{Thm}

\begin{Thm}\label{RHL}
Let $0<\theta,r<1$ and $\alpha >0$ with $\frac{1}{\theta}+\frac{1}{r}=2+\frac{\alpha}{3}$. If $f\in L^\theta(\mathbb{R}^3)$ and $g\in L^r(\mathbb{R}^3)$, then there exists a constant $\tilde{C}_{\theta,r,\alpha}>0$ independent of $f,g$, such that
\begin{equation}\label{RHLS}
  \displaystyle\int_{\mathbb{R}^3}\int_{\mathbb{R}^3}|f(x)g(y)||x-y|^{\alpha}dxdy\geq\tilde{C}_{\theta,r,\alpha}\|f\|_{L^\theta(\mathbb{R}^3)}\|g\|_{L^r(\mathbb{R}^3)}.
\end{equation}
If $\theta=r=\frac{6}{6+\alpha}$, then 
\begin{equation}\label{Rdefinition of C alpha}
    \tilde{C}_{\theta,r,\alpha}=\tilde{C}_{\alpha}=\pi^{-\frac{\alpha}{2}}\frac{\Gamma(\frac{3+\alpha}{2})}{\Gamma(3+\frac{\alpha}{2})}\left(\frac{\Gamma(3)}{\Gamma(\frac{3}{2})}\right)^{\frac{3+\alpha}{3}}.
\end{equation}
In this case, the equality in \eqref{RHLS} holds if and only if $f\equiv(const.)g$, where 
\begin{equation*}
  g(x)=A(a^2+|x-b|^2)^{-\frac{6+\alpha}{2}},\text{~for some~}A\in \mathbb{C},0\neq a\in \mathbb{R}\text{~and~}b\in \mathbb{R}^3.
\end{equation*}
\end{Thm}

\begin{Rem}\label{remark of HLS}
\begin{enumerate}
    \item By the HLS inequality and the Sobolev inequality, we have, for any $u\in \dot{H}^{1}(\R^{3})$,
\begin{equation*}
\begin{split}
   \left(\displaystyle{\int_{\R^3}}\displaystyle{\int_{\R^3}}\frac{|u(x)|^{6-\alpha}|u(y)|^{6-\alpha}}{|x-y|^{\alpha}}dxdy\right)^{\frac{1}{6-\alpha}}&\leq C_{\alpha}^{\frac{1}{6-\alpha}}\left(\displaystyle{\int_{\R^3}}|u(x)|^{6}dx\right)^{\frac{1}{3}}\\
   &\leq C_{\alpha}^{\frac{1}{6-\alpha}} S^{-1}\int_{\R^{3}}|\nabla u(x)|^{2}dx.
\end{split} 
\end{equation*}
    \item From the HLS inequality, the  functional
$$u\mapsto\displaystyle{\int_{\R^3}}\displaystyle{\int_{\R^3}}\frac{|u(x)|^{q}|u(y)|^{q}}{|x-y|^{\alpha}}dxdy$$
is well-defined on $H^{1}(\R^3)\times H^{1}(\R^3)$ provided $\frac{6-\alpha}{3}\leq q\leq6-\alpha$. Hence, it is natural to call $2_{\alpha}:=\frac{6-\alpha}{3}$ the lower Hardy-Littlewood-Sobolev critical exponent and $2^*_{\alpha}:=6-\alpha$ the upper Hardy-Littlewood-Sobolev critical exponent. 
\end{enumerate} 
\end{Rem}

We recall the following nondegeneracy property for the upper critical Choquard equation, which was established by Li et al. \cite[Theorem 1.5]{lixuemei2023nondegeneracy}.
\begin{Thm}\label{non-degeneracy property}
If $v\in \dot{H}^{1}(\mathbb{R}^3)$ is a solution of the following equation
\begin{equation*}
-\Delta v-(6-\alpha)\left(\int_{\mathbb{R}^3}\frac{\bar{U}_{\xi,\lambda}^{5-\alpha}(y)v(y)}{|x-y|^\alpha}dy\right)\bar{U}_{\xi,\lambda}^{5-\alpha}-
 (5-\alpha)\left(\int_{\mathbb{R}^3}\frac{\bar{U}_{\xi,\lambda}^{6-\alpha}(y)}{|x-y|^\alpha}dy\right)\bar{U}_{\xi,\lambda}^{4-\alpha}v=0,
\end{equation*}
then there exist constants $a_{0},a_{i}\in \R$, $i=1,2,3$ such that
\begin{equation*}
  v= a_{0}\frac{\partial\bar{U}_{\xi,\lambda}}{\partial\lambda}+\sum_{i=1}^{3}a_{i}\frac{\partial\bar{U}_{\xi,\lambda}}{\partial \xi_{i}}.
\end{equation*}
\end{Thm}

\begin{Lem} 
Let $u$ be a solution of \eqref{main choquard equation with varepsilon} and $x_{0}\in\Omega$. Then
    \begin{equation}\label{pohozaev-identiy}
        -\frac{1}{2}\int_{\partial\Omega}(x-x_{0},\nu)|\nabla u|^{2}d\sigma=\int_{\Omega}\left(a+\frac{(x,\nabla a)}{2}\right)u^{2}dx,
    \end{equation}
where $\nu=\nu(x)$ denotes the unit outward normal to the boundary $\partial\Omega$.
\end{Lem}

\begin{proof}
Without loss of generality, we may suppose that $x_{0}=0$. Multiplying both sides of \eqref{main choquard equation with varepsilon} by $(x,\nabla u)$ and integrating on $\Omega$, we obtain
\begin{equation*}
    -\int_{\Omega}\Delta u(x,\nabla u)dx+\int_{\Omega}au(x,\nabla u)dx=\int_{\Omega}(x,\nabla u)\left(\int_{\Omega}\frac{u^{6-\alpha}(y)}{|x-y|^{\alpha}}dy\right)u^{5-\alpha}dx.
\end{equation*}
By the divergence theorem, we see that
\begin{equation*}
\begin{aligned}
  -\int_{\Omega}\Delta u(x,\nabla u)dx&=-\frac{1}{2}\int_{\partial\Omega}(x,\nu)|\nabla u|^{2}d\sigma-\frac{1}{2}\int_{\Omega}|\nabla u|^{2}  \\
  &=-\frac{1}{2}\int_{\partial\Omega}(x,\nu)|\nabla u|^{2}d\sigma-\frac{1}{2}\int_{\Omega}\int_{\Omega}\frac{u^{6-\alpha}(y)u^{6-\alpha}(x)}{|x-y|^{\alpha}}dxdy+\frac{1}{2}\int_{\Omega}au^{2}dx
\end{aligned}   
\end{equation*}
and
\begin{equation*}
    \begin{aligned}
        \int_{\Omega}au(x,\nabla u)dx=-\frac{3}{2}\int_{\Omega}au^{2}dx-\frac{1}{2}\int_{\Omega}(x,\nabla a)u^{2}dx.
    \end{aligned}
\end{equation*}
On the other hand, similar to the estimate (2.5) in \cite{Squassina}, we obtain
\begin{equation*}
   \int_{\Omega}(x,\nabla u)\left(\int_{\Omega}\frac{u^{6-\alpha}(y)}{|x-y|^{\alpha}}dy\right)u^{5-\alpha}dx=-\frac{1}{2} \int_{\Omega}\int_{\Omega}\frac{u^{6-\alpha}(y)u^{6-\alpha}(x)}{|x-y|^{\alpha}}dxdy.
\end{equation*}
Combining the above estimates, we conclude that \eqref{pohozaev-identiy} holds.
\end{proof}

We now present further properties of the Green's function, as given in the appendix of \cite{Rey1990TheRO,Frank2021BlowupOS}.
\begin{Lem}\label{lem-H-0}
Let $x\in\Omega$ and $d_{x}:=\text{dist}(x,\partial\Omega)$. We have
\begin{equation*}
    |H_{0}(x,y)|\lesssim d_{x}^{-1},\quad |\nabla_{y}H_{0}(x,y)|\lesssim d_{x}^{-2},\quad\forall y\in\Omega.
\end{equation*}
Moreover, as $d_{x}\to 0$, 
\begin{equation*}
\phi_{0}(x)=-\frac{1}{8\pi d_{x}}(1+O(d_{x}))    
\end{equation*}
and 
\begin{equation*}
    \nabla \phi_{0}(x)=-\frac{1}{8\pi d_{x}^{2}}\frac{x^{\prime}-x}{d_{x}}+O(d_{x}^{-1}),
    \end{equation*}
where $x'\in\partial\Omega$ is the unique point satisfying $d(x,\partial\Omega)=|x-x'|$.
\end{Lem}

\begin{Lem}\label{lem-H-a}
Let $x\in\Omega$ and $d_{x}:=\text{dist}(x,\partial\Omega)$. We have
\begin{equation*}
    |H_{a}(x,y)|\lesssim d_{x}^{-1},\quad\forall y\in\Omega, 
\end{equation*}
and, for $0<\mu<1$
    \begin{equation*}
        H_{a}(x,y)=\phi_{a}(x)+\frac{1}{2}\nabla\phi_{a}\cdot(y-x)+\frac{1}{8\pi}a(x)|y-x|+O(|y-x|^{1+\mu})\text{~as~}y\to x.
    \end{equation*}
Moreover, if $\tilde{a}\geq a$ and $\tilde{a}\not\equiv a$, then
\begin{equation*}
    \phi_{a}(x)>\phi_{\tilde{a}}(x),\text{~for any~}x\in\Omega.
\end{equation*}    
\end{Lem}


\begin{Lem}\label{lemma PU}
Let $\xi\in\Omega$, $d_{\xi}:=\text{dist}(\xi,\partial\Omega)$ and $\varphi_{\xi,\lambda}:=U_{\xi,\lambda}-PU_{\xi,\lambda}$. We have
    \begin{equation*}
        \quad\| \varphi_{\xi,\lambda}\|_{L^{6}(\Omega)} \lesssim \lambda^{-\frac{1}{2}} d_{\xi}^{-\frac{1}{2}},\quad \|\partial_{\lambda}\varphi_{\xi,\lambda}\|_{L^{6}(\Omega)} \lesssim \lambda^{-\frac{3}{2}} d_{\xi}^{-\frac{1}{2}},\quad \|\partial_{\xi_{i}}\varphi_{\xi,\lambda}\|_{L^{6}(\Omega)} \lesssim \lambda^{-\frac{1}{2}} d_{\xi}^{-\frac{1}{2}},
    \end{equation*}
    and
    \begin{equation*}
        \| \varphi_{\xi,\lambda}\|_{L^{\infty}(\Omega)}\lesssim \lambda^{-\frac{1}{2}}d_{\xi}^{-1},\quad  \|\partial_{\lambda} \varphi_{\xi,\lambda}\|_{L^{\infty}(\Omega)}\lesssim \lambda^{-\frac{3}{2}}d_{\xi}^{-1},\quad  \|\partial_{\xi_{i}} \varphi_{\xi,\lambda}\|_{L^{\infty}(\Omega)}\lesssim \lambda^{-\frac{1}{2}}d_{\xi}^{-2}.
    \end{equation*}
    Moreover,  $0\leq \varphi_{\xi,\lambda}\leq U_{\xi,\lambda}$ and
    \begin{equation*}
        \varphi_{\xi,\lambda}=-4\pi\lambda^{-\frac{1}{2}}H_{0}(\xi,\cdot)+f_{\xi,\lambda},
    \end{equation*}
    where
    \begin{equation*}
        \|f_{\xi,\lambda}\|_{L^{\infty}(\Omega)} \lesssim \lambda^{-\frac{5}{2}} d_{\xi}^{-3},\quad\|\partial_{\lambda}f_{\xi,\lambda}\|_{L^{\infty}(\Omega)} \lesssim \lambda^{-\frac{7}{2}} d_{\xi}^{-3},\quad\|\partial_{\xi_{i}}f_{\xi,\lambda}\|_{L^{\infty}(\Omega)} \lesssim \lambda^{-\frac{5}{2}} d_{\xi}^{-4}.
    \end{equation*}
\end{Lem}

\begin{Lem}\label{lemma-g}
We define the function
\begin{equation*}
    g_{\xi,\lambda}(x):=\frac{\lambda^{-\frac{1}{2}}}{|\xi-x|}-U_{\xi,\lambda}(x).
\end{equation*}
As $\lambda\to\infty$, we have
\begin{equation*}
    \|g_{\xi,\lambda}\|_{L^{p}(\R^{3})}\lesssim\lambda^{\frac{1}{2}-\frac{3}{p}},\quad \|\partial_{\lambda}g_{\xi,\lambda}\|_{L^{p}(\R^{3})}\lesssim\lambda^{-\frac{1}{2}-\frac{3}{p}}
\end{equation*}
for all $1\leq p<3$. Moreover, $\nabla g_{\xi,\lambda}\in L^{p}(\R^{3})$ for all $1\leq p<\frac{3}{2}$.
\end{Lem}

Finally, we recall some elementary inequalities from \cite{Iacopetti-2016-CCM}.
\begin{Lem}\label{lem inequality 1}
Let $\alpha$ be a positive real number. If $\alpha \leq 1$, there holds, for all $x, y >0$,
\begin{equation*}
    (x+y)^\alpha \leq x^\alpha + y^\alpha.
\end{equation*}
If $\alpha \geq 1$, we have, for all $x, y >0$,
\begin{equation*}
    (x+y)^\alpha \leq 2^{\alpha-1} (x^\alpha + y^\alpha).
\end{equation*}
 \end{Lem}

\begin{Lem}\label{lem inequality 2}
Let $q$ be a positive real number. There exists a positive constant $c$, depending only on $q$, such that for any $a, b\in \mathbb R$,
\begin{equation*}
||a+b|^q-|a|^q| \leq
\begin{cases}
c(q) \min\{|b|^q, |a|^{q-1}|b|\}, &\ \hbox{if}\ 0<q<1,\\
c(q) (|a|^{q-1}|b|+|b|^q), & \ \hbox{if}\ q\geq1.
\end{cases}
\end{equation*}
Moreover, if $q>2$ then
\begin{equation*}
\left||a+b|^q-|a|^q-q |a|^{q-2}ab\right|\leq c(q)\left(|a|^{q-2}|b|^2+|b|^q\right).
\end{equation*}
\end{Lem}


\section{Proof of Theorem \ref{thm-00}}\label{section-3}

\begin{Prop}\label{prop of SHL and SHLa}
It holds that $0<S_{HL}(a)\leq S_{HL}(0)=S_{HL}.$
\end{Prop}

\begin{proof}
First, by \cite[ Lemma 1.3]{Gao2016TheBT}, we know that $S_{HL}(0)=S_{HL}$ and $S_{HL}(0)$ is never achieved unless $\Omega=\R^{3}$. Furthermore, it follows from \eqref{optimization problem for shl} and \eqref{optimization problem for s} that
$U(x)=\frac{3^{\frac{1}{4}}}{(1+|x|^{2})^{\frac{1}{2}}}$ is a minimizer for both $S_{HL}$ and $S$.
Without loss of generality, we may assume that $0\in\Omega$ and $B_{2\delta}(0)\subset\Omega$ for some $\delta>0$. Let $\psi\in C_{c}^{\infty}(\Omega)$ be a cut-off function such that
\begin{equation*}
    \left\{\begin{array}{l}
\displaystyle \psi(x)=\left\{\begin{array}{l}
\displaystyle 1\qquad \mbox{if}\quad x\in B_{\delta}(0),\\
\displaystyle 0\qquad \mbox{if} \quad x\in \mathbb{R}^3 \setminus B_{2\delta}(0),\\
\end{array}
\right.\\
\displaystyle 0\leq\psi(x)\leq1,~~|\nabla\psi(x)|\leq C, \hspace{2.14mm} \forall x\in \mathbb{R}^3.\\
\end{array}
\right.
\end{equation*}
For $\varepsilon > 0$, we define
\begin{equation*}
    U_{\varepsilon}(x):=\varepsilon^{-\frac{1}{2}}U\left(\frac{x}{\varepsilon}\right)\text{~~and~~}
u_{\varepsilon}(x):=\psi(x)U_{\varepsilon}(x).\\
\end{equation*}
Using estimates (3.2) and (3.9) in \cite{Gao2016TheBT}, we have, as $\varepsilon\rightarrow0$,
\begin{equation}\label{energy estimate proof 1}
\int_{\Omega}|\nabla u_{\varepsilon}|^{2}dx=S^{\frac{3}{2}}+O(\varepsilon)
=C_{\alpha}^{\frac{1}{6-\alpha}\cdot\frac{3}{2}}S_{HL}^{\frac{3}{2}}+O(\varepsilon)
\end{equation}
and
\begin{equation}\label{energy estimate proof 2}
\begin{aligned}
    \left(\int_{\Omega}\int_{\Omega}\frac{u_{\varepsilon}(x)^{6-\alpha}u_{\varepsilon}(y)^{6-\alpha}}
{|x-y|^{\alpha}}dxdy\right)^{\frac{1}{6-\alpha}}\geq\left(C_{\alpha}^{\frac{3}{2}}S_{HL}^{\frac{6-\alpha}{2}}+O(\varepsilon^{\frac{6-\alpha}{2}})\right)^{\frac{1}{6-\alpha}}.
\end{aligned}
\end{equation}
On the other hand, a direct computation yields that
\begin{equation}\label{energy estimate proof 3}
\int_{\Omega}u_{\varepsilon}^{2}dx\leq \int_{B_{2\delta}(0)}U_{\varepsilon}^{2}dx=O(\varepsilon).
\end{equation}
Combining \eqref{energy estimate proof 1}--\eqref{energy estimate proof 3}, we obtain
\begin{equation*}
    \begin{aligned}
    S_{HL}(a)&\leq\frac{C_{\alpha}^{\frac{1}{6-\alpha}\cdot\frac{3}{2}}S_{HL}^{\frac{3}{2}}+O(\varepsilon)}
{\left(C_{\alpha}^{\frac{3}{2}}S_{HL}^{\frac{6-\alpha}{2}}+O(\varepsilon^{\frac{6-\alpha}{2}})\right)^{\frac{1}{6-\alpha}}
}
\to S_{HL},\quad \text{~as~}\varepsilon\to0.
    \end{aligned}
\end{equation*}
Since the operator $-\Delta + a$ is coercive, we have $S_{HL}(a) \geq 0$. We now claim that $S_{HL}(a) > 0$. Otherwise, there exists a sequence $\{u_n\} \subset H_0^1(\Omega)$ such that
\begin{equation*}
    \|u_n\|_{HL} = 1 \quad \text{and} \quad \int_\Omega |\nabla u_n|^2 + a u_n^2 \, dx = o(1).
\end{equation*}
By the coercivity of the operator $-\Delta + a$, it follows that $u_n \to 0$ in $H_0^1(\Omega)$. However, this contradicts $\|u_n\|_{HL} = 1$ due to the HLS inequality. Thus, $S_{HL}(a) > 0$ and the proof is complete.
\end{proof}

\begin{Prop}\label{coro for S-HL-a-delta}
There exists a unique constant $B(a)\in \R$ such that 
\begin{equation*}
S_{HL}(a+\varepsilon)<S_{HL}\text{~ for~} \varepsilon<B(a) \text{~and~} S_{HL}(a+\varepsilon)=S_{HL} \text{~for~}\varepsilon\geq B(a).    
\end{equation*}
\end{Prop}

\begin{proof}
From the definition of $S_{HL}(a+\varepsilon)$, we see that $S_{HL}(a+\varepsilon)$ is monotonically increasing in $\varepsilon$. Moreover, we claim that $S_{HL}(a+\varepsilon)$ is Lipschitz continuous with respect to $\varepsilon$. Indeed, for any $u \in H_0^1(\Omega)$, it follows from the H\"{o}lder inequality and the RHLS inequality that
\begin{equation*}
    \begin{aligned}
      \int_{\Omega}u^{2}(x)dx&\leq \left(\int_{\Omega}|u(x)|^{5-\alpha}dx\right)^{\frac{2}{5-\alpha}}|\Omega|^{\frac{3-\alpha}{5-\alpha}}\\
      &\lesssim \left(\int_{\Omega}\int_{\Omega}|u(x)|^{6-\alpha}|u(y)|^{6-\alpha}|x-y|^{\frac{6}{5-\alpha}}dxdy\right)^{\frac{1}{6-\alpha}}\\
      &\lesssim \left(\int_{\Omega}\int_{\Omega}\frac{|u(x)|^{6-\alpha}|u(y)|^{6-\alpha}}{|x-y|^{\alpha}}dxdy\right)^{\frac{1}{6-\alpha}}.
    \end{aligned}
\end{equation*}
This yields that
\begin{equation}\label{energy estimate proof 14}
    \begin{aligned}
        |S_{HL}(a+\varepsilon)-S_{HL}(a)|\leq C(\Omega,\alpha)|\varepsilon|.
    \end{aligned}
\end{equation}

Suppose $\varepsilon + \min_{x \in \bar{\Omega}} a(x) \geq 0$. Then by Proposition \ref{prop of SHL and SHLa}, we have $S_{HL}(a+\varepsilon) = S_{HL}(0) = S_{HL}$. Let $\lambda_{1}(a) > 0$ be the first eigenvalue of $-\Delta + a$. For $\varepsilon$ satisfying $-\lambda_{1}(a) < \varepsilon < -\lambda_{1}(a) + \delta$ with sufficiently small $\delta > 0$, estimate \eqref{energy estimate proof 14} gives $S_{HL}(a+\varepsilon)\leq C(\Omega,\alpha) \delta < S_{HL}$. This completes the proof.
\end{proof}

In particular, for $a$ being a negative constant, we have the following proposition.
\begin{Prop}\label{prop sharp soblev inequality}
There exists a unique constant $0<\lambda^*(\Omega)<\lambda_{1}$ such that for $\lambda\in\R$ we have
$S_{HL}(-\lambda)=S_{HL}\text{~ for every~}\lambda\leq \lambda^{*} \text{~and~}S_{HL}(-\lambda)<S_{HL}\text{~ for every~} \lambda^{*}<\lambda.$   
\end{Prop}

\begin{proof}
First, it follows from the HLS inequality that $\|u\|_{HL}^{2}\leq C_{\alpha}^{\frac{1}{6-\alpha}}\|u\|_{L^{6}(\Omega)}^{2}$. This, together with \cite[Corollary 1.1]{Brezis1983CPAM} and \eqref{definition of SHL}, implies that there exists a constant $\lambda_{0}(\Omega)$ with $0<\lambda_{0}<\lambda_{1}$ such that
\begin{equation*}
\begin{aligned}
   \|\nabla u\|_{L^{2}(\Omega)}^{2}&\geq S\|u\|_{L^{6}(\Omega)}^{2}+\lambda_{0}\|u\|_{L^{2}(\Omega)}^{2}\\
   &\geq S_{HL}\|u\|_{HL}^{2}+\lambda_{0}\|u\|_{L^{2}(\Omega)}^{2}\quad\text{~for all~}u\in H^{1}_{0}(\Omega).
\end{aligned}
\end{equation*}
The desired conclusion then follows from Proposition \ref{prop of SHL and SHLa} and Proposition \ref{coro for S-HL-a-delta}.
\end{proof}

\begin{Cor}
  If $\|a\|_{L^{\infty}(\Omega)}$ is sufficiently small, then $S_{HL}(a) = S_{HL}$. 
\end{Cor}

\begin{Thm}\label{thm-1}
The implications $(1)\Rightarrow (2)\Rightarrow(3)$ hold.
\end{Thm}

\begin{proof}
\textbf{Step 1.}
In this step, we prove $(1)\Rightarrow(2)$. Without loss of generality, we may assume that $0\in\Omega$ and $\phi_{a}(0)>0$. We consider the solutions $\phi_{\varepsilon}$ of the following equation
    \begin{equation}\label{equivlent theorem 1-1}
    \begin{cases}
        -\Delta \phi_{\varepsilon}+a\phi_{\varepsilon}=-\Delta U_{\varepsilon}&\quad\text{~in~}\Omega,\\
        \phi_{\varepsilon}=0&\quad\text{~on~}\partial\Omega,\\
    \end{cases}
    \end{equation}
where $U_{\varepsilon}(x)$ is defined by
\begin{equation*}
    U_{\varepsilon}=\frac{\varepsilon^{\frac{1}{2}}}{(\varepsilon^{2}+|x|^{2})^{1/2}}.
\end{equation*}  
We now claim that, as $\varepsilon\to0$,
\begin{equation}\label{equivlent theorem 1-3}
    J(\phi_{\varepsilon})=S_{HL}-C\phi_{a}(0)\varepsilon+o(\varepsilon),
\end{equation}
where $C$ is some positive constant and
\begin{equation*}
J(\phi_{\varepsilon})=\frac{\int_\Omega |\nabla \phi_{\varepsilon}|^2+ a \phi_{\varepsilon}^2dx}{\left(\int_{\Omega}\int_{\Omega}\frac{|\phi_{\varepsilon}(x)|^{6-\alpha}|\phi_{\varepsilon}(y)|^{6-\alpha}}{|x-y|^{\alpha}}dxdy\right)^{\frac{1}{6-\alpha}}}.
\end{equation*}
Therefore, the condition $\phi_{a}(0) > 0$ implies that $S_{HL}(a) < S_{HL}$, and the proof is complete. It remains to show that \eqref{equivlent theorem 1-3} holds. Let
\begin{equation*}
    h_{\varepsilon}=(\phi_{\varepsilon}-U_{\varepsilon})/\sqrt{\varepsilon}.
\end{equation*}
Then by \eqref{equivlent theorem 1-1}, we have
\begin{equation*}
    \begin{cases}
        -\Delta h_{\varepsilon}+a h_{\varepsilon}=-\frac{a(x)}{(\varepsilon^2+|x|^{2})^{1/2}}&\quad\text{~in~}\Omega,\\
        h_{\varepsilon}=-\frac{1}{(\varepsilon^2+|x|^{2})^{1/2}}&\quad\text{~on~}\partial\Omega.
    \end{cases}
    \end{equation*}
Since $\frac{a(x)}{(\varepsilon^2+|x|^{2})^{1/2}}$ remains bounded in $L^{2}(\Omega)$, we deduce from standard elliptic estimate that $h_{\varepsilon}\to h_{0}$ uniformly on $\bar{\Omega}$, where $h_{0}$ is the solution of 
\begin{equation*}
    \begin{cases}
        -\Delta h_{0}+a h_{0}=-\frac{a(x)}{|x|}&\quad\text{~in~}\Omega,\\
        h_{0}=-\frac{1}{|x|}&\quad\text{~on~}\partial\Omega.
    \end{cases}
\end{equation*}
Then \eqref{Ha-pde} gives that $h_{0}(x)=4\pi H_{a}(x,0)$. On the other hand, a direct computation yields that
\begin{equation}\label{equivlent theorem 1-8}
\int_{\Omega}U_{\varepsilon}^{6}=\kappa+O(\varepsilon^{3})\text{~and~} \frac{1}{\sqrt{\varepsilon}}U_{\varepsilon}^{5}\to \kappa'\delta_{0} \text{~weakly in the sense of measure}, 
\end{equation}
where 
\begin{equation*}
    \kappa=\int_{\R^{3}}\frac{dx}{(1+|x|^{2})^{3}}=\frac{\pi^{2}}{4}\text{~and~}\kappa'=\int_{\R^{3}}\frac{dx}{(1+|x|^{2})^{5/2}}=\frac{4}{3}\pi.
\end{equation*}
It then follows that
\begin{equation}\label{equivlent theorem 1-10}
\begin{aligned}
    \int_\Omega |\nabla \phi_{\varepsilon}|^2+ a \phi_{\varepsilon}^2dx&=\int_\Omega (-\Delta
\phi_{\varepsilon}+a \phi_{\varepsilon})\phi_{\varepsilon}dx=\int_\Omega  (-\Delta U_{\varepsilon})\phi_{\varepsilon}dx\\
&=3\int_\Omega U_{\varepsilon}^{5}(U_{\varepsilon}+\sqrt{\varepsilon}h_{\varepsilon})dx=3\kappa+3\kappa'4\pi\phi_{a}(0)\varepsilon+o(\varepsilon).
\end{aligned}
\end{equation}
Furthermore, by Lemma \ref{lem inequality 2} and the HLS inequality, we obtain
\begin{equation}\label{equivlent theorem 1-11}
\begin{aligned}
&\int_{\Omega}\int_{\Omega}\frac{|\phi_{\varepsilon}(x)|^{6-\alpha}|\phi_{\varepsilon}(y)|^{6-\alpha}}{|x-y|^{\alpha}}dxdy=    \int_{\Omega}\int_{\Omega}\frac{|U_{\varepsilon}+\sqrt{\varepsilon}h_{\varepsilon}|^{6-\alpha}(x)|U_{\varepsilon}+\sqrt{\varepsilon}h_{\varepsilon}|^{6-\alpha}(y)}{|x-y|^{\alpha}}dxdy\\
&=\int_{\Omega}\int_{\Omega}\frac{U_{\varepsilon}^{6-\alpha}(x)U_{\varepsilon}^{6-\alpha}(y)}{|x-y|^{\alpha}}dxdy+2(6-\alpha)\int_{\Omega}\int_{\Omega}\frac{U_{\varepsilon}^{6-\alpha}(x)U_{\varepsilon}^{5-\alpha}(y)\sqrt{\varepsilon}h_{\varepsilon}(y)}{|x-y|^{\alpha}}dxdy+o(\varepsilon)\\
&:=I_{1}+I_{2}.
\end{aligned}
\end{equation}
We now estimate the terms $I_{1}$ and $I_{2}$. First, by the HLS inequality, we have
\begin{equation*}
\int_{\R^{3}\setminus\Omega}\int_{\R^{3}}\frac{U_{\varepsilon}^{6-\alpha}(x)U_{\varepsilon}^{6-\alpha}(y)}{|x-y|^{\alpha}}dxdy\leq C_{\alpha}\left(\int_{\R^{3}\setminus\Omega}U_{\varepsilon}^{6}dx\right)^{\frac{6-\alpha}{6}}\left(\int_{\R^{3}}U_{\varepsilon}^{6}dx\right)^{\frac{6-\alpha}{6}}=O(\varepsilon^{\frac{6-\alpha}{2}}).
\end{equation*}
This, together with \eqref{U-x-lamda for choquard equation} and \eqref{Emden-Fowler equation}, gives that
\begin{equation}\label{equivlent theorem 1-13}
\begin{aligned}
I_{1}&=\int_{\R^{3}}\int_{\R^{3}}\frac{U_{\varepsilon}^{6-\alpha}(x)U_{\varepsilon}^{6-\alpha}(y)}{|x-y|^{\alpha}}-\int_{\R^{3}\setminus\Omega}\int_{\R^{3}}\frac{U_{\varepsilon}^{6-\alpha}(x)U_{\varepsilon}^{6-\alpha}(y)}{|x-y|^{\alpha}}-\int_{\Omega}\int_{\R^{3}\setminus\Omega}\frac{U_{\varepsilon}^{6-\alpha}(x)U_{\varepsilon}^{6-\alpha}(y)}{|x-y|^{\alpha}}\\
&=\frac{1}{\bar{C}_{\alpha}^{2(5-\alpha)}}\int_{\R^{3}}|\nabla U_{\varepsilon}|^{2}dx+o(\varepsilon)=\frac{3}{\bar{C}_{\alpha}^{2(5-\alpha)}}\int_{\R^{3}} U_{\varepsilon}^{6}dx+o(\varepsilon)=\frac{3\kappa}{\bar{C}_{\alpha}^{2(5-\alpha)}}+o(\varepsilon).
\end{aligned}    
\end{equation}
On the other hand, by \eqref{important-identity-1} and the HLS inequality, we obtain
\begin{equation}\label{equivlent theorem 1-15}
    I_{2}=\frac{6(6-\alpha)}{\bar{C}_{\alpha}^{2(5-\alpha)}}\int_{\Omega
    }U_{\varepsilon}^{5}\sqrt{\varepsilon}h_{\varepsilon}dx+o(\varepsilon)=\frac{6(6-\alpha)}{\bar{C}_{\alpha}^{2(5-\alpha)}}\kappa'4\pi\phi_{a}(0)\varepsilon+o(\varepsilon).
\end{equation}
Combining \eqref{equivlent theorem 1-11}--\eqref{equivlent theorem 1-15}, we have
\begin{equation*}
\left(\int_{\Omega}\int_{\Omega}\frac{|\phi_{\varepsilon}(x)|^{6-\alpha}|\phi_{\varepsilon}(y)|^{6-\alpha}}{|x-y|^{\alpha}}dxdy\right)^{\frac{1}{6-\alpha}}=\left(\frac{3\kappa}{\bar{C}_{\alpha}^{2(5-\alpha)}}\right)^{\frac{1}{6-\alpha}}\left(1+\frac{k'}{k}8\pi\phi_{a}(0)\varepsilon\right)+o(\varepsilon).
\end{equation*}
This together with \eqref{equivlent theorem 1-10} yields that
\begin{equation*}
    J(\phi_{\varepsilon})=\frac{3\kappa(1+\frac{\kappa'}{\kappa}4\pi\phi_{a}(0)\varepsilon)+o(\varepsilon)}{\left(\frac{3\kappa}{\bar{C}_{\alpha}^{2(5-\alpha)}}\right)^{\frac{1}{6-\alpha}}\left(1+\frac{k'}{k}8\pi\phi_{a}(0)\varepsilon\right)+o(\varepsilon)}.
\end{equation*}
Moreover, by \eqref{best sobolev constant}, \eqref{definition of SHL} and \eqref{equivlent theorem 1-8}, we obtain
\begin{equation*}
    J(\phi_{\varepsilon})=S_{HL}-C\phi_{a}(0)\varepsilon+o(\varepsilon),
\end{equation*}
where $C=4\pi\kappa'S_{HL}/\kappa=\frac{64}{3}S_{HL}$. This completes the proof of this step.

\textbf{Step 2.} In this step, we prove $(2)\Rightarrow(3)$. Suppose that $S_{HL}(a)<S_{HL}$, and let $\{u_{\varepsilon}\}\subset H_{0}^{1}(\Omega)$ be a minimizing sequence for \eqref{minimizing-problem} such that, as $\varepsilon\to0$,
\begin{equation}\label{energy estimate proof 4}
    \|u_{\varepsilon}\|_{HL}=1\text{~and~}\int_\Omega |\nabla u_{\varepsilon}|^2+ a u_{\varepsilon}^2dx=S_{HL}(a)+o(1). 
\end{equation}
Since the operator $-\Delta+a$ is coercive, the sequence $u_{\varepsilon}$ is bounded in $H_{0}^{1}(\Omega)$. Thus, we can extract a subsequence, still denoted by $u_{\varepsilon}$, such that
\begin{equation*}
\begin{aligned}
    u_{\varepsilon}&\to u\text{~weakly in~}H_{0}^{1}(\Omega),\\
     u_{\varepsilon}&\to u\text{~strongly in~}L^{2}(\Omega),\\
      u_{\varepsilon}&\to u\text{~a.e. on~}\Omega,
\end{aligned}
\end{equation*}
and consequently $\|u\|_{HL}\leq 1.$ Set $v_{\varepsilon}=u_{\varepsilon}-u$, then
\begin{equation}\label{energy estimate proof 6}
    \begin{aligned}
        v_{\varepsilon}&\to 0\text{~weakly in~}H_{0}^{1}(\Omega),\\
     v_{\varepsilon}&\to 0\text{~strongly in~}L^{2}(\Omega),\\
      v_{\varepsilon}&\to 0\text{~a.e. on~}\Omega.
    \end{aligned}
\end{equation}
Moreover, since, $S_{HL}(0)=S_{HL}$ and $\|u_{\varepsilon}\|_{HL}=1$, we have $\|\nabla u_{\varepsilon}\|_{2}^{2}\geq S_{HL}.$ Combining this with \eqref{energy estimate proof 4}, we find that 
\begin{equation*}
    -\int_{\Omega} au_{\varepsilon}^{2}dx\geq S_{HL}+o(1)-S_{HL}(a)>0,
\end{equation*}
for sufficiently small $\varepsilon$. Hence, $u \not\equiv 0$. Furthermore, combining \eqref{energy estimate proof 4}, \eqref{energy estimate proof 6} and Brezis-Lieb lemma \cite{BrezisLieblemma,Gao2016TheBT} yields
\begin{equation}\label{energy estimate proof 7}
    \int_\Omega |\nabla v_{\varepsilon}|^2+|\nabla u|^2+a u^2dx=S_{HL}(a)+o(1)\text{~and~}1=\|v_{\varepsilon}\|_{HL}^{2(6-\alpha)}+\|u\|_{HL}^{2(6-\alpha)}+o(1). 
\end{equation}
Consequently,
\begin{equation*}
    1\leq\|v_{\varepsilon}\|_{HL}^{2}+\|u\|_{HL}^{2}+o(1). 
\end{equation*}
Since $S_{HL}(0)=S_{HL}$ and $S_{HL}(a)<S_{HL}$, we deduce that
\begin{equation*}
    S_{HL}(a)\leq  S_{HL}(a)\|u\|_{HL}^{2}+\frac{S_{HL}(a)}{S_{HL}}\|\nabla v_{\varepsilon}\|_{2}^{2}+o(1)\leq S_{HL}(a)\|u\|_{HL}^{2}+ \|\nabla v_{\varepsilon}\|_{L^{2}(\Omega)}^{2}+o(1).
\end{equation*}
This together with \eqref{energy estimate proof 7} yields that
\begin{equation*}
    \int_\Omega |\nabla u|^2+a u^2dx\leq S_{HL}(a)\|u\|_{HL}^{2}. 
\end{equation*}
Therefore, $u$ is a minimizer of $S_{HL}(a)$, which completes the proof.
\end{proof}

\begin{Prop}\label{assumption on a}
Assume $\alpha\in(0,3)$ is sufficiently small. Then $(3)\Rightarrow(2)$ holds. 
\end{Prop}

\begin{proof}
We argue by contradiction. Suppose that $S_{HL}(a)=S_{HL}$ and that $S_{HL}(a)$ is achieved. Then there exists $u_{a}$ such that
\begin{equation*}
    \begin{cases}
        -\Delta u_{a}+au_{a}=S_{HL}\left(\int_{\Omega}\frac{u_{a}^{6-\alpha}(y)}{|x-y|^\alpha}dy\right)u_{a}^{5-\alpha}
   &\mbox{in}\ \Omega,\\
  u_{a}>0
  \ \  &\mbox{in}\  \Omega,\\
  u_{a}=0
  \ \  &\mbox{on}\ \partial \Omega,
    \end{cases}
\end{equation*}
and
\begin{equation}\label{assumption on a-eq-2}
    \int_{\Omega}\int_{\Omega}\frac{u_{a}^{6-\alpha}(x)u_{a}^{6-\alpha}(y)}{|x-y|^{\alpha}}dxdy=1.
\end{equation}
Since $S_{HL}(a)=S_{HL}$, for any $\varphi\in C^{\infty}(\R^3)$ and $\varepsilon>0$, it holds that
\begin{equation*}
\begin{aligned}
&\left(\int_{\Omega}\int_{\Omega}\frac{(u_{a}(1+\varepsilon\varphi))^{6-\alpha}(x)(u_{a}(1+\varepsilon\varphi))^{6-\alpha}(y)}{|x-y|^{\alpha}}dxdy\right)^{\frac{1}{6-\alpha}}\\
&\quad\leq S_{HL}^{-1} \int_\Omega |\nabla( u_{a}(1+\varepsilon\varphi))|^2+ a (u_{a}(1+\varepsilon\varphi))^2dx .
\end{aligned}
\end{equation*}
By \eqref{assumption on a-eq-2}, the Taylor's expansion and direct computations, we obtain 
\begin{equation*}
\begin{aligned}
&\left(\int_{\Omega}\int_{\Omega}\frac{|u_{a}(x)(1+\varepsilon\varphi(x))|^{6-\alpha}|u_{a}(y)(1+\varepsilon\varphi(y))|^{6-\alpha}}{|x-y|^{\alpha}}dxdy\right)^{\frac{1}{6-\alpha}}\\
&\quad=1+2A_{1}\varepsilon+\left((6-\alpha)A_{2}+(5-\alpha)(A_{3}-2A_{1}^{2})\right)\varepsilon^{2}+o(\varepsilon^2)
\end{aligned}
\end{equation*}
and
\begin{equation*}
\begin{aligned}
&S_{HL}^{-1}\int_\Omega |\nabla( u_{a}(1+\varepsilon\varphi))|^2+ a (u_{a}(1+\varepsilon\varphi))^2dx \\   
&\quad=S_{HL}^{-1}\int_\Omega u_{a}(-\Delta
u_{a}+au_{a})(1+\varepsilon\varphi)^2+\varepsilon^2 u_{a}^2|\nabla \varphi|^{2}dx\\
&\quad=\int_{\Omega}\left(\int_\Omega\frac{u_{a}^{6-\alpha}(y)}{|x-y|^\alpha}dy\right)u_{a}^{6-\alpha}(x)(1+\varepsilon\varphi(x))^2dx+S_{HL}^{-1}\varepsilon^2\int_{\Omega}u_{a}^2|\nabla \varphi|^{2}dx\\
&\quad=1+2A_{1}\varepsilon+A_{3}\varepsilon^2+S_{HL}^{-1}\varepsilon^2\int_{\Omega}u_{a}^2|\nabla \varphi|^{2}dx,
\end{aligned}
\end{equation*}
where
\begin{equation*}
\begin{aligned}
A_{1}&:=\int_{\Omega}\int_{\Omega}\frac{u_{a}^{6-\alpha}(x)u_{a}^{6-\alpha}(y)\varphi(y)}{|x-y|^{\alpha}}dxdy,\\
A_{2}&:=\int_{\Omega}\int_{\Omega}\frac{u_{a}^{6-\alpha}(x)\varphi(x)u_{a}^{6-\alpha}(y)\varphi(y)}{|x-y|^{\alpha}}dxdy, \\ 
A_{3}&:=\int_{\Omega}\int_{\Omega}\frac{u_{a}^{6-\alpha}(x)u_{a}^{6-\alpha}(y)\varphi^2(y)}{|x-y|^{\alpha}}dxdy.
\end{aligned}
\end{equation*}
It then follows that for any $\varphi\in C^{\infty}(\R^3)$
\begin{equation*}
(6-\alpha)A_{2}+(4-\alpha)A_{3}\leq 2(5-\alpha)A_{1}^{2}+S_{HL}^{-1}\int_{\Omega}u_{a}^2|\nabla \varphi|^{2}dx.
\end{equation*}
Observe that by \eqref{assumption on a-eq-2}, the semigroup property of the Riesz potential and the Cauchy–Schwarz inequality, we have
\begin{equation*}
    \begin{aligned}
        A_{1}^{2}=\left(\int_{\Omega}\int_{\Omega}\frac{u_{a}^{6-\alpha}(x)u_{a}^{6-\alpha}(y)\varphi(y)}{|x-y|^{\alpha}}dxdy\right)^{2}&\leq \int_{\Omega}\int_{\Omega}\frac{u_{a}^{6-\alpha}(x)\varphi(x)u_{a}^{6-\alpha}(y)\varphi(y)}{|x-y|^{\alpha}}dxdy\\\
        &\quad\times \int_{\Omega}\int_{\Omega}\frac{u_{a}^{6-\alpha}(x)u_{a}^{6-\alpha}(y)}{|x-y|^{\alpha}}dxdy\\
        &=A_{2}.
    \end{aligned}
\end{equation*}
Thus for any $\varphi\in C^{\infty}(\R^3)$
\begin{equation}\label{assumption on a-eq-57}
(4-\alpha)A_{3}\leq (4-\alpha)A_{1}^{2}+S_{HL}^{-1}\int_{\Omega}u_{a}^2|\nabla \varphi|^{2}dx.
\end{equation}

For $(z,t)\in\R^{3}\times(0,\infty)$, we define 
\begin{equation*}
    \begin{aligned}
        F(z,t):=\int_{\Omega}\int_{\Omega}\frac{u_{a}^{6-\alpha}(x)u_{a}^{6-\alpha}(y)2t(y-z)}{|x-y|^{\alpha}(1+t^{2}|y-z|^{2})}dxdy
    \end{aligned}
\end{equation*}
and
\begin{equation*}
    \begin{aligned}
        G(z,t):=\int_{\Omega}\int_{\Omega}\frac{u_{a}^{6-\alpha}(x)u_{a}^{6-\alpha}(y)(1-t^{2}|y-z|^{2})}{|x-y|^{\alpha}(1+t^{2}|y-z|^{2})}dxdy.
    \end{aligned}
\end{equation*}
Moreover, we define the function $H:\R^{3}\times\R\to \R^{4}$ by 
\begin{equation*}
    \begin{aligned}
        H(z,s):=\left(F\left(z,\frac{s+\sqrt{s^{2}+4}}{2}\right)+z,G\left(z,\frac{s+\sqrt{s^{2}+4}}{2}\right)+s\right).
    \end{aligned}
\end{equation*}
It follows from \eqref{assumption on a-eq-2} that
\begin{equation*}
    |H(z,s)|^{2}\leq |z|^{2}+s^{2}
\end{equation*}
for $|z|^{2}+s^{2}$ sufficiently large. Let $t:=\frac{s+\sqrt{s^2+4}}{2}>0$ and define
\begin{equation*}
    \varphi_{i}(x)=\frac{2t(x_{i}-z_{i})}{1+t^{2}|x-z|^{2}},~i=1,2,3~,\quad \varphi_{4}(x)=\frac{1-t^{2}|x-z|^{2}}{1+t^{2}|x-z|^{2}}.
\end{equation*}
Then, by the Brouwer fixed point theorem, there exists $(z,t)\in\R^{3}\times(0,\infty)$ with $|z|^{2}+(t-t^{-1})^{2}$ sufficiently large such that for every $i=1,\cdots,4$,
\begin{equation*}
    \begin{aligned}
         \int_{\Omega}\int_{\Omega}\frac{u_{a}^{6-\alpha}(x)u_{a}^{6-\alpha}(y)\varphi_{i}(y)}{|x-y|^{\alpha}}dxdy=0.
    \end{aligned}
\end{equation*}
On the other hand, a direct computation yields that 
\begin{equation*}
    \begin{aligned}
        \sum_{i=1}^{4}\varphi_{i}^{2}=1,\quad \sum_{i=1}^{4}|\nabla\varphi_{i}|^{2}=\frac{12t^{2}}{(1+t^{2}|x-z|^{2})^{2}}.
    \end{aligned}
\end{equation*}
This together with \eqref{assumption on a-eq-2} and \eqref{assumption on a-eq-57} yields that
\begin{equation}\label{assumption on a-eq-67}
    \begin{aligned}
        (4-\alpha)\leq12S_{HL}^{-1}\int_{\Omega}\frac{t^{2}}{(1+t^{2}|x-z|^{2})^{2}}u_{a}^{2}dx.
    \end{aligned}
\end{equation}

For any $1<r<\frac{r_{1}}{2}<3-\frac{\alpha}{2}$ with $\theta=\frac{r}{r-1}$, it follows from the H\"older inequality, the RHLS inequality and \eqref{assumption on a-eq-2} that
\begin{equation*}
    \begin{aligned}
        &\int_{\Omega}\frac{t^{2}}{(1+t^{2}|x-z|^{2})^{2}}u_{a}^{2}dx\\
        &\leq \left(\int_{\Omega}\frac{t^{2\theta}}{(1+t^{2}|x-z|^{2})^{2\theta}}dx\right)^{\frac{1}{\theta}}\left(\int_{\Omega}u_{a}^{2r}dx\right)^{\frac{1}{r}}\\
         &\leq \left(\int_{\Omega}\frac{t^{2\theta}}{(1+t^{2}|x-z|^{2})^{2\theta}}dx\right)^{\frac{1}{\theta}}\left(\left(\int_{\Omega}(u_{a}^{r_{1}})^{\frac{2r}{r_{1}}}dx\right)^{\frac{r_{1}}{2r}}\right)^{\frac{2}{r_{1}}}\\
        &\leq \tilde{C}_{3(\frac{r_{1}}{r}-2)}^{-\frac{1}{r_{1}}}\left(\int_{\Omega}\frac{t^{2\theta}}{(1+t^{2}|x-z|^{2})^{2\theta}}dx\right)^{\frac{1}{\theta}}\left(\int_{\Omega}\int_{\Omega}u_{a}^{r_{1}}(x)u_{a}^{r_{1}}(y)|x-y|^{3(\frac{r_{1}}{r}-2)}dxdy\right)^{\frac{1}{r_{1}}}\\
        &\leq \tilde{C}_{3(\frac{r_{1}}{r}-2)}^{-\frac{1}{r_{1}}}\left(\int_{\Omega}\frac{t^{2\theta}}{(1+t^{2}|x-z|^{2})^{2\theta}}dx\right)^{\frac{1}{\theta}}\left(\int_{\Omega}\int_{\Omega}\frac{u_{a}^{6-\alpha}(x)u_{a}^{6-\alpha}(y)}{|x-y|^{\alpha}}dxdy\right)^{\frac{1}{6-\alpha}}\\
        &\quad\times\left(\int_{\Omega}\int_{\Omega}|x-y|^{(3(\frac{r_{1}}{r}-2)+\frac{\alpha r_{1}}{6-\alpha})(\frac{6-\alpha}{6-\alpha-r_{1}})}dxdy\right)^{\frac{6-\alpha-r_{1}}{(6-\alpha)r_{1}}}\\
         &\leq \tilde{C}_{3(\frac{r_{1}}{r}-2)}^{-\frac{1}{r_{1}}}\left(\int_{\Omega}\frac{t^{\frac{2r}{r-1}}}{(1+t^{2}|x-z|^{2})^{\frac{2r}{r-1}}}dx\right)^{\frac{r-1}{r}}\left(\int_{\Omega}\int_{\Omega}|x-y|^{(3(\frac{r_{1}}{r}-2)+\frac{\alpha r_{1}}{6-\alpha})(\frac{6-\alpha}{6-\alpha-r_{1}})}dxdy\right)^{\frac{6-\alpha-r_{1}}{(6-\alpha)r_{1}}}.
    \end{aligned}
\end{equation*}
Let
\begin{equation*}
    r=3-2\alpha\text{~and~}r_{1}=6-2\alpha.
\end{equation*}
As $\alpha\to0$, it follows that
\begin{equation*}
    \begin{aligned}
        \tilde{C}_{3(\frac{r_{1}}{r}-2)}^{-\frac{1}{r_{1}}}=\left(\pi^{-\frac{3(\frac{r_{1}}{r}-2)}{2}}\frac{\Gamma(\frac{3+3(\frac{r_{1}}{r}-2)}{2})}{\Gamma(3+\frac{3(\frac{r_{1}}{r}-2)}{2})}\left(\frac{\Gamma(3)}{\Gamma(\frac{3}{2})}\right)^{\frac{3+3(\frac{r_{1}}{r}-2)}{3}}\right)^{-\frac{1}{r_{1}}}\to 1,
    \end{aligned}
\end{equation*}
\begin{equation*}
    \begin{aligned}
        &\left(\int_{\Omega}\frac{t^{\frac{2r}{r-1}}}{(1+t^{2}|x-z|^{2})^{\frac{2r}{r-1}}}dx\right)^{\frac{r-1}{r}}\\
        &=\left(\int_{\R^{3}}\frac{t^{\frac{2r}{r-1}}}{(1+t^{2}|x-z|^{2})^{\frac{2r}{r-1}}}dx\right)^{\frac{r-1}{r}}-\left(\int_{\R^{3}\setminus\Omega}\frac{t^{\frac{2r}{r-1}}}{(1+t^{2}|x-z|^{2})^{\frac{2r}{r-1}}}dx\right)^{\frac{r-1}{r}}\\
        &\to\left(\frac{\pi}{2}\right)^{\frac{4}{3}}-\left(\int_{\R^{3}\setminus{t(\Omega-z)}}\frac{1}{(1+|x|^{2})^{3}}dx\right)^{\frac{2}{3}},
    \end{aligned}
\end{equation*}
\begin{equation*}
    \left(\int_{\Omega}\int_{\Omega}|x-y|^{(3(\frac{r_{1}}{r}-2)+\frac{\alpha r_{1}}{6-\alpha})(\frac{6-\alpha}{6-\alpha-r_{1}})}dxdy\right)^{\frac{6-\alpha-r_{1}}{(6-\alpha)r_{1}}}\to 1,
\end{equation*}
\begin{equation*}
    \begin{aligned}
        12S_{HL}^{-1}=12S^{-1}C_{\alpha}^{\frac{1}{6-\alpha}}=4\left( \frac\pi 2 \right)^{-4/3}\left(\pi^{\frac{\alpha}{2}}\frac{\Gamma(\frac{3-\alpha}{2})}{\Gamma(3-\frac{\alpha}{2})}\left(\frac{\Gamma(3)}{\Gamma(\frac{3}{2})}\right)^{\frac{3-\alpha}{3}}\right)^{\frac{1}{6-\alpha}}\to 4\left( \frac\pi 2 \right)^{-\frac{4}{3}}.
    \end{aligned}
\end{equation*}
Thus we obtain that, as $\alpha\to0$, 
\begin{equation*}
    12S_{HL}^{-1}\int_{\Omega}\frac{t^{2}}{(1+t^{2}|x-z|^{2})^{2}}u_{a}^{2}dx\to 4- 4\left( \frac\pi 2 \right)^{-4/3}\left(\int_{\R^{3}\setminus{t(\Omega}-z)}\frac{1}{(1+|x|^{2})^{3}}dx\right)^{\frac{2}{3}}<4.
\end{equation*}
This contradicts \eqref{assumption on a-eq-67}, and the proof is complete.
\end{proof}

We now assume $S_{HL}(a)<S_{HL}$. It follows from Proposition \ref{coro for S-HL-a-delta} that there exists a critical constant $B(a)>0$ such that
\begin{equation}\label{assum-1}
    S_{HL}(a+\varepsilon)=S_{HL}\text{~for any~}\varepsilon\geq B(a)\text{~and~}S_{HL}(a+\varepsilon)<S_{HL}\text{~for any~}\varepsilon<B(a).
\end{equation}
Let $\bar{a}:=a+B(a)$. Theorem \ref{thm-1} yields that $S_{HL}(\bar{a}-\varepsilon)$ is achieved for each $\varepsilon>0$. Let $u_{\varepsilon}\in H^{1}_{0}(\Omega)$ be a minimizer for $S_{HL}(\bar{a}-\varepsilon)$. Then by the standard elliptic regularity theory, the Lagrange multiplier theorem and the maximum principle, we obtain
\begin{equation}\label{equation for u varepsilon}
    \begin{cases}
        -\Delta u_{\varepsilon}+(\bar{a}-\varepsilon) u_{\varepsilon}=\left(\int_{\Omega}\frac{u_{\varepsilon}^{6-\alpha}(y)}{|x-y|^\alpha}dy\right)u_{\varepsilon}^{5-\alpha}
  \ \  &\mbox{in}\ \Omega,\\
  u_{\varepsilon}>0 \ \ &\mbox{in}\ \Omega,\\
  u_{\varepsilon}=0
  \ \  &\mbox{on}\ \partial \Omega,
    \end{cases}
\end{equation}
and 
\begin{equation}\label{renormalization u varepsilon}
   \int_{\Omega}\int_{\Omega}\frac{u_{\varepsilon}^{6-\alpha}(x)u_{\varepsilon}^{6-\alpha}(y)}{|x-y|^{\alpha}}dxdy=S_{HL}^{\frac{6-\alpha}{5-\alpha}}(\bar{a}-\varepsilon). 
\end{equation}

\begin{Prop} \label{lemma-1}  
The sequence $\{u_{\varepsilon}\}$ is bounded in $H^{1}_{0}(\Omega)$. Moreover, up to a subsequence, we have as $\varepsilon\to 0$ 
\begin{equation*}
    \begin{aligned}
     u_{\varepsilon}&\to 0\text{~weakly but not strongly in~}H_{0}^{1}(\Omega),\\
     u_{\varepsilon}&\to 0\text{~strongly in~}L^{2}(\Omega),\\
     u_{\varepsilon}&\to 0\text{~a.e. on~}\Omega.
    \end{aligned}
\end{equation*}
Furthermore, there are sequences $\{\mu_{\varepsilon}\}\subset\R$, $\{\xi_{\varepsilon}\}\subset\Omega$, $\{\lambda_{\varepsilon}\}\subset\R^+$ and $\{w_{\varepsilon}\}\subset T_{\xi_{\varepsilon},\lambda_{\varepsilon}}^{\perp}$ such that, up to a subsequence,
\begin{equation*}
u_{\varepsilon}=\mu_{\varepsilon}(P\bar{U}_{\xi_{\varepsilon},\lambda_{\varepsilon}}+w_{\varepsilon}).
\end{equation*}
Let $d_{\varepsilon}:=\text{dist~}(\xi_{\varepsilon},\partial\Omega)$. As $\varepsilon\to0$, the following hold
\begin{equation*}
    \mu_{\varepsilon}\to 1,\ \ \xi_{\varepsilon}\to \xi_{0}\in\bar{\Omega},\ \ \lambda_{\varepsilon}d_{\varepsilon}\to\infty,\ \|w_{\varepsilon}\|_{H^{1}_{0}(\Omega)}\to 0.
\end{equation*}
\end{Prop}

\begin{proof}
Integrating the equation \eqref{equation for u varepsilon} against $u_{\varepsilon}$, and using \eqref{assum-1} and \eqref{renormalization u varepsilon}, we have
   \begin{equation}\label{first expansion-proof-1}
   \begin{aligned}
        \int_\Omega |\nabla u_{\varepsilon}|^2+ (\bar{a}-\varepsilon) u_{\varepsilon}^2dx&=\int_{\Omega}\int_{\Omega}\frac{u_{\varepsilon}^{6-\alpha}(x)u_{\varepsilon}^{6-\alpha}(y)}{|x-y|^{\alpha}}dxdy\\
        &=S_{HL}^{\frac{6-\alpha}{5-\alpha}}(\bar{a}-\varepsilon)\to S_{HL}^{\frac{6-\alpha}{5-\alpha}}\text{~as~}\varepsilon\to0.
   \end{aligned}    
   \end{equation}
Combining this with the coercivity of the operator $-\Delta+(\bar{a}-\varepsilon)$, we conclude that $\{u_{\varepsilon}\}$ is bounded in $H^{1}_{0}(\Omega)$. Then, up to a subsequence, there exists $u_{0}\in H^{1}_{0}(\Omega)$ such that $u_{\varepsilon}\to u_{0}$ weakly in $H^{1}_{0}(\Omega)$ as $\varepsilon\to 0$. Next, we show that $u_{0}\equiv 0$. Let $v_{\varepsilon}:=u_{\varepsilon}-u_{0}$. By Rellich theorem, up to a subsequence, we have
\begin{equation*}
    \begin{aligned}
     v_{\varepsilon}&\to 0\text{~weakly in~}H_{0}^{1}(\Omega),\\
     v_{\varepsilon}&\to 0\text{~strongly in~}L^{2}(\Omega),\\
      v_{\varepsilon}&\to 0\text{~a.e. on~}\Omega.
    \end{aligned}
\end{equation*}
It then follows from the Brezis-Lieb lemma (see \cite{BrezisLieblemma,MOROZ2013153}) that
\begin{equation*}
 \mathcal{T}+\int_\Omega |\nabla u_{0}|^2dx+\bar{a}\int_\Omega | u_{0}|^2dx=S_{HL}\left(\mathcal{M}+\|u_{0}\|_{HL}^{2(6-\alpha)}\right)^{\frac{1}{6-\alpha}}, 
\end{equation*}
where $\|\cdot\|_{HL}$ is the norm defined in \eqref{defin-norm-HL},
\begin{equation*}
    \mathcal{T}:=\lim_{\varepsilon\to 0}\int_\Omega |\nabla v_{\varepsilon}|^2dx\text{~and~}\mathcal{M}:=\lim_{\varepsilon\to 0}\|v_{\varepsilon}\|_{HL}^{2(6-\alpha)}.
\end{equation*}
Moreover, by $S_{HL}(0)=S_{HL}$ and Lemma \ref{lem inequality 1}, we find that 
\begin{equation*}
    \mathcal{T}\geq S_{HL}\mathcal{M}^{\frac{1}{6-\alpha}}
\end{equation*}
and 
\begin{equation*}
  \left(\mathcal{M}+\|u_{0}\|_{HL}^{2(6-\alpha)}\right)^{\frac{1}{6-\alpha}}\leq \mathcal{M}^{\frac{1}{6-\alpha}} + \|u_{0}\|_{HL}^{2}.
\end{equation*}
Combining the estimates above, we obtain
\begin{equation*}
    S_{HL}\|u_{0}\|_{HL}^{2}\geq\int_\Omega |\nabla u_{0}|^2dx+\bar{a}\int_\Omega | u_{0}|^2dx.
\end{equation*}
Thus, either $u_{0}\equiv 0$ or $u_{0}$ is a minimizer of $S_{HL}(\bar{a})$. By Proposition \ref{assumption on a}, we conclude that $u_{0}\equiv 0$. Now, if $u_{\varepsilon}\to 0$ strongly in $H^{1}_{0}(\Omega)$ as $\varepsilon\to 0$, then by the HLS inequality and the Sobolev inequality, we have $\|u\|_{HL}^{2(6-\alpha)}\to 0$ as $\varepsilon\to 0$, which contradicts \eqref{renormalization u varepsilon}. Therefore, $u_{\varepsilon}$ does not converge strongly to $0$ in $H^{1}_{0}(\Omega)$. 

Since $u_{\varepsilon}\to 0$ strongly in $L^{2}(\Omega)$, it follows that $u_{\varepsilon}$ is a minimizing sequence for $S_{HL}$ (see \eqref{optimization problem for shl}). Therefore, by the concentration-compactness theorem (see \cite{Yang2023BlowUpBO}), there exist sequences $\{z_{\varepsilon}\}\subset\Omega$, $\eta_{\varepsilon}\subset\R^{+}$ and $\sigma_{\varepsilon}\subset \dot H^{1}(\R^{3})$ such that
\begin{equation*}
    u_{\varepsilon}=\bar{U}_{z_{\varepsilon},\eta_{\varepsilon}}+\sigma_{\varepsilon}, 
\end{equation*}
with $\eta_{\varepsilon}\text{dist}(z_{\varepsilon},\partial\Omega)\to\infty \text{~and~}\sigma_{\varepsilon}\to0\text{~strongly in~}\dot H^{1}(\R^{3})\text{~as~}\varepsilon\to0$.
Moreover, by Lemma \ref{lemma PU}, we have
\begin{equation*}
    \|u_{\varepsilon}-P\bar{U}_{z_{\varepsilon},\eta_{\varepsilon}}\|_{H^{1}_{0}(\Omega)}\to0\text{~as~}\varepsilon\to0.
\end{equation*}
Combining this with \cite[Proposition 7]{Bahri1988CPAM}, we find that there exist sequences $\{\mu_{\varepsilon}\}\subset\R$, $\{\xi_{\varepsilon}\}\subset \Omega$, $\{\lambda_{\varepsilon}\}\subset \R^{+}$, and functions $w_{\varepsilon}\in T_{\xi_{\varepsilon},\lambda_{\varepsilon}}^{\perp}$ such that
\begin{equation*}
u_{\varepsilon}=\mu_{\varepsilon}(P\bar{U}_{\xi_{\varepsilon},\lambda_{\varepsilon}}+w_{\varepsilon}),
\end{equation*}
where $\{\mu_{\varepsilon}\}$ is bounded, $\lambda_{\varepsilon}\text{dist}(\xi_{\varepsilon},\partial\Omega)\to\infty$, and $\|w_{\varepsilon}\|_{H^{1}_{0}(\Omega)}\to 0$ as $\varepsilon\to 0$. Finally, by Lemma \ref{lemma PU}, we obtain
\begin{equation*}
    \begin{aligned}
         \int_\Omega |\nabla u_{\varepsilon}|^2+ (\bar{a}-\varepsilon) u_{\varepsilon}^2dx&=\mu_{\varepsilon}^{2}\int_{\Omega}|\nabla P\bar{U}_{\xi_{\varepsilon},\lambda_{\varepsilon}}|^{2}dx+o(1)=\mu_{\varepsilon}^{2}S_{HL}^{\frac{6-\alpha}{5-\alpha}}+o(1).
    \end{aligned}
\end{equation*}
It then follows from \eqref{first expansion-proof-1} that $\mu_{\varepsilon}\to 1$. This completes the proof.
\end{proof}

\begin{Lem}\label{lemma coercivity}
There exists constant $\rho>0$ such that 
\begin{equation} \label{coercivity}
\begin{aligned}
\int_\Omega |\nabla v|^2 + \bar{a} v^2 dx&-(6-\alpha)\int_{\Omega}\int_{\Omega}\frac{P\bar{U}_{\xi_{\varepsilon},\lambda_{\varepsilon}}^{5-\alpha}(y)v(y)P\bar{U}_{\xi_{\varepsilon},\lambda_{\varepsilon}}^{5-\alpha}(x)v(x)}{|x-y|^{\alpha}}dxdy\\
 &-(5-\alpha)\int_{\Omega}\int_{\Omega}\frac{P\bar{U}_{\xi_{\varepsilon},\lambda_{\varepsilon}}^{6-\alpha}(y)P\bar{U}_{\xi_{\varepsilon},\lambda_{\varepsilon}}^{4-\alpha}(x)v^2(x)}{|x-y|^{\alpha}}dxdy \\
 &\geq \rho \int_\Omega |\nabla v|^2dx,  
\end{aligned}
\end{equation}
for any $v\in T_{\xi_{\varepsilon},\lambda_{\varepsilon}}^\bot$ and any sufficiently small $\varepsilon>0$. 
\end{Lem}

\begin{proof}
\textbf{Step 1} In this step, we will show that there exists a constant $\rho>0$ such that
 \begin{equation}\label{coercivity proof step 1-1} 
\begin{aligned}
\int_\Omega |\nabla v|^2 dx&-(6-\alpha)\int_{\Omega}\int_{\Omega}\frac{P\bar{U}_{\xi_{\varepsilon},\lambda_{\varepsilon}}^{5-\alpha}(y)v(y)P\bar{U}_{\xi_{\varepsilon},\lambda_{\varepsilon}}^{5-\alpha}(x)v(x)}{|x-y|^{\alpha}}dxdy\\
 &-(5-\alpha)\int_{\Omega}\int_{\Omega}\frac{P\bar{U}_{\xi_{\varepsilon},\lambda_{\varepsilon}}^{6-\alpha}(y)P\bar{U}_{\xi_{\varepsilon},\lambda_{\varepsilon}}^{4-\alpha}(x)v^2(x)}{|x-y|^{\alpha}}dxdy \\
 &\geq \rho \int_\Omega |\nabla v|^2 ,    
\end{aligned}
\end{equation}
for any $\varepsilon>0$ small enough and any $v\in T_{\xi_{\varepsilon},\lambda_{\varepsilon}}^\bot$. First, we define the operator $L_{\varepsilon}:H^{1}_{0}(\Omega)\to H^{1}_{0}(\Omega)$ as follows
\begin{equation*}
\begin{aligned}
     (L_{\varepsilon}v)(x):=(-\Delta v)(x)&-(6-\alpha)\left(\int_{\Omega}\frac{P\bar{U}_{\xi_{\varepsilon},\lambda_{\varepsilon}}^{5-\alpha}(y)v(y)}{|x-y|^{\alpha}}dy\right)P\bar{U}_{\xi_{\varepsilon},\lambda_{\varepsilon}}^{5-\alpha}(x)\\
     &-(5-\alpha)\left(\int_{\Omega}\frac{P\bar{U}_{\xi_{\varepsilon},\lambda_{\varepsilon}}^{6-\alpha}(y)}{|x-y|^{\alpha}}dy\right)P\bar{U}_{\xi_{\varepsilon},\lambda_{\varepsilon}}^{4-\alpha}(x)v(x).
\end{aligned}
\end{equation*}
Notice that operator $\Pi_{\xi_{\varepsilon},\lambda_{\varepsilon}}^{\bot}L_{\varepsilon}:T^{\bot}_{\xi_{\varepsilon},\lambda_{\varepsilon}}\to T^{\bot}_{\xi_{\varepsilon},\lambda_{\varepsilon}}$ is self-adjoint and for any given $u,v\in T^{\bot}_{\xi_{\varepsilon},\lambda_{\varepsilon}}$
\begin{equation*}
    \begin{aligned}
        \langle \Pi_{\xi_{\varepsilon},\lambda_{\varepsilon}}^{\bot}L_{\varepsilon}u,v\rangle&=\langle L_{\varepsilon}u,v\rangle-\langle\Pi_{\xi_{\varepsilon},\lambda_{\varepsilon}}L_{\varepsilon}u,v\rangle\\
        &=\int_\Omega \nabla u\cdot\nabla v dx-(6-\alpha)\int_{\Omega}\int_{\Omega}\frac{P\bar{U}_{\xi_{\varepsilon},\lambda_{\varepsilon}}^{5-\alpha}(y)u(y)P\bar{U}_{\xi_{\varepsilon},\lambda_{\varepsilon}}^{5-\alpha}(x)v(x)}{|x-y|^{\alpha}}dxdy\\
 &\quad-(5-\alpha)\int_{\Omega}\int_{\Omega}\frac{P\bar{U}_{\xi_{\varepsilon},\lambda_{\varepsilon}}^{6-\alpha}(y)P\bar{U}_{\xi_{\varepsilon},\lambda_{\varepsilon}}^{4-\alpha}(x)u(x)v(x)}{|x-y|^{\alpha}}dxdy.
    \end{aligned}
\end{equation*}
It then suffices to show that there exists a constant $\rho>0$ such that 
\begin{equation*}
    \|\Pi_{\xi_{\varepsilon},\lambda_{\varepsilon}}^{\bot}L_{\varepsilon}v\|_{H^{1}_{0}}\geq \rho \|v\|_{H^{1}_{0}},
\end{equation*}
for any $\varepsilon>0$ small enough and any $v\in T_{\xi_{\varepsilon},\lambda_{\varepsilon}}^\bot$. Assume by contradiction that there exists $v_{\varepsilon}\in T_{\xi_{\varepsilon},\lambda_{\varepsilon}}^\bot$ with $\|v_{\varepsilon}\|_{H^{1}_{0}(\Omega)}=1$ such that for any $\varphi\in T_{\xi_{\varepsilon},\lambda_{\varepsilon}}^\bot$
\begin{equation*}
  \begin{aligned}
\int_\Omega \nabla v_{\varepsilon}\cdot\nabla\varphi dx&-(6-\alpha)\int_{\Omega}\int_{\Omega}\frac{P\bar{U}_{\xi_{\varepsilon},\lambda_{\varepsilon}}^{5-\alpha}(y)v_{\varepsilon}(y)P\bar{U}_{\xi_{\varepsilon},\lambda_{\varepsilon}}^{5-\alpha}(x)\varphi(x)}{|x-y|^{\alpha}}dxdy\\
 &-(5-\alpha)\int_{\Omega}\int_{\Omega}\frac{P\bar{U}_{\xi_{\varepsilon},\lambda_{\varepsilon}}^{6-\alpha}(y)P\bar{U}_{\xi_{\varepsilon},\lambda_{\varepsilon}}^{4-\alpha}(x)v_{\varepsilon}(x)\varphi(x)}{|x-y|^{\alpha}}dxdy \\
 &=\langle\Pi_{\xi_{\varepsilon},\lambda_{\varepsilon}}^{\bot} L_{\varepsilon}v_{\varepsilon},\varphi \rangle\leq \|\Pi_{\xi_{\varepsilon},\lambda_{\varepsilon}}^{\bot}L_{\varepsilon}v_{\varepsilon}\|_{H^{1}_{0}}\|\varphi\|_{H^{1}_{0}}=o(1)\|v_{\varepsilon}\|_{H^{1}_{0}}\|\varphi\|_{H^{1}_{0}}, \text{~as~}\varepsilon\to0.    
\end{aligned}  
\end{equation*}
In particular, taking $\varphi=v_{\varepsilon}$, we have
\begin{equation}\label{indetity for v eps}
    \begin{aligned}
        \int_\Omega |\nabla v_{\varepsilon}|^2 dx&-(6-\alpha)\int_{\Omega}\int_{\Omega}\frac{P\bar{U}_{\xi_{\varepsilon},\lambda_{\varepsilon}}^{5-\alpha}(y)v_{\varepsilon}(y)P\bar{U}_{\xi_{\varepsilon},\lambda_{\varepsilon}}^{5-\alpha}(x)v_{\varepsilon}(x)}{|x-y|^{\alpha}}dxdy\\
 &-(5-\alpha)\int_{\Omega}\int_{\Omega}\frac{P\bar{U}_{\xi_{\varepsilon},\lambda_{\varepsilon}}^{6-\alpha}(y)P\bar{U}_{\xi_{\varepsilon},\lambda_{\varepsilon}}^{4-\alpha}(x)v_{\varepsilon}^2(x)}{|x-y|^{\alpha}}dxdy \\
 &=o(1),\text{~as~}\varepsilon\to0. 
    \end{aligned}
\end{equation}
Next, we define 
\begin{equation*}
\tilde{v}_{\varepsilon}(x):=\lambda_{\varepsilon}^{-1/2}v_{\varepsilon}(\lambda_{\varepsilon}^{-1}x+\xi_{\varepsilon}),\quad \forall x\in \Omega_{\varepsilon}:=\{x\in\R^3:\lambda_{\varepsilon}^{-1}x+\xi_{\varepsilon}\in\Omega\},
\end{equation*}
and set $\tilde{v}_{\varepsilon}(x):=0$, if $x\in\R^{3}\setminus \Omega_{\varepsilon}$. Then $\int_{\R^3}|\nabla \tilde{v}_{\varepsilon} |^2dx=\int_{\Omega}|\nabla {v}_{\varepsilon}| ^2dx=1$. Up to a subsequence, we may assume that $\tilde{v}_{\varepsilon}\to v$ weakly in $\dot{H}^{1}(\R^3)$. Similar to the proof of \cite[Lemma 3.4]{Yang2023BlowUpBO}, we obtain that $v$ satisfies
\begin{equation*}
-\Delta v-(6-\alpha)\left(\int_{\R^3}\frac{\bar{U}_{0,1}^{5-\alpha}(y)v(y)}{|x-y|^{\alpha}}dy\right)\bar{U}_{0,1}^{5-\alpha}-(5-\alpha)\left(\int_{\R^3}\frac{\bar{U}_{0,1}^{6-\alpha}(y)}{|x-y|^{\alpha}}dy\right)\bar{U}_{0,1}^{4-\alpha}v=0.    
\end{equation*}
Then Theorem \ref{non-degeneracy property} yields that there exist constants $a_{0},a_{j}\in \R$, $j=1,2,3$ such that
\begin{equation}\label{form of v}
  v= a_{0}\frac{\partial\bar{U}_{0,\lambda}}{\partial\lambda}|_{\lambda=1}+\sum_{j=1}^{3}a_{j}\frac{\partial\bar{U}_{\xi,1}}{\partial \xi_{j}}|_{\xi=0}.
\end{equation}
Since $v_{\varepsilon}\in T_{\xi_{\varepsilon},\lambda_{\varepsilon}}^{\bot}$ and $\tilde{v}_{\varepsilon}(x)=0$ if $x\in\R^{3}\setminus \Omega_{\varepsilon}$, it holds that
\begin{equation*}
\begin{aligned}
    0=\int_{\Omega} \nabla \frac{\partial P\bar{U}_{\xi_{\varepsilon},\lambda_{\varepsilon}}}{\partial\lambda}\cdot \nabla v_{\varepsilon}dx&=\int_{\Omega} \nabla \frac{\partial \bar{U}_{\xi_{\varepsilon},\lambda_{\varepsilon}}}{\partial\lambda} \cdot\nabla v_{\varepsilon}dx\\
    &=\lambda_{\varepsilon}^{-1}\int_{\Omega_{\varepsilon}}\nabla \frac{\partial \bar{U}_{0,\lambda}}{\partial\lambda}|_{\lambda=1} \cdot\nabla \tilde{v}_{\varepsilon}dx=\lambda_{\varepsilon}^{-1}\int_{\R^3}\nabla \frac{\partial \bar{U}_{0,\lambda}}{\partial\lambda}|_{\lambda=1}\cdot \nabla \tilde{v}_{\varepsilon}dx,
\end{aligned}
\end{equation*}
and for $j=1,2,3$
\begin{equation*}
\begin{aligned}
   0=\int_{\Omega} \nabla \frac{\partial P\bar{U}_{\xi_{\varepsilon},\lambda_{\varepsilon}}}{\partial \xi_{j}} \cdot\nabla v_{\varepsilon}dx&=\int_{\Omega} \nabla \frac{\partial \bar{U}_{\xi_{\varepsilon},\lambda_{\varepsilon}}}{\partial \xi_{j}}\cdot \nabla v_{\varepsilon}dx\\
   &=\lambda_{\varepsilon}\int_{\Omega_{\varepsilon}} \nabla \frac{\partial \bar{U}_{\xi,1}}{\partial \xi_{j}}|_{\xi=0} \cdot\nabla \tilde{v}_{\varepsilon}dx=\lambda_{\varepsilon}\int_{\R^3} \nabla \frac{\partial \bar{U}_{\xi,1}}{\partial \xi_{j}}|_{\xi=0} \cdot\nabla \tilde{v}_{\varepsilon}dx. 
\end{aligned}  
\end{equation*}
Notice that $\tilde{v}_{\varepsilon}\to v$ weakly in $\dot{H}^{1}(\R^3)$. Letting $\varepsilon\to0$, we have 
\begin{equation}\label{orthogonality of v}
    \int_{\R^3}\nabla \frac{\partial \bar{U}_{0,\lambda}}{\partial\lambda}|_{\lambda=1}\cdot \nabla {v}dx=\int_{\R^3} \nabla \frac{\partial \bar{U}_{\xi,1}}{\partial \xi_{j}}|_{\xi=0}\cdot \nabla {v}dx=0,\quad\text{for~}j=1,2,3.
\end{equation}
Combining \eqref{form of v} and \eqref{orthogonality of v}, we conclude that $v=0$. Moreover, by \eqref{important-identity-1}, the HLS inequality, the H\"older inequality and the Sobolev embedding theorem, we obtain
\begin{equation*}
    \begin{aligned}
&-(6-\alpha)\int_{\Omega}\int_{\Omega}\frac{P\bar{U}_{\xi_{\varepsilon},\lambda_{\varepsilon}}^{5-\alpha}(y)v_{\varepsilon}(y)P\bar{U}_{\xi_{\varepsilon},\lambda_{\varepsilon}}^{5-\alpha}(x)v_{\varepsilon}(x)}{|x-y|^{\alpha}}dxdy\\
&\quad-(5-\alpha)\int_{\Omega}\int_{\Omega}\frac{P\bar{U}_{\xi_{\varepsilon},\lambda_{\varepsilon}}^{6-\alpha}(y)P\bar{U}_{\xi_{\varepsilon},\lambda_{\varepsilon}}^{4-\alpha}(x)v_{\varepsilon}^2(x)}{|x-y|^{\alpha}}dxdy \\
 &\lesssim\left(\int_{B(0,R)}\bar{U}_{0,1}^{\frac{6(5-\alpha)}{3-\alpha}}dx\right)^{\frac{3-\alpha}{3}}\left(\int_{B(0,R)}\tilde{v}_{\varepsilon}^2dx\right)+\left(\int_{\Omega\setminus B(\xi_{\varepsilon},\lambda_{\varepsilon}^{-1}R)}\bar{U}_{\xi_{\varepsilon},\lambda_{\varepsilon}}^{6}dx\right)^{\frac{5-\alpha}{3}}\|\tilde{v}_{\varepsilon}\|_{H^{1}_{0}}^{2}\\
 &\quad +\int_{B(0,R)}\bar{U}_{0,1}^{4}\tilde{v}_{\varepsilon}^{2}dx+\left(\int_{\Omega\setminus B(\xi_{\varepsilon},\lambda_{\varepsilon}^{-1}R)}\bar{U}_{\xi_{\varepsilon},\lambda_{\varepsilon}}^{6}dx\right)^{\frac{2}{3}}\|\tilde{v}_{\varepsilon}\|_{H^{1}_{0}}^{2}\\
 &=o(1),
    \end{aligned}
\end{equation*}
for any sufficiently large $R$ and any sufficiently small $\varepsilon$. This together with \eqref{indetity for v eps} implies that
\begin{equation*}
    \int_{\Omega}|\nabla v_{\varepsilon}|^2dx=o(1),
\end{equation*}
a contradiction to the assumption $\|v_{\varepsilon}\|_{H^{1}_{0}}=1$. Thus \eqref{coercivity proof step 1-1} holds.

\textbf{Step 2} In this step, we use a compactness argument to complete the proof. We first define
 \begin{equation*}
 \begin{aligned}
C_{\varepsilon}:=\inf_{v\in T_{\xi_{\varepsilon},\lambda_{\varepsilon}}^{\bot},\|\nabla v\|_{L^{2}}=1}\left\{1+\int_{\Omega}\bar{a}v^2dx\right.&-(6-\alpha)\int_{\Omega}\int_{\Omega}\frac{P\bar{U}_{\xi_{\varepsilon},\lambda_{\varepsilon}}^{5-\alpha}(y)v(y)P\bar{U}_{\xi_{\varepsilon},\lambda_{\varepsilon}}^{5-\alpha}(x)v(x)}{|x-y|^{\alpha}}dxdy\\
 &\left.-(5-\alpha)\int_{\Omega}\int_{\Omega}\frac{P\bar{U}_{\xi_{\varepsilon},\lambda_{\varepsilon}}^{6-\alpha}(y)P\bar{U}_{\xi_{\varepsilon},\lambda_{\varepsilon}}^{4-\alpha}(x)v^2(x)}{|x-y|^{\alpha}}dxdy\right\}.
 \end{aligned}  
 \end{equation*}
Then $C_{\varepsilon}$ is bounded from below. We first claim that $C_{\varepsilon}$ is attained if $C_{\varepsilon}<1$. Indeed, fix $\varepsilon$ and let $v_{n}\in  T_{\xi_{\varepsilon},\lambda_{\varepsilon}}^{\bot}$ be a minimizing sequence for $C_{\varepsilon}$ such that, up to a subsequence, $v_{n}\to v_{\varepsilon}$ weakly in $H^{1}_{0}(\Omega)$ as $n\to\infty$. Then $\|\nabla v_{\varepsilon}\|_{L^{2}}\leq 1$. By the H\"older inequality, Sobolev embedding theorem and Rellich compact embedding theorem, we have
\begin{equation*}
    \int_{\Omega}\bar{a}(v_{n}^2-v_{\varepsilon}^2)dx\lesssim \|v_{n}+v_{\varepsilon}\|_{L^{2}}\|v_{n}-v_{\varepsilon}\|_{L^{2}}\to 0,\text{~as~}n\to \infty.
\end{equation*}
Moreover, it follows from \eqref{important-identity-1} and the HLS inequality that, as $n\to \infty$
\begin{equation*}
    \begin{aligned}
        &\left|\int_{\Omega}\int_{\Omega}\frac{P\bar{U}_{\xi_{\varepsilon},\lambda_{\varepsilon}}^{6-\alpha}(y)P\bar{U}_{\xi_{\varepsilon},\lambda_{\varepsilon}}^{4-\alpha}(x)(v_{n}^2(x)-v_{\varepsilon}^{2}(x))}{|x-y|^{\alpha}}dxdy\right|\\
        &\lesssim \|P\bar{U}_{\xi_{\varepsilon},\lambda_{\varepsilon}}^{4}\|_{L^{3}}\|v_{n}+v_{\varepsilon}\|_{L^{6}}\|v_{n}-v_{\varepsilon}\|_{L^{2}}\to 0
    \end{aligned}
\end{equation*}
and
\begin{equation*}
\begin{aligned}
   &\left|\int_{\Omega}\int_{\Omega}\frac{P\bar{U}_{\xi_{\varepsilon},\lambda_{\varepsilon}}^{5-\alpha}(y)v_{n}(y)P\bar{U}_{\xi_{\varepsilon},\lambda_{\varepsilon}}^{5-\alpha}(x)v_{n}(x)}{|x-y|^{\alpha}}dxdy\right.\\
   &\quad-\left.\int_{\Omega}\int_{\Omega}\frac{P\bar{U}_{\xi_{\varepsilon},\lambda_{\varepsilon}}^{5-\alpha}(y)v_{\varepsilon}(y)P\bar{U}_{\xi_{\varepsilon},\lambda_{\varepsilon}}^{5-\alpha}(x)v_{\varepsilon}(x)}{|x-y|^{\alpha}}dxdy\right|\\
   &\leq \left|\int_{\Omega}\int_{\Omega}\frac{P\bar{U}_{\xi_{\varepsilon},\lambda_{\varepsilon}}^{5-\alpha}(y)(v_{n}(y)-v_{\varepsilon}(y))P\bar{U}_{\xi_{\varepsilon},\lambda_{\varepsilon}}^{5-\alpha}(x)v_{n}(x)}{|x-y|^{\alpha}}dxdy\right|\\
   &\quad +\left|\int_{\Omega}\int_{\Omega}\frac{P\bar{U}_{\xi_{\varepsilon},\lambda_{\varepsilon}}^{5-\alpha}(y)v_{\varepsilon}(y)P\bar{U}_{\xi_{\varepsilon},\lambda_{\varepsilon}}^{5-\alpha}(x)(v_{n}(x)-v_{\varepsilon}(x))}{|x-y|^{\alpha}}dxdy\right|\\
   &\lesssim \|\bar{U}_{\xi_{\varepsilon},\lambda_{\varepsilon}}^{5-\alpha}\|_{L^{\frac{6}{5-\alpha}}}\|v_{n}+v_{\varepsilon}\|_{L^{6}}\|\bar{U}_{\xi_{\varepsilon},\lambda_{\varepsilon}}^{5-\alpha}\|_{L^\frac{6}{3-\alpha}}\|v_{n}-v_{\varepsilon}\|_{L^{2}}\to 0.
\end{aligned}
\end{equation*}
Combining the estimates above, we conclude that
\begin{equation*}
   \begin{aligned}
C_{\varepsilon}=1+\int_{\Omega}a(x)v_{\varepsilon}^2dx&-(6-\alpha)\int_{\Omega}\int_{\Omega}\frac{P\bar{U}_{\xi_{\varepsilon},\lambda_{\varepsilon}}^{5-\alpha}(y)v_{\varepsilon}(y)P\bar{U}_{\xi_{\varepsilon},\lambda_{\varepsilon}}^{5-\alpha}(x)v_{\varepsilon}(x)}{|x-y|^{\alpha}}dxdy\\
 &-(5-\alpha)\int_{\Omega}\int_{\Omega}\frac{P\bar{U}_{\xi_{\varepsilon},\lambda_{\varepsilon}}^{6-\alpha}(y)P\bar{U}_{\xi_{\varepsilon},\lambda_{\varepsilon}}^{4-\alpha}(x)v_{\varepsilon}^2(x)}{|x-y|^{\alpha}}dxdy.
 \end{aligned}   
\end{equation*}
Next, we define
\begin{equation*}
\begin{aligned}
   F_{\xi_{\varepsilon},\lambda_{\varepsilon}}(v_{\varepsilon}):=\int_{\Omega}a(x)v_{\varepsilon}^2dx&-(6-\alpha)\int_{\Omega}\int_{\Omega}\frac{P\bar{U}_{\xi_{\varepsilon},\lambda_{\varepsilon}}^{5-\alpha}(y)v_{\varepsilon}(y)P\bar{U}_{\xi_{\varepsilon},\lambda_{\varepsilon}}^{5-\alpha}(x)v_{\varepsilon}(x)}{|x-y|^{\alpha}}dxdy\\
 &-(5-\alpha)\int_{\Omega}\int_{\Omega}\frac{P\bar{U}_{\xi_{\varepsilon},\lambda_{\varepsilon}}^{6-\alpha}(y)P\bar{U}_{\xi_{\varepsilon},\lambda_{\varepsilon}}^{4-\alpha}(x)v_{\varepsilon}^2(x)}{|x-y|^{\alpha}}dxdy. 
\end{aligned}  
\end{equation*}
Since $C_{\varepsilon}<1$, we have $v_{\varepsilon}\neq 0$ and the inequality
\begin{equation*}
(1-C_{\varepsilon})\int_{\Omega}|\nabla v_{\varepsilon}|^2dx+F_{\xi_{\varepsilon},\lambda_{\varepsilon}}(v_{\varepsilon})\leq (1-C_{\varepsilon})+F_{\xi_{\varepsilon},\lambda_{\varepsilon}}(v_{\varepsilon})=0,
\end{equation*}
holds. It then follows from the minimality of $C_{\varepsilon}$ that the previous inequality must be an equality and $\int_{\Omega}|\nabla v_{\varepsilon}|^2=1$. Otherwise, $\bar{v}_{\varepsilon}:=\frac{v_{\varepsilon}}{\|\nabla v_{\varepsilon}\|_{L^{2}}}$ yields a contradiction to the definition of $C_{\varepsilon}$. Therefore $C_{\varepsilon}$ is achieved by $v_{\varepsilon}$ if $C_{\varepsilon}<1$.

We now show that $\liminf_{\varepsilon\to0}C_{\varepsilon}>0$. Otherwise, there exists a sequence of minimizers $v_{\varepsilon}$ for $C_{\varepsilon}$ such that $C_{\varepsilon}\to L\leq 0$ and $v_{\varepsilon}\to v$ weakly in $H^{1}_{0}(\Omega)$ as $\varepsilon\to0$. Moreover, by Lagrange multiplier theorem, $v_{\varepsilon}$ satisfies
\begin{equation*}
    \begin{aligned}
        -(1-C_{\varepsilon})\Delta v_{\varepsilon}+\bar{a} v_{\varepsilon}&-(6-\alpha)\left(\int_{\Omega}\frac{P\bar{U}_{\xi_{\varepsilon},\lambda_{\varepsilon}}^{5-\alpha}(y)v_{\varepsilon}(y)}{|x-y|^{\alpha}}dy\right)P\bar{U}_{\xi_{\varepsilon},\lambda_{\varepsilon}}^{5-\alpha}(x)\\
        &-(5-\alpha)\left(\int_{\Omega}\frac{P\bar{U}_{\xi_{\varepsilon},\lambda_{\varepsilon}}^{6-\alpha}(y)}{|x-y|^{\alpha}}dy\right) P\bar{U}_{\xi_{\varepsilon},\lambda_{\varepsilon}}^{4-\alpha}(x)v_{\varepsilon}(x)=0.
    \end{aligned}
\end{equation*}
For any given $\varphi \in C_{c}^{\infty}(\Omega)$, by \eqref{important-identity-1}, the HLS inequality, the H\"older inequality, and Sobolev embedding theorem, we have
\begin{equation*}
\begin{aligned}
    &(6-\alpha)\int_{\Omega}\int_{\Omega}\frac{P\bar{U}_{\xi_{\varepsilon},\lambda_{\varepsilon}}^{5-\alpha}(y)v_{\varepsilon}(y)P\bar{U}_{\xi_{\varepsilon},\lambda_{\varepsilon}}^{5-\alpha}(x)\varphi(x)}{|x-y|^{\alpha}}dydx\\
    &\quad+(5-\alpha) \int_{\Omega}\int_{\Omega}\frac{P\bar{U}_{\xi_{\varepsilon},\lambda_{\varepsilon}}^{6-\alpha}(y)P\bar{U}_{\xi_{\varepsilon},\lambda_{\varepsilon}}^{4-\alpha}(x)v_{\varepsilon}(x)\varphi(x)}{|x-y|^{\alpha}}dydx\\
    &\lesssim \|\bar{U}_{\xi_{\varepsilon},\lambda_{\varepsilon}}^{5-\alpha}\|_{L^{\frac{6}{5-\alpha}}}\|v_{\varepsilon}\|_{L^{6}}\|\bar{U}_{\xi_{\varepsilon},\lambda_{\varepsilon}}^{5-\alpha}\|_{L^{\frac{6}{6-\alpha}}}\|\varphi\|_{L^{\infty}}+\|\bar{U}_{\xi_{\varepsilon},\lambda_{\varepsilon}}^{4}\|_{L^{\frac{6}{5}}}\|v_{\varepsilon}\|_{L^{6}}\|\varphi\|_{L^{\infty}}\\
    &\to 0,\text{as~}\varepsilon\to 0.
\end{aligned}
\end{equation*}
We then conclude that $v$ satisfies
\begin{equation*}
    -(1-L)\Delta v+\bar{a}v=0.
\end{equation*}
This together with the coercivity of $-\Delta +\bar{a}$ yields that $v=0$. In view of the compactness of the embedding of $H^1_{0}(\Omega)$ into $L^2(\Omega)$ and \eqref{coercivity proof step 1-1}, we get that for any $\varepsilon$ small enough
\begin{equation*}
\begin{aligned}
   C_{\varepsilon}=\int_\Omega |\nabla v_{\varepsilon}|^2 dx+o(1)&-(6-\alpha)\int_{\Omega}\int_{\Omega}\frac{P\bar{U}_{\xi_{\varepsilon},\lambda_{\varepsilon}}^{5-\alpha}(y)v_{\varepsilon}(y)P\bar{U}_{\xi_{\varepsilon},\lambda_{\varepsilon}}^{5-\alpha}(x)v_{\varepsilon}(x)}{|x-y|^{\alpha}}dxdy\\
 &-(5-\alpha)\int_{\Omega}\int_{\Omega}\frac{P\bar{U}_{\xi_{\varepsilon},\lambda_{\varepsilon}}^{6-\alpha}(y)P\bar{U}_{\xi_{\varepsilon},\lambda_{\varepsilon}}^{4-\alpha}(x)v_{\varepsilon}^2(x)}{|x-y|^{\alpha}}dxdy \\
 &\geq \rho-o(1)>0,
\end{aligned}
\end{equation*}   
which derives a contradiction to $C_{\varepsilon}\to L\leq 0$. This completes the proof.
\end{proof}

\begin{Lem}\label{boundw} As $\varepsilon\to 0$, it holds that
	\begin{equation*}
		\label{bound w subsec}
		\|\nabla w_{\varepsilon}\|_{L^{2}(\Omega)}=O((\lambda_{\varepsilon}d_{\varepsilon})^{-1/2}). 
	\end{equation*}
\end{Lem}

\begin{proof}
Since $u_{\varepsilon}=\mu_{\varepsilon}\left(P\bar{U}_{\xi_{\varepsilon},\lambda_{\varepsilon}}+w_{\varepsilon}\right)$ and
 \begin{equation*}
 \begin{cases}
     -\Delta P\bar{U}_{\xi_{\varepsilon},\lambda_{\varepsilon}}=\left(\int_{\R^{3}}\frac{\bar{U}_{\xi_{\varepsilon},\lambda_{\varepsilon}}^{6-\alpha}(y)}{|x-y|^{\alpha}}dy\right)\bar{U}_{\xi_{\varepsilon},\lambda_{\varepsilon}}^{5-\alpha}&\text{in}\ \Omega,\\
    \quad\ \  P\bar{U}_{\xi_{\varepsilon},\lambda_{\varepsilon}}=0&\text{on}\ \partial\Omega,
 \end{cases}
 \end{equation*}
from \eqref{equation for u varepsilon}, the reminder term $w_{\varepsilon}$ satisfies
\begin{equation}\label{equation w}
\begin{cases}
    \begin{aligned}
 &-\Delta w_{\varepsilon}+(\bar{a}-\varepsilon)\left(w_{\varepsilon}+P\bar{U}_{\xi_{\varepsilon},\lambda_{\varepsilon}}\right)\\
 &\quad=\mu_{\varepsilon}^{10-2\alpha}\left(\int_{\Omega}\frac{(P\bar{U}_{\xi_{\varepsilon},\lambda_{\varepsilon}}+w_{\varepsilon})^{6-\alpha}(y)}{|x-y|^{\alpha}}dy\right)(P\bar{U}_{\xi_{\varepsilon},\lambda_{\varepsilon}}+w_{\varepsilon})^{5-\alpha}\\
 &\qquad-\left(\int_{\R^{3}}\frac{\bar{U}_{\xi_{\varepsilon},\lambda_{\varepsilon}}^{6-\alpha}(y)}{|x-y|^{\alpha}}dy\right)\bar{U}_{\xi_{\varepsilon},\lambda_{\varepsilon}}^{5-\alpha}
\end{aligned}&\text{~in~}\Omega,\\
 w_{\varepsilon}=0&\text{~on~}\partial\Omega.
\end{cases}
\end{equation}
Integrating this equation against $w_{\varepsilon}$ and recalling 
\begin{equation}\label{esposito eq1-3}
 \int_{\Omega}\int_{\R^{3}}\frac{\bar{U}_{\xi_{\varepsilon},\lambda_{\varepsilon}}^{6-\alpha}(y)\bar{U}_{\xi_{\varepsilon},\lambda_{\varepsilon}}^{5-\alpha}(x)w_{\varepsilon}(x)}{|x-y|^{\alpha}}dydx=\int_{\Omega} \nabla P\bar{U}_{\xi_{\varepsilon},\lambda_{\varepsilon}}\cdot\nabla w_{\varepsilon}dx=0, 
\end{equation}
we obtain
\begin{equation}\label{esposito eq1} 
\begin{aligned}
 &\int_\Omega|\nabla w_{\varepsilon}|^2+(\bar{a}-\varepsilon)(w_{\varepsilon}^2+P\bar{U}_{\xi_{\varepsilon},\lambda_{\varepsilon}}w_{\varepsilon})dx\\
 &\quad=\mu_{\varepsilon}^{10-2\alpha}\int_{\Omega}\int_{\Omega}\frac{(P\bar{U}_{\xi_{\varepsilon},\lambda_{\varepsilon}}+w_{\varepsilon})^{6-\alpha}(y)(P\bar{U}_{\xi_{\varepsilon},\lambda_{\varepsilon}}+w_{\varepsilon})^{5-\alpha}(x)w_{\varepsilon}(x)}{|x-y|^{\alpha}}dydx.
\end{aligned}
\end{equation} 

We now turn to estimating the right-hand side of \eqref{esposito eq1}. From \eqref{esposito eq1-3} and Lemma \ref{lem inequality 2}, we find that
\begin{equation*}
    \begin{aligned}
    &\int_{\Omega}\int_{\Omega}\frac{(P\bar{U}_{\xi_{\varepsilon},\lambda_{\varepsilon}}+w_{\varepsilon})^{6-\alpha}(y)(P\bar{U}_{\xi_{\varepsilon},\lambda_{\varepsilon}}+w_{\varepsilon})^{5-\alpha}(x)w_{\varepsilon}(x)}{|x-y|^{\alpha}}dxdy\\
&=(6-\alpha)\int_{\Omega}\int_{\Omega}\frac{P\bar{U}_{\xi_{\varepsilon},\lambda_{\varepsilon}}^{5-\alpha}(y)w_{\varepsilon}(y)P\bar{U}_{\xi_{\varepsilon},\lambda_{\varepsilon}}^{5-\alpha}(x)w_{\varepsilon}(x)}{|x-y|^{\alpha}}dxdy\\
&\quad+(5-\alpha)\int_{\Omega}\int_{\Omega}\frac{P\bar{U}_{\xi_{\varepsilon},\lambda_{\varepsilon}}^{6-\alpha}(y)P\bar{U}_{\xi_{\varepsilon},\lambda_{\varepsilon}}^{4-\alpha}(x)w_{\varepsilon}^2(x)}{|x-y|^{\alpha}}dxdy\\
&\quad+O\left( \int_{\Omega}\int_{\R^{3}\setminus\Omega}\frac{P\bar{U}_{\xi_{\varepsilon},\lambda_{\varepsilon}}^{6-\alpha}(y)P\bar{U}_{\xi_{\varepsilon},\lambda_{\varepsilon}}^{5-\alpha}(x)|w_{\varepsilon}|(x)}{|x-y|^{\alpha}}dydx\right)\\
&\quad+O\left(\int_{\Omega}\int_{\Omega}\frac{PU_{\xi_{\varepsilon},\lambda_{\varepsilon}}^{6-\alpha}(y)(P\bar{U}_{\xi_{\varepsilon},\lambda_{\varepsilon}}^{3-\alpha}|w_{\varepsilon}|^{3}+|w_{\varepsilon}|^{6-\alpha})(x)}{|x-y|^{\alpha}}dxdy\right)\\
&\quad+O\left(\int_{\Omega}\int_{\Omega}\frac{(PU_{\xi_{\varepsilon},\lambda_{\varepsilon}}^{5-\alpha}|w_{\varepsilon}|)(y)(P\bar{U}_{\xi_{\varepsilon},\lambda_{\varepsilon}}^{4-\alpha}|w_{\varepsilon}|^{2}+P\bar{U}_{\xi_{\varepsilon},\lambda_{\varepsilon}}^{3-\alpha}|w_{\varepsilon}|^{3}+|w_{\varepsilon}|^{6-\alpha})(x)}{|x-y|^{\alpha}}dxdy\right)\\
&\quad+O\left(\int_{\Omega}\int_{\Omega}\frac{(PU_{\xi_{\varepsilon},\lambda_{\varepsilon}}^{4-\alpha}|w_{\varepsilon}|^{2}+|w_{\varepsilon}|^{6-\alpha})(y)(u_{\varepsilon}^{5-\alpha}|w_{\varepsilon}|)(x)}{|x-y|^{\alpha}}dxdy\right).
    \end{aligned}
\end{equation*}
It follows from \eqref{important-identity-1}, the HLS inequality, the H\"{o}lder inequality, Sobolev embedding theorem, Lemma \ref{lemma-1}, and the fact $\|w_{\varepsilon}\|_{H^{1}_{0}(\Omega)}\to0$ that
\begin{equation*}
    \begin{aligned}
        \int_{\Omega}\int_{\R^{3}\setminus\Omega}\frac{P\bar{U}_{\xi_{\varepsilon},\lambda_{\varepsilon}}^{6-\alpha}(y)P\bar{U}_{\xi_{\varepsilon},\lambda_{\varepsilon}}^{5-\alpha}(x)|w_{\varepsilon}|(x)}{|x-y|^{\alpha}}dydx&\lesssim \int_{\R^{3}\setminus\Omega}{U}_{\xi_{\varepsilon},\lambda_{\varepsilon}}^{5}|w_{\varepsilon}|dx\lesssim \|{U}_{\xi_{\varepsilon},\lambda_{\varepsilon}}\|_{L^{6}(\R^{3}\setminus\Omega)}^{5}\|w_{\varepsilon}\|_{L^{6}(\Omega)}\\
        &\lesssim (\lambda_{\varepsilon}d_{\varepsilon})^{-\frac{5}{2}}\|\nabla w_{\varepsilon}\|_{L^{2}(\Omega)},
    \end{aligned}
\end{equation*}
\begin{equation*}
\begin{aligned}
     \int_{\Omega}&\int_{\Omega}\frac{PU_{\xi_{\varepsilon},\lambda_{\varepsilon}}^{6-\alpha}(y)(P\bar{U}_{\xi_{\varepsilon},\lambda_{\varepsilon}}^{3-\alpha}|w_{\varepsilon}|^{3}+|w_{\varepsilon}|^{6-\alpha})(x)}{|x-y|^{\alpha}}dxdy\\
     &\lesssim \|U_{\xi_{\varepsilon},\lambda_{\varepsilon}}\|_{L^{6}(\Omega)}^{6-\alpha}(\|U_{\xi_{\varepsilon},\lambda_{\varepsilon}}\|_{L^{6}(\Omega)}^{3-\alpha}\|w_{\varepsilon}\|_{L^{6}(\Omega)}^{3}+\|w_{\varepsilon}\|_{L^{6}(\Omega)}^{6-\alpha})\\
     &\lesssim\|\nabla w_{\varepsilon}\|_{L^{2}(\Omega)}^{3},
\end{aligned}
\end{equation*}
\begin{equation*}
    \begin{aligned}
       &\int_{\Omega}\int_{\Omega}\frac{(PU_{\xi_{\varepsilon},\lambda_{\varepsilon}}^{5-\alpha}|w_{\varepsilon}|)(y)(P\bar{U}_{\xi_{\varepsilon},\lambda_{\varepsilon}}^{4-\alpha}|w_{\varepsilon}|^{2}+P\bar{U}_{\xi_{\varepsilon},\lambda_{\varepsilon}}^{3-\alpha}|w_{\varepsilon}|^{3}+|w_{\varepsilon}|^{6-\alpha})(x)}{|x-y|^{\alpha}}dxdy\\
       &\lesssim \|U_{\xi_{\varepsilon},\lambda_{\varepsilon}}\|_{L^{6}(\Omega)}^{5-\alpha}\|w_{\varepsilon}\|_{L^{6}(\Omega)}(\|U_{\xi_{\varepsilon},\lambda_{\varepsilon}}\|_{L^{6}(\Omega)}^{4-\alpha}\|w_{\varepsilon}\|_{L^{6}(\Omega)}^{2}+\|U_{\xi_{\varepsilon},\lambda_{\varepsilon}}\|_{L^{6}(\Omega)}^{3-\alpha}\|w_{\varepsilon}\|_{L^{6}(\Omega)}^{3}+\|w_{\varepsilon}\|_{L^{6}(\Omega)}^{6-\alpha})\\
     &\lesssim\|\nabla w_{\varepsilon}\|_{L^{2}(\Omega)}^{3},
    \end{aligned}
\end{equation*}
and
\begin{equation*} 
    \begin{aligned}
        \int_{\Omega}&\int_{\Omega}\frac{(PU_{\xi_{\varepsilon},\lambda_{\varepsilon}}^{4-\alpha}|w_{\varepsilon}|^{2}+|w_{\varepsilon}|^{6-\alpha})(y)(u_{\varepsilon}^{5-\alpha}|w_{\varepsilon}|)(x)}{|x-y|^{\alpha}}dxdy\\
        &\lesssim(\|U_{\xi_{\varepsilon},\lambda_{\varepsilon}}\|_{L^{6}(\Omega)}^{4-\alpha}\|w_{\varepsilon}\|^{2}_{L^{6}(\Omega)}+\|w_{\varepsilon}\|^{6-\alpha}_{L^{6}(\Omega)})\|u_{\varepsilon}\|_{L^{6}(\Omega)}^{5-\alpha}\|w_{\varepsilon}\|_{L^{6}(\Omega)}\\
        &\lesssim\|\nabla w_{\varepsilon}\|_{L^{2}(\Omega)}^{3}.
    \end{aligned}
\end{equation*}
Similarly, we can prove that
\begin{equation*}
    \begin{aligned}
        \int_{\Omega}&\int_{\Omega}\frac{P\bar{U}_{\xi_{\varepsilon},\lambda_{\varepsilon}}^{5-\alpha}(y)w_{\varepsilon}(y)P\bar{U}_{\xi_{\varepsilon},\lambda_{\varepsilon}}^{5-\alpha}(x)w_{\varepsilon}(x)}{|x-y|^{\alpha}}dxdy+\int_{\Omega}\int_{\Omega}\frac{P\bar{U}_{\xi_{\varepsilon},\lambda_{\varepsilon}}^{6-\alpha}(y)P\bar{U}_{\xi_{\varepsilon},\lambda_{\varepsilon}}^{4-\alpha}(x)w_{\varepsilon}^2(x)}{|x-y|^{\alpha}}dxdy\\
        &\lesssim \|\nabla w_{\varepsilon}\|_{L^{2}(\Omega)}^{2}.
    \end{aligned}
\end{equation*}
Since $\mu_{\varepsilon}\to 1$ as $\varepsilon\to0$, it follows that 
\begin{equation}\label{esposito eq1-450} 
    \begin{aligned}
        &\mu_{\varepsilon}^{10-2\alpha}\int_{\Omega}\int_{\Omega}\frac{(P\bar{U}_{\xi_{\varepsilon},\lambda_{\varepsilon}}+w_{\varepsilon})^{6-\alpha}(y)(P\bar{U}_{\xi_{\varepsilon},\lambda_{\varepsilon}}+w_{\varepsilon})^{5-\alpha}(x)w_{\varepsilon}(x)}{|x-y|^{\alpha}}dydx\\
        &=(6-\alpha)\int_{\Omega}\int_{\Omega}\frac{P\bar{U}_{\xi_{\varepsilon},\lambda_{\varepsilon}}^{5-\alpha}(y)w_{\varepsilon}(y)P\bar{U}_{\xi_{\varepsilon},\lambda_{\varepsilon}}^{5-\alpha}(x)w_{\varepsilon}(x)}{|x-y|^{\alpha}}dxdy\\
&\quad+(5-\alpha)\int_{\Omega}\int_{\Omega}\frac{P\bar{U}_{\xi_{\varepsilon},\lambda_{\varepsilon}}^{6-\alpha}(y)P\bar{U}_{\xi_{\varepsilon},\lambda_{\varepsilon}}^{4-\alpha}(x)w_{\varepsilon}^2(x)}{|x-y|^{\alpha}}dxdy+o(\|\nabla w_{\varepsilon}\|^{2}_{L^{2}(\Omega)}).
    \end{aligned}
\end{equation}

On the other hand, by the H\"{o}lder inequality and Sobolev embedding theorem, we have 
\begin{equation*} 
    \int_\Omega (\bar{a}-\varepsilon ) P\bar{U}_{\xi_{\varepsilon},\lambda_{\varepsilon}} w_{\varepsilon} dx\lesssim \|w_{\varepsilon}\|_{L^{6}(\Omega)} \|U_{\xi_{\varepsilon},\lambda_{\varepsilon}}\|_{L^{\frac{6}{5}}(\Omega)} \lesssim (\lambda_{\varepsilon}d_{\varepsilon})^{-1/2} \|\nabla w_{\varepsilon}\|_{L^{2}(\Omega)}
\end{equation*}
and
\begin{equation}\label{esposito eq1-452} 
    \varepsilon \int_\Omega  w_{\varepsilon}^2dx \lesssim \varepsilon \| w\|_{L^{6}(\Omega)}^2 = o(\|\nabla w_{\varepsilon}\|_{L^{2}(\Omega)}^2).
\end{equation}
Combining \eqref{esposito eq1}--\eqref{esposito eq1-452}, we obtain
\begin{equation*}
\begin{aligned}
 \int_\Omega |\nabla w_{\varepsilon}|^2 + \bar{a} w_{\varepsilon}^2 dx&=(6-\alpha)\int_{\Omega}\int_{\Omega}\frac{P\bar{U}_{\xi_{\varepsilon},\lambda_{\varepsilon}}^{5-\alpha}(y)w_{\varepsilon}(y)P\bar{U}_{\xi_{\varepsilon},\lambda_{\varepsilon}}^{5-\alpha}(x)w_{\varepsilon}(x)}{|x-y|^{\alpha}}dxdy\\
 &\quad+(5-\alpha)\int_{\Omega}\int_{\Omega}\frac{P\bar{U}_{\xi_{\varepsilon},\lambda_{\varepsilon}}^{6-\alpha}(y)P\bar{U}_{\xi_{\varepsilon},\lambda_{\varepsilon}}^{4-\alpha}(x)w_{\varepsilon}^2(x)}{|x-y|^{\alpha}}dxdy \\
 &={O}((\lambda_{\varepsilon}d_{\varepsilon})^{-1/2}\|\nabla w_{\varepsilon}\|_2)+o(\|\nabla w_{\varepsilon}\|_{2}^{2}).
\end{aligned}
\end{equation*}
This together with the coercivity inequality from Lemma \ref{lemma coercivity} yields that
\begin{equation*}
  \|\nabla w_{\varepsilon}\|_{L^{2}(\Omega)}=O((\lambda_{\varepsilon}d_{\varepsilon})^{-1/2}).
\end{equation*}
We complete the proof.
\end{proof}

\begin{Lem}\label{lemma-energy-1}
As $\varepsilon\to0$, it holds that
\begin{equation*}
    d_{\varepsilon}^{-1}=O(1).
\end{equation*}
\end{Lem}

\begin{proof}
Using the decomposition $u_{\varepsilon}=\mu_{\varepsilon}(P\bar{U}_{\xi_{\varepsilon},\lambda_{\varepsilon}}+w_{\varepsilon})$, we obtain
\begin{equation}\label{lemma-energy-1-proof-0}
    S_{HL}(\bar{a}-\varepsilon)=\frac{\int_{\Omega}|\nabla (P\bar{U}_{\xi_{\varepsilon},\lambda_{\varepsilon}}+w_{\varepsilon})|^{2}dx+\int_{\Omega}(\bar{a}-\varepsilon)(P\bar{U}_{\xi_{\varepsilon},\lambda_{\varepsilon}}+w_{\varepsilon})^{2}dx}{\left(\int_{\Omega}\int_{\Omega}\frac{(P\bar{U}_{\xi_{\varepsilon},\lambda_{\varepsilon}}+w_{\varepsilon})^{6-\alpha}(y)(P\bar{U}_{\xi_{\varepsilon},\lambda_{\varepsilon}}+w_{\varepsilon})^{6-\alpha}(x)}{|x-y|^{\alpha}}dydx\right)^{\frac{1}{6-\alpha}}}.
\end{equation}
We now estimate the numerator of $S_{HL}(\bar{a}-\varepsilon)$. First, it follows from the orthogonality that
\begin{equation}\label{lemma-energy-1-proof-1}
      \int_{\Omega}|\nabla (P\bar{U}_{\xi_{\varepsilon},\lambda_{\varepsilon}}+w_{\varepsilon})|^{2}dx=\int_{\Omega}|\nabla P\bar{U}_{\xi_{\varepsilon},\lambda_{\varepsilon}}|^{2}dx+\int_{\Omega}|\nabla w_{\varepsilon}|^{2}dx.
  \end{equation}  
From Lemma \ref{lemma PU}, we have
\begin{equation*}
\begin{aligned}
     P{U}_{\xi_{\varepsilon},\lambda_{\varepsilon}}&={U}_{\xi_{\varepsilon},\lambda_{\varepsilon}}-\varphi_{\xi_{\varepsilon},\lambda_{\varepsilon}}\\
     &={U}_{\xi_{\varepsilon},\lambda_{\varepsilon}}+4\pi\lambda_{\varepsilon}^{-\frac{1}{2}}H_{0}(\xi_{\varepsilon},\cdot)+O(\lambda_{\varepsilon}^{-\frac{5}{2}}d_{\varepsilon}^{-3}).
\end{aligned}
    \end{equation*}
It then follows that
\begin{equation}\label{lemma-energy-1-proof-3}
\begin{aligned}
    \int_{\Omega}|\nabla P\bar{U}_{\xi_{\varepsilon},\lambda_{\varepsilon}}|^{2}dx&=\int_{\Omega}P\bar{U}_{\xi_{\varepsilon},\lambda_{\varepsilon}}(-\Delta) P\bar{U}_{\xi_{\varepsilon},\lambda_{\varepsilon}} dx=3\bar{C}_{\alpha}^{2}\int_{\Omega}P{U}_{\xi_{\varepsilon},\lambda_{\varepsilon}} {U}^{5}_{\xi_{\varepsilon},\lambda_{\varepsilon}} dx\\
    &=3\bar{C}_{\alpha}^{2}\left\{\int_{\Omega}{U}^{6}_{\xi_{\varepsilon},\lambda_{\varepsilon}} dx+4\pi\lambda_{\varepsilon}^{-\frac{1}{2}}\int_{\Omega}H_{0}(\xi_{\varepsilon},\cdot){U}^{5}_{\xi_{\varepsilon},\lambda_{\varepsilon}} dx\right\}\\
    &\quad+O\left(\lambda_{\varepsilon}^{-\frac{5}{2}}d_{\varepsilon}^{-3}\int_{\Omega}{U}^{5}_{\xi_{\varepsilon},\lambda_{\varepsilon}} dx\right).
\end{aligned}
\end{equation}
Using \eqref{best sobolev constant}, \eqref{definition-of-U-bar} and Lemma \ref{lemma PU}, we have
\begin{equation*}
    \begin{aligned}
        3\bar{C}_{\alpha}^{2}\int_{\Omega}{U}^{6}_{\xi_{\varepsilon},\lambda_{\varepsilon}} dx=3\bar{C}_{\alpha}^{2}\int_{\R^{3}}{U}^{6}_{\xi_{\varepsilon},\lambda_{\varepsilon}} dx-3\bar{C}_{\alpha}^{2}\int_{\R^{3}\setminus\Omega}{U}^{6}_{\xi_{\varepsilon},\lambda_{\varepsilon}} dx=S_{HL}^{\frac{6-\alpha}{5-\alpha}}+O((\lambda_{\varepsilon}d_{\varepsilon})^{-3})
    \end{aligned}
\end{equation*}
and
\begin{equation*}
    \lambda_{\varepsilon}^{-\frac{5}{2}}d_{\varepsilon}^{-3}\int_{\Omega}{U}^{5}_{\xi_{\varepsilon},\lambda_{\varepsilon}} dx=O((\lambda_{\varepsilon}d_{\varepsilon})^{-3}).
\end{equation*}
Moreover, Taylor's expansion of $H_{0}(\xi_{\varepsilon},\cdot)$ gives that
\begin{equation}\label{lemma-energy-1-proof-6}
\begin{aligned}
     &\int_{\Omega}H_{0}(\xi_{\varepsilon},\cdot){U}^{5}_{\xi_{\varepsilon},\lambda_{\varepsilon}} dx\\
     &=\int_{B_{d_{\varepsilon}}(\xi_{\varepsilon})}H_{0}(\xi_{\varepsilon},\cdot){U}^{5}_{\xi_{\varepsilon},\lambda_{\varepsilon}} dx+\int_{\Omega\setminus B_{d_{\varepsilon}}(\xi_{\varepsilon})}H_{0}(\xi_{\varepsilon},\cdot){U}^{5}_{\xi_{\varepsilon},\lambda_{\varepsilon}} dx\\
     &=\phi_{0}(\xi_{\varepsilon})\int_{B_{d_{\varepsilon}}(\xi_{\varepsilon})}{U}^{5}_{\xi_{\varepsilon},\lambda_{\varepsilon}} dx+O\left(\|\nabla H_{0}(\xi_{\varepsilon},\cdot)\|_{L^{\infty}(B_{d_{\varepsilon}}(\xi_{\varepsilon}))}\int_{B_{d_{\varepsilon}}(\xi_{\varepsilon})}{U}^{5}_{\xi_{\varepsilon},\lambda_{\varepsilon}}|x-\xi_{\varepsilon}| dx\right)\\
     &\quad+\int_{\Omega\setminus B_{d_{\varepsilon}}(\xi_{\varepsilon})}H_{0}(\xi_{\varepsilon},\cdot){U}^{5}_{\xi_{\varepsilon},\lambda_{\varepsilon}} dx.
\end{aligned}
\end{equation}
By Lemma \ref{lem-H-0} and some direct computations, we have
\begin{equation*}
    \begin{aligned}
        &\lambda_{\varepsilon}^{-\frac{1}{2}}\phi_{0}(\xi_{\varepsilon})\int_{B_{d_{\varepsilon}}(\xi_{\varepsilon})}{U}^{5}_{\xi_{\varepsilon},\lambda_{\varepsilon}} dx\\
        &=\lambda_{\varepsilon}^{-\frac{1}{2}}\phi_{0}(\xi_{\varepsilon})\int_{\R^{3}}{U}^{5}_{\xi_{\varepsilon},\lambda_{\varepsilon}} dx-\lambda_{\varepsilon}^{-\frac{1}{2}}\phi_{0}(\xi_{\varepsilon})\int_{\R^{3}\setminus B_{d_{\varepsilon}}(\xi_{\varepsilon})}{U}^{5}_{\xi_{\varepsilon},\lambda_{\varepsilon}} dx\\
        &=\frac{4\pi}{3}\lambda_{\varepsilon}^{-1}\phi_{0}(\xi_{\varepsilon})+O(\lambda_{\varepsilon}^{-3}d_{\varepsilon}^{-3}),
    \end{aligned}
\end{equation*}
\begin{equation*}
    \begin{aligned}
        \lambda_{\varepsilon}^{-\frac{1}{2}}\|\nabla H_{0}(\xi_{\varepsilon},\cdot)\|_{L^{\infty}(B_{d_{\varepsilon}}(\xi_{\varepsilon}))}\int_{B_{d_{\varepsilon}}(\xi_{\varepsilon})}{U}^{5}_{\xi_{\varepsilon},\lambda_{\varepsilon}}|x-\xi_{\varepsilon}| dx=O((\lambda_{\varepsilon}d_{\varepsilon})^{-2}),
    \end{aligned}
\end{equation*}
and
\begin{equation}\label{lemma-energy-1-proof-8}
    \begin{aligned}
        \lambda_{\varepsilon}^{-\frac{1}{2}}\int_{\Omega\setminus B_{d_{\varepsilon}}(\xi_{\varepsilon})}H_{0}(\xi_{\varepsilon},\cdot){U}^{5}_{\xi_{\varepsilon},\lambda_{\varepsilon}} dx=O((\lambda_{\varepsilon}d_{\varepsilon})^{-3}).
    \end{aligned}
\end{equation}
Combining \eqref{lemma-energy-1-proof-3}--\eqref{lemma-energy-1-proof-8}, we obtain
\begin{equation}\label{lemma-energy-1-proof-9}
    \int_{\Omega}|\nabla P\bar{U}_{\xi_{\varepsilon},\lambda_{\varepsilon}}|^{2}dx=S_{HL}^{\frac{6-\alpha}{5-\alpha}}+16\pi^{2}\bar{C}_{\alpha}^{2}\lambda_{\varepsilon}^{-1}\phi_{0}(\xi_{\varepsilon})+O((\lambda_{\varepsilon}d_{\varepsilon})^{-2}).
\end{equation}
On the other hand, it holds that
\begin{equation*}
    \begin{aligned}
    \int_{\Omega}(\bar{a}-\varepsilon)(P\bar{U}_{\xi_{\varepsilon},\lambda_{\varepsilon}}+w_{\varepsilon})^{2}dx=\int_{\Omega}(\bar{a}-\varepsilon)(P\bar{U}_{\xi_{\varepsilon},\lambda_{\varepsilon}}^{2}+w_{\varepsilon}^{2}+2P\bar{U}_{\xi_{\varepsilon},\lambda_{\varepsilon}}w_{\varepsilon})dx.
    \end{aligned}
\end{equation*}
By the H\"{o}lder inequality and the Sobolev embedding theorem, we obtain
\begin{equation*}
    \begin{aligned}
        \int_{\Omega}(\bar{a}-\varepsilon)P\bar{U}_{\xi_{\varepsilon},\lambda_{\varepsilon}}^{2}dx\lesssim\int_{\Omega}\bar{U}_{\xi_{\varepsilon},\lambda_{\varepsilon}}^{2}dx=O(\lambda_{\varepsilon}^{-1}d_{\varepsilon}),
    \end{aligned}
\end{equation*}
\begin{equation*}
    \varepsilon\int_{\Omega}w_{\varepsilon}^{2}dx=o(\|\nabla w_{\varepsilon}\|_{L^{2}}^{2}),
\end{equation*}
and
\begin{equation}\label{lemma-energy-1-proof-12}
    \begin{aligned}
        2\int_{\Omega}(\bar{a}-\varepsilon)P\bar{U}_{\xi_{\varepsilon},\lambda_{\varepsilon}}w_{\varepsilon}dx&\lesssim \|\bar{U}_{\xi_{\varepsilon},\lambda_{\varepsilon}}\|_{L^{\frac{6}{5}}(\Omega)}\|w_{\varepsilon}\|_{L^{6}(\Omega)}\\
        &=O(\lambda_{\varepsilon}^{-\frac{1}{2}}d_{\varepsilon}^{\frac{3}{2}}\|\nabla w_{\varepsilon}\|_{L^{2}(\Omega)}).
    \end{aligned}
\end{equation}
Combining \eqref{lemma-energy-1-proof-1} and \eqref{lemma-energy-1-proof-9}--\eqref{lemma-energy-1-proof-12}, the numerator of $S_{HL}(\bar{a}-\varepsilon)$ satisfies
\begin{equation}\label{lemma-energy-1-proof-75}
\begin{aligned}
     &\int_{\Omega}|\nabla (P\bar{U}_{\xi_{\varepsilon},\lambda_{\varepsilon}}+w_{\varepsilon})|^{2}dx+\int_{\Omega}(\bar{a}-\varepsilon)(P\bar{U}_{\xi_{\varepsilon},\lambda_{\varepsilon}}+w_{\varepsilon})^{2}dx\\
     &=S_{HL}^{\frac{6-\alpha}{5-\alpha}}+16\pi^{2}\bar{C}_{\alpha}^{2}\lambda_{\varepsilon}^{-1}\phi_{0}(\xi_{\varepsilon})+\|\nabla w_{\varepsilon}\|_{L^{2}(\Omega)}^{2}+\int_{\Omega}\bar{a}w_{\varepsilon}^{2}dx\\
   &\quad+O(\lambda_{\varepsilon}^{-1}d_{\varepsilon}+\lambda_{\varepsilon}^{-\frac{1}{2}}d_{\varepsilon}^{\frac{3}{2}}\|\nabla w_{\varepsilon}\|_{L^{2}(\Omega)}+(\lambda_{\varepsilon}d_{\varepsilon})^{-2}).
\end{aligned}
\end{equation}

Next, we estimate the denominator of $S_{HL}(\bar{a}-\varepsilon)$. From the Taylor's formula, we have
\begin{equation*}
\begin{aligned}
    (P\bar{U}_{\xi_{\varepsilon},\lambda_{\varepsilon}}+w_{\varepsilon})^{6-\alpha}&=P\bar{U}_{\xi_{\varepsilon},\lambda_{\varepsilon}}^{6-\alpha}+(6-\alpha)P\bar{U}_{\xi_{\varepsilon},\lambda_{\varepsilon}}^{5-\alpha}w_{\varepsilon}+\frac{(6-\alpha)(5-\alpha)}{2}P\bar{U}_{\xi_{\varepsilon},\lambda_{\varepsilon}}^{4-\alpha}w_{\varepsilon}^{2}\\
    &\quad+O\left(P\bar{U}_{\xi_{\varepsilon},\lambda_{\varepsilon}}^{3-\alpha}|w_{\varepsilon}|^{3}+|w_{\varepsilon}|^{6-\alpha}\right).
\end{aligned}  
\end{equation*}
The HLS inequality, the H\"{o}lder inequality and the Sobolev embedding theorem yield that
\begin{equation}\label{lemma-energy-1-proof-77}
    \begin{aligned}
        &\int_{\Omega}\int_{\Omega}\frac{(P\bar{U}_{\xi_{\varepsilon},\lambda_{\varepsilon}}+w_{\varepsilon})^{6-\alpha}(x)(P\bar{U}_{\xi_{\varepsilon},\lambda_{\varepsilon}}+w_{\varepsilon})^{6-\alpha}(y)}{|x-y|^{\alpha}}dxdy\\
        &=\int_{\Omega}\int_{\Omega}\frac{P\bar{U}_{\xi_{\varepsilon},\lambda_{\varepsilon}}^{6-\alpha}(x)P\bar{U}_{\xi_{\varepsilon},\lambda_{\varepsilon}}^{6-\alpha}(y)}{|x-y|^{\alpha}}dxdy\\
        &\quad+2(6-\alpha)\int_{\Omega}\int_{\Omega}\frac{P\bar{U}_{\xi_{\varepsilon},\lambda_{\varepsilon}}^{6-\alpha}(x)P\bar{U}_{\xi_{\varepsilon},\lambda_{\varepsilon}}^{5-\alpha}(y)w_{\varepsilon}(y)}{|x-y|^{\alpha}}dxdy\\
         &\quad+(6-\alpha)^{2}\int_{\Omega}\int_{\Omega}\frac{P\bar{U}_{\xi_{\varepsilon},\lambda_{\varepsilon}}^{5-\alpha}(x)w_{\varepsilon}(x)P\bar{U}_{\xi_{\varepsilon},\lambda_{\varepsilon}}^{5-\alpha}(y)w_{\varepsilon}(y)}{|x-y|^{\alpha}}dxdy\\
        &\quad+(6-\alpha)(5-\alpha)\int_{\Omega}\int_{\Omega}\frac{P\bar{U}_{\xi_{\varepsilon},\lambda_{\varepsilon}}^{6-\alpha}(x)P\bar{U}_{\xi_{\varepsilon},\lambda_{\varepsilon}}^{4-\alpha}(y)w^{2}_{\varepsilon}(y)}{|x-y|^{\alpha}}dxdy+O(\|\nabla w_{\varepsilon}\|_{L^{2}(\Omega)}^{3}).
    \end{aligned}
\end{equation}
Notice that $P{U}_{\xi_{\varepsilon},\lambda_{\varepsilon}}={U}_{\xi_{\varepsilon},\lambda_{\varepsilon}}-\varphi_{\xi_{\varepsilon},\lambda_{\varepsilon}}$. Then
\begin{equation*}
\begin{aligned}
     P{U}_{\xi_{\varepsilon},\lambda_{\varepsilon}}^{6-\alpha}&=({U}_{\xi_{\varepsilon},\lambda_{\varepsilon}}-\varphi_{\xi_{\varepsilon},\lambda_{\varepsilon}})^{6-\alpha}\\
     &={U}_{\xi_{\varepsilon},\lambda_{\varepsilon}}^{6-\alpha}-(6-\alpha)U_{\xi_{\varepsilon},\lambda_{\varepsilon}}^{5-\alpha}\varphi_{\xi_{\varepsilon},\lambda_{\varepsilon}}+O(U_{\xi_{\varepsilon},\lambda_{\varepsilon}}^{4-\alpha}\varphi^{2}_{\xi_{\varepsilon},\lambda_{\varepsilon}})
\end{aligned}  
\end{equation*}
and
\begin{equation*}
    \begin{aligned}
        &\frac{1}{\bar{C}_{\alpha}^{2(6-\alpha)}}\int_{\Omega}\int_{\Omega}\frac{P\bar{U}_{\xi_{\varepsilon},\lambda_{\varepsilon}}^{6-\alpha}(x)P\bar{U}_{\xi_{\varepsilon},\lambda_{\varepsilon}}^{6-\alpha}(y)}{|x-y|^{\alpha}}dxdy\\
        &=\int_{\Omega}\int_{\Omega}\frac{{U}_{\xi_{\varepsilon},\lambda_{\varepsilon}}^{6-\alpha}(x){U}_{\xi_{\varepsilon},\lambda_{\varepsilon}}^{6-\alpha}(y)}{|x-y|^{\alpha}}dxdy-2(6-\alpha)\int_{\Omega}\int_{\Omega}\frac{{U}_{\xi_{\varepsilon},\lambda_{\varepsilon}}^{6-\alpha}(x){U}_{\xi_{\varepsilon},\lambda_{\varepsilon}}^{5-\alpha}(y)\varphi_{\xi_{\varepsilon},\lambda_{\varepsilon}}(y)}{|x-y|^{\alpha}}dxdy\\
        &\quad+O\left(\int_{\Omega}\int_{\Omega}\frac{{U}_{\xi_{\varepsilon},\lambda_{\varepsilon}}^{6-\alpha}(x){U}_{\xi_{\varepsilon},\lambda_{\varepsilon}}^{4-\alpha}(y)\varphi^{2}_{\xi_{\varepsilon},\lambda_{\varepsilon}}(y)}{|x-y|^{\alpha}}dxdy\right.\\
        &\quad\quad\left.+\int_{\Omega}\int_{\Omega}\frac{{U}_{\xi_{\varepsilon},\lambda_{\varepsilon}}^{5-\alpha}(x)\varphi_{\xi_{\varepsilon},\lambda_{\varepsilon}}(x){U}_{\xi_{\varepsilon},\lambda_{\varepsilon}}^{5-\alpha}(y)\varphi_{\xi_{\varepsilon},\lambda_{\varepsilon}}(y)}{|x-y|^{\alpha}}dxdy\right).
    \end{aligned}
\end{equation*}
By \eqref{important-identity-1}, \eqref{lemma-energy-1-proof-6}, Lemma \ref{lemma PU}, the HLS inequality and the H\"{o}lder inequality, we have
\begin{equation*}
    \begin{aligned}
        &\int_{\Omega}\int_{\Omega}\frac{{U}_{\xi_{\varepsilon},\lambda_{\varepsilon}}^{6-\alpha}(x){U}_{\xi_{\varepsilon},\lambda_{\varepsilon}}^{6-\alpha}(y)}{|x-y|^{\alpha}}dxdy\\
        &= \int_{\R^{3}}\int_{\Omega}\frac{{U}_{\xi_{\varepsilon},\lambda_{\varepsilon}}^{6-\alpha}(x){U}_{\xi_{\varepsilon},\lambda_{\varepsilon}}^{6-\alpha}(y)}{|x-y|^{\alpha}}dxdy- \int_{\R^{3}\setminus\Omega}\int_{\Omega}\frac{{U}_{\xi_{\varepsilon},\lambda_{\varepsilon}}^{6-\alpha}(x){U}_{\xi_{\varepsilon},\lambda_{\varepsilon}}^{6-\alpha}(y)}{|x-y|^{\alpha}}dxdy\\
        &=\frac{3}{\bar{C}_{\alpha}^{2(5-\alpha)}}\int_{\R^{3}}{U}_{\xi_{\varepsilon},\lambda_{\varepsilon}}^{6}dx+O(\|{U}_{\xi_{\varepsilon},\lambda_{\varepsilon}}\|^{6}_{L^{6}(\R^{3}\setminus\Omega)})\\
        &=\frac{3\pi^{2}}{4\bar{C}_{\alpha}^{2(5-\alpha)}}+O((\lambda_{\varepsilon}d_{\varepsilon})^{-3}),
    \end{aligned}
\end{equation*}
\begin{equation*}
    \begin{aligned}
        &\int_{\Omega}\int_{\Omega}\frac{{U}_{\xi_{\varepsilon},\lambda_{\varepsilon}}^{6-\alpha}(x){U}_{\xi_{\varepsilon},\lambda_{\varepsilon}}^{5-\alpha}(y)\varphi_{\xi_{\varepsilon},\lambda_{\varepsilon}}(y)}{|x-y|^{\alpha}}dxdy\\
        &=\int_{\Omega}\int_{\R^{3}}\frac{{U}_{\xi_{\varepsilon},\lambda_{\varepsilon}}^{6-\alpha}(x){U}_{\xi_{\varepsilon},\lambda_{\varepsilon}}^{5-\alpha}(y)\varphi_{\xi_{\varepsilon},\lambda_{\varepsilon}}(y)}{|x-y|^{\alpha}}dxdy-\int_{\Omega}\int_{\R^{3}\setminus\Omega}\frac{{U}_{\xi_{\varepsilon},\lambda_{\varepsilon}}^{6-\alpha}(x){U}_{\xi_{\varepsilon},\lambda_{\varepsilon}}^{5-\alpha}(y)\varphi_{\xi_{\varepsilon},\lambda_{\varepsilon}}(y)}{|x-y|^{\alpha}}dxdy\\
        &=\frac{3}{\bar{C}_{\alpha}^{2(5-\alpha)}}\int_{\Omega}{U}_{\xi_{\varepsilon},\lambda_{\varepsilon}}^{5}\varphi_{\xi_{\varepsilon},\lambda_{\varepsilon}}dy+O(\|{U}_{\xi_{\varepsilon},\lambda_{\varepsilon}}\|^{6-\alpha}_{L^{6}(\R^{3}\setminus\Omega)}\|{U}_{\xi_{\varepsilon},\lambda_{\varepsilon}}\|^{5-\alpha}_{L^{\frac{6(5-\alpha)}{6-\alpha}}(\Omega)}\|\varphi_{\xi_{\varepsilon},\lambda_{\varepsilon}}\|_{L^{\infty}(\Omega)})\\
        &=-\frac{12\pi}{\bar{C}_{\alpha}^{2(5-\alpha)}}\lambda_{\varepsilon}^{-\frac{1}{2}}\int_{\Omega}{U}_{\xi_{\varepsilon},\lambda_{\varepsilon}}^{5}H_{0}(\xi_{\varepsilon},\cdot)dy+O(\lambda_{\varepsilon}^{-\frac{5}{2}}d_{\varepsilon}^{-3}\|{U}_{\xi_{\varepsilon},\lambda_{\varepsilon}}\|^{5}_{L^{5}(\Omega)}+(\lambda_{\varepsilon}d_{\varepsilon})^{-\frac{8-\alpha}{2}})\\
         &=-\frac{16\pi^{2}}{\bar{C}_{\alpha}^{2(5-\alpha)}}\lambda_{\varepsilon}^{-1}\phi_{0}(\xi_{\varepsilon})+O((\lambda_{\varepsilon}d_{\varepsilon})^{-2}),
    \end{aligned}
\end{equation*}
and
\begin{equation*}
    \begin{aligned}
    &O\left(\int_{\Omega}\int_{\Omega}\frac{{U}_{\xi_{\varepsilon},\lambda_{\varepsilon}}^{6-\alpha}(x){U}_{\xi_{\varepsilon},\lambda_{\varepsilon}}^{4-\alpha}(y)\varphi^{2}_{\xi_{\varepsilon},\lambda_{\varepsilon}}(y)}{|x-y|^{\alpha}}dxdy\right.\\
        &\quad\quad\left.+\int_{\Omega}\int_{\Omega}\frac{{U}_{\xi_{\varepsilon},\lambda_{\varepsilon}}^{5-\alpha}(x)\varphi_{\xi_{\varepsilon},\lambda_{\varepsilon}}(x){U}_{\xi_{\varepsilon},\lambda_{\varepsilon}}^{5-\alpha}(y)\varphi_{\xi_{\varepsilon},\lambda_{\varepsilon}}(y)}{|x-y|^{\alpha}}dxdy\right)\\
        &=O\left(\|{U}_{\xi_{\varepsilon},\lambda_{\varepsilon}}\|_{L^{\frac{6(5-\alpha)}{6-\alpha}}(\Omega)}^{2(5-\alpha)}+\|{U}_{\xi_{\varepsilon},\lambda_{\varepsilon}}\|_{L^{4}(\Omega)}^{4}\right)\|\varphi_{\xi_{\varepsilon},\lambda_{\varepsilon}}\|_{L^{\infty}(\Omega)}^{2}=O((\lambda_{\varepsilon}d_{\varepsilon})^{-2}).
    \end{aligned}
\end{equation*}
It then follows that
\begin{equation}\label{lemma-energy-1-proof-83}
    \begin{aligned}
        \int_{\Omega}\int_{\Omega}\frac{P\bar{U}_{\xi_{\varepsilon},\lambda_{\varepsilon}}^{6-\alpha}(x)P\bar{U}_{\xi_{\varepsilon},\lambda_{\varepsilon}}^{6-\alpha}(y)}{|x-y|^{\alpha}}dxdy=S_{HL}^{\frac{6-\alpha}{5-\alpha}}+32(6-\alpha)\pi^{2}\bar{C}_{\alpha}^{2}\lambda_{\varepsilon}^{-1}\phi_{0}(\xi_{\varepsilon})+O((\lambda_{\varepsilon}d_{\varepsilon})^{-2}).
    \end{aligned}
\end{equation}
On the other hand, it follows from Lemma \ref{lem inequality 2} that
\begin{equation*}
    \begin{aligned}
        &\int_{\Omega}\int_{\Omega}\frac{P\bar{U}_{\xi_{\varepsilon},\lambda_{\varepsilon}}^{6-\alpha}(x)P\bar{U}_{\xi_{\varepsilon},\lambda_{\varepsilon}}^{5-\alpha}(y)w_{\varepsilon}(y)}{|x-y|^{\alpha}}dxdy\\
        &=\bar{C}_{\alpha}^{11-2\alpha}\int_{\Omega}\int_{\Omega}\frac{({U}_{\xi_{\varepsilon},\lambda_{\varepsilon}}-\varphi_{\xi_{\varepsilon},\lambda_{\varepsilon}})^{6-\alpha}(x)({U}_{\xi_{\varepsilon},\lambda_{\varepsilon}}-\varphi_{\xi_{\varepsilon},\lambda_{\varepsilon}})^{5-\alpha}(y)w_{\varepsilon}(y)}{|x-y|^{\alpha}}dxdy\\
        &=\bar{C}_{\alpha}^{11-2\alpha}\int_{\Omega}\int_{\Omega}\frac{{U}_{\xi_{\varepsilon},\lambda_{\varepsilon}}^{6-\alpha}(x){U}_{\xi_{\varepsilon},\lambda_{\varepsilon}}^{5-\alpha}(y)w_{\varepsilon}(y)}{|x-y|^{\alpha}}dxdy\\
        &\quad+O\left(\int_{\Omega}\int_{\Omega}\frac{{U}_{\xi_{\varepsilon},\lambda_{\varepsilon}}^{6-\alpha}(x){U}_{\xi_{\varepsilon},\lambda_{\varepsilon}}^{4-\alpha}(y)\varphi_{\xi_{\varepsilon},\lambda_{\varepsilon}}(y)|w_{\varepsilon}|(y)}{|x-y|^{\alpha}}dxdy\right.\\
        &\left.\qquad+\int_{\Omega}\int_{\Omega}\frac{{U}_{\xi_{\varepsilon},\lambda_{\varepsilon}}^{5-\alpha}(x)\varphi_{\xi_{\varepsilon},\lambda_{\varepsilon}}(x){U}_{\xi_{\varepsilon},\lambda_{\varepsilon}}^{5-\alpha}(y)|w_{\varepsilon}|(y)}{|x-y|^{\alpha}}dxdy\right).
    \end{aligned}
\end{equation*}
By \eqref{important-identity-1}, Lemma \ref{lemma PU}, the HLS inequality, the H\"{o}lder inequality and the orthogonality, we obtain
\begin{equation*}
    \begin{aligned}
        &\int_{\Omega}\int_{\Omega}\frac{{U}_{\xi_{\varepsilon},\lambda_{\varepsilon}}^{6-\alpha}(x){U}_{\xi_{\varepsilon},\lambda_{\varepsilon}}^{5-\alpha}(y)w_{\varepsilon}(y)}{|x-y|^{\alpha}}dxdy\\
        &=\int_{\Omega}\int_{\R^{3}}\frac{{U}_{\xi_{\varepsilon},\lambda_{\varepsilon}}^{6-\alpha}(x){U}_{\xi_{\varepsilon},\lambda_{\varepsilon}}^{5-\alpha}(y)w_{\varepsilon}(y)}{|x-y|^{\alpha}}dxdy-\int_{\Omega}\int_{\R^{3}\setminus\Omega}\frac{{U}_{\xi_{\varepsilon},\lambda_{\varepsilon}}^{6-\alpha}(x){U}_{\xi_{\varepsilon},\lambda_{\varepsilon}}^{5-\alpha}(y)w_{\varepsilon}(y)}{|x-y|^{\alpha}}dxdy\\
        &=\frac{3}{\bar{C}_{\alpha}^{2(5-\alpha)}}\int_{\Omega}{U}_{\xi_{\varepsilon},\lambda_{\varepsilon}}^{5}(y)w_{\varepsilon}(y)dy+O(\|{U}_{\xi_{\varepsilon},\lambda_{\varepsilon}}\|_{L^{6}(\R^{3}\setminus\Omega)}^{6-\alpha}\|w_{\varepsilon}\|_{L^{6}(\Omega)})\\
        &=O((\lambda_{\varepsilon}d_{\varepsilon})^{-\frac{6-\alpha}{2}}\|\nabla w_{\varepsilon}\|_{L^{2}(\Omega)})
    \end{aligned}
\end{equation*}
and
\begin{equation*}
    \begin{aligned}
        &O\left(\int_{\Omega}\int_{\Omega}\frac{{U}_{\xi_{\varepsilon},\lambda_{\varepsilon}}^{6-\alpha}(x){U}_{\xi_{\varepsilon},\lambda_{\varepsilon}}^{4-\alpha}(y)\varphi_{\xi_{\varepsilon},\lambda_{\varepsilon}}(y)|w_{\varepsilon}|(y)}{|x-y|^{\alpha}}dxdy\right.\\
        &\left.\qquad+\int_{\Omega}\int_{\Omega}\frac{{U}_{\xi_{\varepsilon},\lambda_{\varepsilon}}^{5-\alpha}(x)\varphi_{\xi_{\varepsilon},\lambda_{\varepsilon}}(x){U}_{\xi_{\varepsilon},\lambda_{\varepsilon}}^{5-\alpha}(y)|w_{\varepsilon}|(y)}{|x-y|^{\alpha}}dxdy\right)\\
        &=\|\varphi_{\xi_{\varepsilon},\lambda_{\varepsilon}}\|_{L^{\infty}(\Omega)}O\left(\int_{\Omega}{U}_{\xi_{\varepsilon},\lambda_{\varepsilon}}^{4}(y)|w_{\varepsilon}|(y)dy+\|{U}_{\xi_{\varepsilon},\lambda_{\varepsilon}}\|_{L^{\frac{6(5-\alpha)}{(6-\alpha)}}(\Omega)}^{5-\alpha}\|w_{\varepsilon}\|_{L^{6}(\Omega)}\right)\\
        &=\|\varphi_{\xi_{\varepsilon},\lambda_{\varepsilon}}\|_{L^{\infty}(\Omega)}\|w_{\varepsilon}\|_{L^{6}(\Omega)}O\left(\|{U}_{\xi_{\varepsilon},\lambda_{\varepsilon}}\|_{L^{\frac{24}{5}}(\Omega)}^{4}+\|{U}_{\xi_{\varepsilon},\lambda_{\varepsilon}}\|_{L^{\frac{6(5-\alpha)}{(6-\alpha)}}(\Omega)}^{5-\alpha}\right)\\
         &=O((\lambda_{\varepsilon}d_{\varepsilon})^{-1}\|\nabla w_{\varepsilon}\|_{L^{2}(\Omega)}).
    \end{aligned}
\end{equation*}
Then it holds that
\begin{equation*}
    \begin{aligned}
        \int_{\Omega}\int_{\Omega}\frac{P\bar{U}_{\xi_{\varepsilon},\lambda_{\varepsilon}}^{6-\alpha}(x)P\bar{U}_{\xi_{\varepsilon},\lambda_{\varepsilon}}^{5-\alpha}(y)w_{\varepsilon}(y)}{|x-y|^{\alpha}}dxdy=O((\lambda_{\varepsilon}d_{\varepsilon})^{-1}\|\nabla w_{\varepsilon}\|_{L^{2}(\Omega)}).
    \end{aligned}
\end{equation*}
This together with \eqref{lemma-energy-1-proof-77} and \eqref{lemma-energy-1-proof-83} gives that
\begin{equation*}
    \begin{aligned}
        &\int_{\Omega}\int_{\Omega}\frac{(P\bar{U}_{\xi_{\varepsilon},\lambda_{\varepsilon}}+w_{\varepsilon})^{6-\alpha}(x)(P\bar{U}_{\xi_{\varepsilon},\lambda_{\varepsilon}}+w_{\varepsilon})^{6-\alpha}(y)}{|x-y|^{\alpha}}dxdy\\
        &=S_{HL}^{\frac{6-\alpha}{5-\alpha}}+32(6-\alpha)\pi^{2}\bar{C}_{\alpha}^{2}\lambda_{\varepsilon}^{-1}\phi_{0}(\xi)\\
        &\quad+(6-\alpha)^{2}\int_{\Omega}\int_{\Omega}\frac{P\bar{U}_{\xi_{\varepsilon},\lambda_{\varepsilon}}^{5-\alpha}(x)w_{\varepsilon}(x)P\bar{U}_{\xi_{\varepsilon},\lambda_{\varepsilon}}^{5-\alpha}(y)w_{\varepsilon}(y)}{|x-y|^{\alpha}}dxdy\\
        &\quad+(6-\alpha)(5-\alpha)\int_{\Omega}\int_{\Omega}\frac{P\bar{U}_{\xi_{\varepsilon},\lambda_{\varepsilon}}^{6-\alpha}(x)P\bar{U}_{\xi_{\varepsilon},\lambda_{\varepsilon}}^{4-\alpha}(y)w^{2}_{\varepsilon}(y)}{|x-y|^{\alpha}}dxdy\\
        &\quad+O((\lambda_{\varepsilon}d_{\varepsilon})^{-2}+(\lambda_{\varepsilon}d_{\varepsilon})^{-1}\|\nabla w_{\varepsilon}\|_{L^{2}(\Omega)}+\|\nabla w_{\varepsilon}\|_{L^{2}(\Omega)}^{3}).
    \end{aligned}
\end{equation*}
Using the Taylor's formula and Lemma \ref{boundw}, we find that the denominator of $S_{HL}(\bar{a}-\varepsilon)$ satisfies
\begin{equation}\label{lemma-energy-1-proof-89}
    \begin{aligned}
        &\left(\int_{\Omega}\int_{\Omega}\frac{(P\bar{U}_{\xi_{\varepsilon},\lambda_{\varepsilon}}+w_{\varepsilon})^{6-\alpha}(x)(P\bar{U}_{\xi_{\varepsilon},\lambda_{\varepsilon}}+w_{\varepsilon})^{6-\alpha}(y)}{|x-y|^{\alpha}}dxdy\right)^{-\frac{1}{6-\alpha}}\\
        &=S_{HL}^{-\frac{1}{5-\alpha}}-S_{HL}^{-\frac{7-\alpha}{5-\alpha}}32\pi^{2}\bar{C}_{\alpha}^{2}\lambda_{\varepsilon}^{-1}\phi_{0}(\xi_{\varepsilon})\\
        &\quad-S_{HL}^{-\frac{7-\alpha}{5-\alpha}}(6-\alpha)\int_{\Omega}\int_{\Omega}\frac{P\bar{U}_{\xi_{\varepsilon},\lambda_{\varepsilon}}^{5-\alpha}(x)w_{\varepsilon}(x)P\bar{U}_{\xi_{\varepsilon},\lambda_{\varepsilon}}^{5-\alpha}(y)w_{\varepsilon}(y)}{|x-y|^{\alpha}}dxdy\\
        &\quad-S_{HL}^{-\frac{7-\alpha}{5-\alpha}}(5-\alpha)\int_{\Omega}\int_{\Omega}\frac{P\bar{U}_{\xi_{\varepsilon},\lambda_{\varepsilon}}^{6-\alpha}(x)P\bar{U}_{\xi_{\varepsilon},\lambda_{\varepsilon}}^{4-\alpha}(y)w^{2}_{\varepsilon}(y)}{|x-y|^{\alpha}}dxdy\\
        &\quad+o((\lambda_{\varepsilon}d_{\varepsilon})^{-1}).
    \end{aligned}
\end{equation}
Suppose that $d_{\varepsilon}\to0$ as $\varepsilon\to0$. Then Lemma \ref{lem-H-0} gives that
\begin{equation}\label{lemma-energy-1-proof-90}
    \phi_{0}(\xi_{\varepsilon})=-\frac{1}{8\pi d_{\varepsilon}}(1+O(d_{\varepsilon})). 
\end{equation}
Moreover, the HLS inequality and the Sobolev embedding theorem imply that
\begin{equation*}
    \begin{aligned}
        &(6-\alpha)\int_{\Omega}\int_{\Omega}\frac{P\bar{U}_{\xi_{\varepsilon},\lambda_{\varepsilon}}^{5-\alpha}(x)w_{\varepsilon}(x)P\bar{U}_{\xi_{\varepsilon},\lambda_{\varepsilon}}^{5-\alpha}(y)w_{\varepsilon}(y)}{|x-y|^{\alpha}}dxdy\\
        &\quad+(5-\alpha)\int_{\Omega}\int_{\Omega}\frac{P\bar{U}_{\xi_{\varepsilon},\lambda_{\varepsilon}}^{6-\alpha}(x)P\bar{U}_{\xi_{\varepsilon},\lambda_{\varepsilon}}^{4-\alpha}(y)w^{2}_{\varepsilon}(y)}{|x-y|^{\alpha}}dxdy\\
        &=O(\|\nabla w_{\varepsilon}\|_{L^{2}}^{2})=O((\lambda_{\varepsilon}d_{\varepsilon})^{-1}).
    \end{aligned}
\end{equation*}
Combining \eqref{lemma-energy-1-proof-0}, \eqref{lemma-energy-1-proof-75}, \eqref{lemma-energy-1-proof-89} and Lemma \ref{boundw}, we have
\begin{equation*}
    \begin{aligned}
        S_{HL}(\bar{a}-\varepsilon)&=S_{HL}-S_{HL}^{-\frac{1}{5-\alpha}}16\pi^{2}\bar{C}_{\alpha}^{2}\lambda_{\varepsilon}^{-1}\phi_{0}(\xi_{\varepsilon})\\
         &\quad-S_{HL}^{-\frac{1}{5-\alpha}}(6-\alpha)\int_{\Omega}\int_{\Omega}\frac{P\bar{U}_{\xi_{\varepsilon},\lambda_{\varepsilon}}^{5-\alpha}(x)w_{\varepsilon}(x)P\bar{U}_{\xi_{\varepsilon},\lambda_{\varepsilon}}^{5-\alpha}(y)w_{\varepsilon}(y)}{|x-y|^{\alpha}}dxdy\\
        &\quad-S_{HL}^{-\frac{1}{5-\alpha}}(5-\alpha)\int_{\Omega}\int_{\Omega}\frac{P\bar{U}_{\xi_{\varepsilon},\lambda_{\varepsilon}}^{6-\alpha}(x)P\bar{U}_{\xi_{\varepsilon},\lambda_{\varepsilon}}^{4-\alpha}(y)w^{2}_{\varepsilon}(y)}{|x-y|^{\alpha}}dxdy\\
        &\quad+S_{HL}^{-\frac{1}{5-\alpha}}\left(\|\nabla w_{\varepsilon}\|_{L^{2}(\Omega)}^{2}+\int_{\Omega}\bar{a}w_{\varepsilon}^{2}dx\right)\\
       &\quad+o((\lambda_{\varepsilon}d_{\varepsilon})^{-1}).
    \end{aligned}
\end{equation*}
This together with Lemma \ref{lemma coercivity} and \eqref{lemma-energy-1-proof-90} yields that
\begin{equation*}
    \begin{aligned}
        S_{HL}>S_{HL}(\bar{a}-\varepsilon)&\geq S_{HL}-S_{HL}^{-\frac{1}{5-\alpha}}16\pi^{2}\bar{C}_{\alpha}^{2}\lambda_{\varepsilon}^{-1}\phi_{0}(\xi_{\varepsilon})+o((\lambda_{\varepsilon}d_{\varepsilon})^{-1})\\
        &=S_{HL}+S_{HL}^{-\frac{1}{5-\alpha}}2\pi\bar{C}_{\alpha}^{2}(\lambda_{\varepsilon}d_{\varepsilon})^{-1}+o((\lambda_{\varepsilon}d_{\varepsilon})^{-1}).\\
    \end{aligned}
\end{equation*}
This leads to a contradiction. Therefore, we have $d_{\varepsilon}^{-1}=O(1)$.
\end{proof}

\begin{Prop}\label{prop-energy-expansion}
As $\varepsilon\to 0$, it holds that
    \begin{equation*}
   S_{HL}(\bar{a}-\varepsilon)=S_{HL}-\frac{64}{3}S_{HL}\lambda_{\varepsilon}^{-1}\phi_{\bar{a}}(\xi_{0})+o(\lambda_{\varepsilon}^{-1}).
\end{equation*}
\end{Prop}

\begin{proof}
Since $d_{\varepsilon}^{-1}=O(1)$, we have $\|w_{\varepsilon}\|_{H^{1}_{0}(\Omega)}=O(\lambda_{\varepsilon}^{-\frac{1}{2}})$. It then follows from \eqref{lemma-energy-1-proof-9}--\eqref{lemma-energy-1-proof-12} that the numerator of $S_{HL}(\bar{a}-\varepsilon)$ satisfies
\begin{equation}\label{prop-energy-expansion-proof-1}
    \begin{aligned}
        \int_{\Omega}|\nabla P\bar{U}_{\xi_{\varepsilon},\lambda_{\varepsilon}}|^{2}dx+\int_{\Omega}|\nabla w_{\varepsilon}|^{2}dx+\int_{\Omega}\bar{a}(P\bar{U}_{\xi_{\varepsilon},\lambda_{\varepsilon}}^{2}+2P\bar{U}_{\xi_{\varepsilon},\lambda_{\varepsilon}}w_{\varepsilon}+w_{\varepsilon}^{2})dx+o(\lambda_{\varepsilon}^{-1}).
    \end{aligned}
\end{equation}
On the other hand, \eqref{esposito eq1} and \eqref{esposito eq1-450}--\eqref{esposito eq1-452} yield that
\begin{equation}\label{prop-energy-expansion-proof-2}
    \begin{aligned}
      \int_\Omega |\nabla w_{\varepsilon}|^2 + \bar{a} (w_{\varepsilon}^2+P\bar{U}_{\xi_{\varepsilon},\lambda_{\varepsilon}}w_{\varepsilon}) dx&=(6-\alpha)\int_{\Omega}\int_{\Omega}\frac{P\bar{U}_{\xi_{\varepsilon},\lambda_{\varepsilon}}^{5-\alpha}(y)w_{\varepsilon}(y)P\bar{U}_{\xi_{\varepsilon},\lambda_{\varepsilon}}^{5-\alpha}(x)w_{\varepsilon}(x)}{|x-y|^{\alpha}}dxdy\\
 &\quad+(5-\alpha)\int_{\Omega}\int_{\Omega}\frac{P\bar{U}_{\xi_{\varepsilon},\lambda_{\varepsilon}}^{6-\alpha}(y)P\bar{U}_{\xi_{\varepsilon},\lambda_{\varepsilon}}^{4-\alpha}(x)w_{\varepsilon}^2(x)}{|x-y|^{\alpha}}dxdy \\
 &\quad+o(\lambda_{\varepsilon}^{-1}).  
    \end{aligned}
\end{equation}
Combining \eqref{lemma-energy-1-proof-9}, \eqref{prop-energy-expansion-proof-1} and \eqref{prop-energy-expansion-proof-2}, the numerator of $S_{HL}(\bar{a}-\varepsilon)$ becomes
\begin{equation*}
    \begin{aligned}
        S_{HL}^{\frac{6-\alpha}{5-\alpha}}+16\pi^{2}\bar{C}_{\alpha}^{2}\lambda_{\varepsilon}^{-1}\phi_{0}(\xi_{\varepsilon})&+\int_{\Omega}\bar{a}(P\bar{U}_{\xi_{\varepsilon},\lambda_{\varepsilon}}^{2}+P\bar{U}_{\xi_{\varepsilon},\lambda_{\varepsilon}}w_{\varepsilon})dx\\
        &+(6-\alpha)\int_{\Omega}\int_{\Omega}\frac{P\bar{U}_{\xi_{\varepsilon},\lambda_{\varepsilon}}^{5-\alpha}(y)w_{\varepsilon}(y)P\bar{U}_{\xi_{\varepsilon},\lambda_{\varepsilon}}^{5-\alpha}(x)w_{\varepsilon}(x)}{|x-y|^{\alpha}}dxdy\\
 &+(5-\alpha)\int_{\Omega}\int_{\Omega}\frac{P\bar{U}_{\xi_{\varepsilon},\lambda_{\varepsilon}}^{6-\alpha}(y)P\bar{U}_{\xi_{\varepsilon},\lambda_{\varepsilon}}^{4-\alpha}(x)w_{\varepsilon}^2(x)}{|x-y|^{\alpha}}dxdy \\
 &+o(\lambda_{\varepsilon}^{-1}).  
    \end{aligned}
\end{equation*}
Recall that the denominator of $S_{HL}(\bar{a}-\varepsilon)$ (see \eqref{lemma-energy-1-proof-89}) satisfies
\begin{equation*}
    \begin{aligned}
        &\left(\int_{\Omega}\int_{\Omega}\frac{(P\bar{U}_{\xi_{\varepsilon},\lambda_{\varepsilon}}+w_{\varepsilon})^{6-\alpha}(x)(P\bar{U}_{\xi_{\varepsilon},\lambda_{\varepsilon}}+w_{\varepsilon})^{6-\alpha}(y)}{|x-y|^{\alpha}}dxdy\right)^{-\frac{1}{6-\alpha}}\\
        &=S_{HL}^{-\frac{1}{5-\alpha}}-S_{HL}^{-\frac{7-\alpha}{5-\alpha}}32\pi^{2}\bar{C}_{\alpha}^{2}\lambda_{\varepsilon}^{-1}\phi_{0}(\xi_{\varepsilon})\\
        &\quad-S_{HL}^{-\frac{7-\alpha}{5-\alpha}}(6-\alpha)\int_{\Omega}\int_{\Omega}\frac{P\bar{U}_{\xi_{\varepsilon},\lambda_{\varepsilon}}^{5-\alpha}(x)w_{\varepsilon}(x)P\bar{U}_{\xi_{\varepsilon},\lambda_{\varepsilon}}^{5-\alpha}(y)w_{\varepsilon}(y)}{|x-y|^{\alpha}}dxdy\\
        &\quad-S_{HL}^{-\frac{7-\alpha}{5-\alpha}}(5-\alpha)\int_{\Omega}\int_{\Omega}\frac{P\bar{U}_{\xi_{\varepsilon},\lambda_{\varepsilon}}^{6-\alpha}(x)P\bar{U}_{\xi_{\varepsilon},\lambda_{\varepsilon}}^{4-\alpha}(y)w^{2}_{\varepsilon}(y)}{|x-y|^{\alpha}}dxdy\\
        &\quad+o(\lambda_{\varepsilon}^{-1}).
    \end{aligned}
\end{equation*}
Then it follows that
\begin{equation}\label{prop-energy-expansion-proof-5}
    \begin{aligned}
        S_{HL}(\bar{a}-\varepsilon)=&S_{HL}-\frac{64}{3}S_{HL}\lambda_{\varepsilon}^{-1}\phi_{0}(\xi_{\varepsilon})\\
        &+S_{HL}^{-\frac{1}{5-\alpha}}\lambda_{\varepsilon}^{-1}\int_{\Omega}\bar{a}((\lambda_{\varepsilon}^{\frac{1}{2}}P\bar{U}_{\xi_{\varepsilon},\lambda_{\varepsilon}})^{2}+(\lambda_{\varepsilon}^{\frac{1}{2}}P\bar{U}_{\xi_{\varepsilon},\lambda_{\varepsilon}})(\lambda_{\varepsilon}^{\frac{1}{2}}w_{\varepsilon}))dx+o(\lambda_{\varepsilon}^{-1}).
    \end{aligned}
\end{equation}

Let $\bar{w}_{\varepsilon}:=\lambda_{\varepsilon}^{\frac{1}{2}}w_{\varepsilon}$. Then $\bar{w}_{\varepsilon}$ satisfies 
   \begin{equation*}
\begin{cases}
    \begin{aligned}
 &-\Delta \bar{w}_{\varepsilon}+(\bar{a}-\varepsilon)\left(\bar{w}_{\varepsilon}+\lambda_{\varepsilon}^{1/2}P\bar{U}_{\xi_{\varepsilon},\lambda_{\varepsilon}}\right)\\
 &\quad=(\mu_{\varepsilon})^{10-2\alpha}\lambda_{\varepsilon}^{-(5-\alpha)}\left(\int_{\Omega}\frac{(\lambda_{\varepsilon}^{1/2}P\bar{U}_{\xi_{\varepsilon},\lambda_{\varepsilon}}+\bar{w}_{\varepsilon})^{6-\alpha}(y)}{|x-y|^{\alpha}}dy\right)(\lambda_{\varepsilon}^{1/2}P\bar{U}_{\xi_{\varepsilon},\lambda_{\varepsilon}}+\bar{w}_{\varepsilon})^{5-\alpha}\\
 &\qquad-\lambda_{\varepsilon}^{-(5-\alpha)}\left(\int_{\R^{3}}\frac{(\lambda_{\varepsilon}^{1/2}\bar{U}_{\xi_{\varepsilon},\lambda_{\varepsilon}})^{6-\alpha}(y)}{|x-y|^{\alpha}}dy\right)(\lambda_{\varepsilon}^{1/2}\bar{U}_{\xi_{\varepsilon},\lambda_{\varepsilon}})^{5-\alpha}
\end{aligned}&\text{~in~}\Omega,\\
\bar{w}_{\varepsilon}=0&\text{~on~}\partial\Omega.
\end{cases}
\end{equation*}
Moreover, Lemma \ref{boundw} and Lemma \ref{lemma-energy-1} imply that $\bar{w}_{\varepsilon}$ is bounded in $H^1_{0}(\Omega)$. Thus there exists $\bar{w}_{0}\in H^1_{0}(\Omega)$ such that, up to a subsequence,
\begin{equation}\label{prop-energy-expansion-proof-7}
    \bar{w}_{\varepsilon}\to \bar{w}_{0}\text{~weakly~in~}H^{1}_{0}(\Omega).
\end{equation}
Given any $\varphi\in C_{c}^{\infty}(\Omega\setminus\{\xi_{0}\})$. By \eqref{important-identity-1}, Proposition \ref{lemma-1}, Lemma \ref{lem inequality 1}, Lemma \ref{boundw}, Lemma \ref{lemma-energy-1}, the HLS inequality and the H\"older inequality, we have
\begin{equation*}
    \begin{aligned}
        &\lambda_{\varepsilon}^{-(5-\alpha)}\int_{\Omega}\int_{\Omega}\frac{(\lambda_{\varepsilon}^{1/2}P\bar{U}_{\xi_{\varepsilon},\lambda_{\varepsilon}}+\bar{w}_{\varepsilon})^{6-\alpha}(y)(\lambda_{\varepsilon}^{1/2}P\bar{U}_{\xi_{\varepsilon},\lambda_{\varepsilon}}+\bar{w}_{\varepsilon})^{5-\alpha}(x)\varphi(x)}{|x-y|^{\alpha}}dydx\\
        &\lesssim \lambda_{\varepsilon}^{-\frac{4-\alpha}{2}}\left(\int_{\Omega}u_{\varepsilon}^{6}dy\right)^{\frac{6-\alpha}{6}}\left(\int_{\Omega\cap supp(\varphi)}(\lambda_{\varepsilon}^{1/2}\bar{U}_{\xi_{\varepsilon},\lambda_{\varepsilon}})^{6}+\bar{w}_{\varepsilon}^{6}dx\right)^{\frac{5-\alpha}{6}}\\
        &\lesssim\lambda_{\varepsilon}^{-\frac{4-\alpha}{2}}
    \end{aligned}
\end{equation*}
and
\begin{equation*}
    \begin{aligned}
        &\lambda_{\varepsilon}^{-(5-\alpha)}\int_{\Omega}\int_{\R^{3}}\frac{(\lambda_{\varepsilon}^{1/2}\bar{U}_{\xi_{\varepsilon},\lambda_{\varepsilon}})^{6-\alpha}(y)(\lambda_{\varepsilon}^{1/2}\bar{U}_{\xi_{\varepsilon},\lambda_{\varepsilon}})^{5-\alpha}(x)\varphi(x)}{|x-y|^{\alpha}}dydx\\
        &=\frac{3}{\bar{C}_{\alpha}^{4}}\lambda_{\varepsilon}^{-2}\int_{\Omega\cap supp(\varphi)}(\lambda_{\varepsilon}^{1/2}\bar{U}_{\xi_{\varepsilon},\lambda_{\varepsilon}})^{5}(x)\varphi(x)dx\\
        &\lesssim \lambda_{\varepsilon}^{-2}.
    \end{aligned}
\end{equation*}
On the other hand, Lemma \ref{lemma PU} yields that
\begin{equation*}
    \begin{aligned}
         \lambda_{\varepsilon}^{\frac{1}{2}}P{U}_{\xi_{\varepsilon},\lambda_{\varepsilon}}&=\lambda_{\varepsilon}^{\frac{1}{2}}{U}_{\xi_{\varepsilon},\lambda_{\varepsilon}}+4\pi\left(G_{0}(\xi_{\varepsilon},\cdot)-\frac{1}{4\pi|\xi_{\varepsilon}-\cdot|}\right)+O(\lambda_{\varepsilon}^{-2})\\
         &=4\pi G_{0}(\xi_{\varepsilon},\cdot)-\lambda_{\varepsilon}^{1/2}\left(\frac{\lambda_{\varepsilon}^{-1/2}}{|\xi_{\varepsilon}-\cdot|}-\frac{\lambda_{\varepsilon}^{1/2}}{(1 + \lambda_{\varepsilon}^2|\xi_{\varepsilon}-\cdot|^2)^{1/2}}\right)+O(\lambda_{\varepsilon}^{-2})\\
          &:=4\pi G_{0}(\xi_{\varepsilon},\cdot)-\lambda_{\varepsilon}^{1/2}g_{\xi_{\varepsilon},\lambda_{\varepsilon}}+O(\lambda_{\varepsilon}^{-2}).
    \end{aligned}
\end{equation*}
This together with \eqref{prop-energy-expansion-proof-7} and Lemma \ref{lemma-g} gives that
\begin{equation*}
    \begin{aligned}
        &\int_{\Omega}\nabla \bar{w}_{\varepsilon}\cdot\nabla \varphi dx+\int_{\Omega}(\bar{a}-\varepsilon)\left(\bar{w}_{\varepsilon}+\lambda_{\varepsilon}^{1/2}P\bar{U}_{\xi_{\varepsilon},\lambda_{\varepsilon}}\right)\varphi dx\\
        &=\int_{\Omega}\nabla \bar{w}_{0}\cdot\nabla \varphi dx+\int_{\Omega}\bar{a}\bar{w}_{0}\varphi dx+4\pi\int_{\Omega}\bar{a} G_{0}(\xi_{0},\cdot)\varphi dx+o(1).
    \end{aligned}
\end{equation*}
Therefore, $\bar{w}_{0}$ is a solution to
\begin{equation*}
\begin{cases}
 -\Delta\bar{w}_{0}+\bar{a}\bar{w}_{0}=-\bar{a}4\pi\bar{C}_{\alpha}G_{0}(\xi_{0},\cdot)&\text{~in~}\Omega,\\
 \bar{w}_{0}=0&\text{~on~}\partial\Omega.
\end{cases}
\end{equation*}
By the coercivity of the operator $-\Delta+\bar{a}$, we conclude that
\begin{equation}\label{convergence-w}
    \bar{w}_{0}=4\pi\bar{C}_{\alpha}(H_{\bar{a}}(\xi_{0},\cdot)-H_{0}(\xi_{0},\cdot)).
\end{equation}
It then follows that
\begin{equation*}
    \begin{aligned}       &\int_{\Omega}\bar{a}\left((\lambda_{\varepsilon}^{1/2}P\bar{U}_{\xi_{\varepsilon},\lambda_{\varepsilon}})^{2}+\lambda_{\varepsilon}^{1/2}P\bar{U}_{\xi_{\varepsilon},\lambda_{\varepsilon}}\bar{w}_{\varepsilon}\right)dx\\
       &=\int_{\Omega}\bar{a}(16\pi^{2}\bar{C}^{2}_{\alpha}G_{0}^{2}(\xi_{\varepsilon},\cdot)+4\pi\bar{C}_{\alpha}G_{0}(\xi_{\varepsilon},\cdot)\bar{w}_{0})dx+o(1)\\
       &=-4\pi\bar{C}_{\alpha}\int_{\Omega}(-\Delta\bar{w}_{0})G_{0}(\xi_{\varepsilon},\cdot)dx+o(1)\\
       &=-4\pi\bar{C}_{\alpha}\bar{w}_{0}(\xi_{0})+o(1)=-16\pi^{2}\bar{C}_{\alpha}^{2}(\phi_{\bar{a}}(\xi_{0})-\phi_{0}(\xi_{0}))+o(1)\\
       &=-\frac{64}{3}S_{HL}^{\frac{6-\alpha}{5-\alpha}}(\phi_{\bar{a}}(\xi_{0})-\phi_{0}(\xi_{0}))+o(1).
    \end{aligned}
\end{equation*}
Combining this estimate with \eqref{prop-energy-expansion-proof-5}, we obtain
\begin{equation}
    \begin{aligned}
        S_{HL}(\bar{a}-\varepsilon)=S_{HL}&-\frac{64}{3}S_{HL}\lambda_{\varepsilon}^{-1}\phi_{\bar{a}}(\xi_{0})+o(\lambda_{\varepsilon}^{-1}).
    \end{aligned}
\end{equation}
This completes the proof.
\end{proof}

\begin{proof}[Proof of Theorem \ref{thm-00}]
It suffices to prove $(2)\Rightarrow(1)$. By Proposition \ref{prop-energy-expansion}, we have
 \begin{equation}
 \begin{aligned}\label{energy-critical-1}
      S_{HL}>S_{HL}(\bar{a}-\varepsilon)=S_{HL}&-\frac{64}{3}S_{HL}\lambda_{\varepsilon}^{-1}\phi_{\bar{a}}(\xi_{0})+o(\lambda_{\varepsilon}^{-1}).
 \end{aligned}
 \end{equation}
 Letting $\varepsilon\to0$, we obtain $\phi_{\bar{a}}(\xi_{0})\geq0$. Moreover, Theorem \ref{thm-1} yields that $\phi_{\bar{a}}(\xi_{0})=0$. Since $\bar{a}=a+B(a)>a$, Lemma \ref{lem-H-a} implies that
 \begin{equation}
     \phi_{a}(\xi_{0})>\phi_{\bar{a}}(\xi_{0})=0.
 \end{equation}
 This completes the proof.
\end{proof}

\section{Asymptotic behavior}\label{section-4}
We first establish an upper bound for $S_{HL}(a+\varepsilon V)$ using the test function defined by
\begin{equation*}
    \psi_{\xi,\lambda}:=P\bar{U}_{\xi,\lambda}+\lambda^{-\frac{1}{2}}4\pi\bar{C}_{\alpha}(H_{a}(\xi,\cdot)-H_{0}(\xi,\cdot)),\quad\xi\in\Omega\text{~and~}\lambda\in \R^{+}.
\end{equation*}
Let us define
\begin{equation}
    F_{\xi,\lambda}:={U}_{\xi,\lambda}+4\pi\lambda^{-\frac{1}{2}}H_{a}(\xi,\cdot).
\end{equation}
It then follows from Lemma \ref{lemma PU} that
    \begin{equation}\label{eq-psi-1}
    \begin{aligned}
        \psi_{\xi,\lambda}&=\bar{U}_{\xi,\lambda}-\bar{C}_{\alpha}\varphi_{\xi,\lambda}+4\pi\bar{C}_{\alpha}\lambda^{-\frac{1}{2}}(H_{a}(\xi,\cdot)-H_{0}(\xi,\cdot))\\
        &=\bar{U}_{\xi,\lambda}+4\pi\bar{C}_{\alpha}\lambda^{-\frac{1}{2}}H_{a}(\xi,\cdot)-\bar{C}_{\alpha}f_{\xi,\lambda}\\
        &=\bar{C}_{\alpha}\left(F_{\xi,\lambda}-f_{\xi,\lambda}\right).
    \end{aligned}
\end{equation}
Moreover, by \eqref{Ha-pde} and \eqref{eq-pU}, the function $\psi_{\xi,\lambda}$ satisfies 
\begin{equation*}
\begin{cases}
    -\Delta \psi_{\xi,\lambda}
    =3\bar{C}_{\alpha}{U}_{\xi,\lambda}^{5}-4\pi\bar{C}_{\alpha}\lambda^{-\frac{1}{2}}aG_{a}(\xi,\cdot)&\text{~in~}\Omega,\\
   \psi_{\xi,\lambda}=0&\text{~on~}\partial\Omega.
\end{cases}
\end{equation*}
For $u\in H^{1}_{0}(\Omega)$, we define the functional
\begin{equation*}
    S_{HL}(a+\varepsilon V)[u]:=\frac{\int_{\Omega}|\nabla u|^{2}dx+\int_{\Omega}(a+\varepsilon V)u^{2}dx}{\left(\int_{\Omega}\int_{\Omega}\frac{u^{6-\alpha}(x)u^{6-\alpha}(y)}{|x-y|^{\alpha}}dxdy\right)^{\frac{1}{6-\alpha}}}.
\end{equation*}

\begin{Prop}\label{prop-upper-estimate}
As $\lambda\to\infty$, the following holds uniformly for $\xi$ in compact subsets of $\Omega$ and for $\varepsilon \geq 0$
\begin{equation}\label{estimate-psi-1}
    \begin{aligned}
        &\int_{\Omega}|\nabla\psi_{\xi,\lambda}|^{2}dx+\int_{\Omega}(a+\varepsilon V)\psi_{\xi,\lambda}^{2}dx\\
        &=S_{HL}^{\frac{6-\alpha}{5-\alpha}}+16\pi^{2}\bar{C}_{\alpha}^{2}\phi_{a}(\xi)\lambda^{-1}+2\pi(4-\pi)\bar{C}_{\alpha}^{2}a(\xi)\lambda^{-2}\\
        &\quad+16\pi^{2}\bar{C}_{\alpha}^{2}\varepsilon\lambda^{-1} Q_{V}(\xi)+o(\lambda^{-2})+o(\varepsilon\lambda^{-1})
    \end{aligned}
\end{equation}
and
\begin{equation}\label{estimate-psi-2}
    \begin{aligned}
        &\int_{\Omega}\int_{\Omega}\frac{\psi_{\xi,\lambda}^{6-\alpha}(y)\psi_{\xi,\lambda}^{6-\alpha}(x)}{|x-y|^{\alpha}}dxdy\\   
        &=S_{HL}^{\frac{6-\alpha}{5-\alpha}}+32(6-\alpha)\pi^{2}\bar{C}_{\alpha}^{2}\phi_{a}(\xi)\lambda^{-1}+8(6-\alpha)\pi\bar{C}_{\alpha}^{2}a(\xi)\lambda^{-2}\\
         &\quad+16(6-\alpha)\pi^{2}\bar{C}_{\alpha}^{2}\phi_{a}^{2}(\xi)\left((6-\alpha){C}_{1,\alpha}+(5-\alpha)3\pi^{2}\right)\lambda^{-2}+o(\lambda^{-2}),
    \end{aligned}
\end{equation}
where
\begin{equation*}
    {C}_{1,\alpha}:=\int_{\R^{3}}\int_{\R^{3}}\frac{\bar{U}_{0,1}^{5-\alpha}(x)\bar{U}_{0,1}^{5-\alpha}(y)}{|x-y|^{\alpha}}dxdy.
\end{equation*}
Moreover, we have
\begin{equation}\label{estimate-psi-3}
    \begin{aligned}
        S_{HL}(a+\varepsilon V)[\psi_{\xi,\lambda}]=&S_{HL}-\frac{64}{3}S_{HL}\phi_{a}(\xi)\lambda^{-1}+\frac{64}{3}S_{HL}Q_{V}(\xi)\varepsilon\lambda^{-1}-\frac{8}{3}S_{HL}a(\xi)\lambda^{-2}\\
         &\quad-\frac{64}{3}S_{HL}\phi_{a}^{2}(\xi)\left((6-\alpha){C}_{1,\alpha}+(5-\alpha)3\pi^{2}-\frac{128(6-\alpha)}{3}\right)\lambda^{-2}\\
         &\quad+o(\lambda^{-2})+o(\varepsilon\lambda^{-1}).
    \end{aligned}
\end{equation}
\end{Prop}

\begin{proof}
First, it follows from \eqref{eq-psi-1}, Lemma \ref{lem inequality 1} and Lemma \ref{lem inequality 2} that
    \begin{equation*}
        \begin{aligned}
            &\int_{\Omega}\int_{\Omega}\frac{\psi_{\xi,\lambda}^{6-\alpha}(y)\psi_{\xi,\lambda}^{6-\alpha}(x)}{|x-y|^{\alpha}}dxdy\\           
            &=\bar{C}_{\alpha}^{2(6-\alpha)}\int_{\Omega}\int_{\Omega}\frac{F_{\xi,\lambda}^{6-\alpha}(y)F_{\xi,\lambda}^{6-\alpha}(x)}{|x-y|^{\alpha}}dxdy+O\left\{\int_{\Omega}\int_{\Omega}\frac{F_{\xi,\lambda}^{6-\alpha}(y)F_{\xi,\lambda}^{5-\alpha}(x)f_{\xi,\lambda}(x)}{|x-y|^{\alpha}}dxdy\right.\\
           &\quad+\int_{\Omega}\int_{\Omega}\frac{F_{\xi,\lambda}^{5-\alpha}(y)f_{\xi,\lambda}(y)F_{\xi,\lambda}^{5-\alpha}(x)f_{\xi,\lambda}(x)}{|x-y|^{\alpha}}dxdy\\
           &\quad+\left.\int_{\Omega}\int_{\Omega}\frac{(F_{\xi,\lambda}^{6-\alpha}+F_{\xi,\lambda}^{5-\alpha}f_{\xi,\lambda}+f_{\xi,\lambda}^{6-\alpha})(y)f^{6-\alpha}_{\xi,\lambda}(x)}{|x-y|^{\alpha}}dxdy\right\}.
        \end{aligned}
    \end{equation*}
By Lemma \ref{lemma PU} and some direct computations, we have
\begin{equation*}
  \|f_{\xi,\lambda}\|_{L^{\infty}(\Omega)}=O(\lambda^{-\frac{5}{2}}),\quad  \|F_{\xi,\lambda}\|^{6-\alpha}_{L^{6}(\Omega)}=O(1),\quad \|F_{\xi,\lambda}\|^{5-\alpha}_{L^{\frac{6(5-\alpha)}{6-\alpha}}(\Omega)}=O(\lambda^{-\frac{1}{2}}).
\end{equation*}
This together with the HLS inequality yields that
\begin{equation*}
    \begin{aligned}
        &\int_{\Omega}\int_{\Omega}\frac{F_{\xi,\lambda}^{6-\alpha}(y)F_{\xi,\lambda}^{5-\alpha}(x)f_{\xi,\lambda}(x)}{|x-y|^{\alpha}}dxdy\\
           &\quad+\int_{\Omega}\int_{\Omega}\frac{F_{\xi,\lambda}^{5-\alpha}(y)f_{\xi,\lambda}(y)F_{\xi,\lambda}^{5-\alpha}(x)f_{\xi,\lambda}(x)}{|x-y|^{\alpha}}dxdy\\
           &\quad+\int_{\Omega}\int_{\Omega}\frac{(F_{\xi,\lambda}^{6-\alpha}+F_{\xi,\lambda}^{5-\alpha}f_{\xi,\lambda}+f_{\xi,\lambda}^{6-\alpha})(y)f^{6-\alpha}_{\xi,\lambda}(x)}{|x-y|^{\alpha}}dxdy\\
           &\lesssim \|F_{\xi,\lambda}\|^{6-\alpha}_{L^{6}(\Omega)}\|F_{\xi,\lambda}\|^{5-\alpha}_{L^{\frac{6(5-\alpha)}{6-\alpha}}(\Omega)}\|f_{\xi,\lambda}\|_{L^{\infty}(\Omega)}+\|F_{\xi,\lambda}\|^{2(5-\alpha)}_{L^{\frac{6(5-\alpha)}{6-\alpha}}(\Omega)}\|f_{\xi,\lambda}\|^{2}_{L^{\infty}(\Omega)}\\
           &\quad+\left(\|F_{\xi,\lambda}\|^{6-\alpha}_{L^{6}(\Omega)}+\|F_{\xi,\lambda}\|^{5-\alpha}_{L^{\frac{6(5-\alpha)}{6-\alpha}}(\Omega)}\|f_{\xi,\lambda}\|_{L^{\infty}(\Omega)}+\|f_{\xi,\lambda}\|_{L^{\infty}(\Omega)}^{6-\alpha}\right)\|f_{\xi,\lambda}\|_{L^{\infty}(\Omega)}^{6-\alpha}\\
           &=o(\lambda^{-2}).
    \end{aligned}
\end{equation*}
Moreover, by Lemma \ref{lem inequality 1}, Lemma \ref{lem inequality 2} and some direct computations, we find that
\begin{equation*}
    \begin{aligned}
        &\int_{\Omega}\int_{\Omega}\frac{F_{\xi,\lambda}^{6-\alpha}(y)F_{\xi,\lambda}^{6-\alpha}(x)}{|x-y|^{\alpha}}dxdy=\int_{\Omega}\int_{\Omega}\frac{{U}_{\xi,\lambda}^{6-\alpha}(y){U}_{\xi,\lambda}^{6-\alpha}(x)}{|x-y|^{\alpha}}dxdy\\
        &\quad+8(6-\alpha)\pi\lambda^{-\frac{1}{2}}\int_{\Omega}\int_{\Omega}\frac{{U}_{\xi,\lambda}^{6-\alpha}(y){U}_{\xi,\lambda}^{5-\alpha}(x)H_{a}(\xi,x)}{|x-y|^{\alpha}}dxdy\\
         &\quad+16(6-\alpha)^{2}\pi^{2}\lambda^{-1}\int_{\Omega}\int_{\Omega}\frac{{U}_{\xi,\lambda}^{5-\alpha}(y)H_{a}(\xi,y){U}_{\xi,\lambda}^{5-\alpha}(x)H_{a}(\xi,x)}{|x-y|^{\alpha}}dxdy\\
        &\quad+16(6-\alpha)(5-\alpha)\pi^{2}\lambda^{-1}\int_{\Omega}\int_{\Omega}\frac{{U}_{\xi,\lambda}^{6-\alpha}(y){U}_{\xi,\lambda}^{4-\alpha}(x)H_{a}^{2}(\xi,x)}{|x-y|^{\alpha}}dxdy\\
        &\quad+O\left\{\lambda^{-\frac{3}{2}}G^{1}_{\xi,\lambda}+\lambda^{-2}G^{2}_{\xi,\lambda}+\lambda^{-\frac{6-\alpha}{2}}\int_{\Omega}\int_{\Omega}\frac{{U}_{\xi,\lambda}^{6-\alpha}(y)H_{a}^{6-\alpha}(\xi,x)}{|x-y|^{\alpha}}dxdy\right\}+o(\lambda^{-2}),
    \end{aligned}
\end{equation*}
where
\begin{equation*}
\begin{aligned}
 G^{1}_{\xi,\lambda}&:=\int_{\Omega}\int_{\Omega}\frac{{U}_{\xi,\lambda}^{6-\alpha}(y){U}_{\xi,\lambda}^{3-\alpha}(x)H_{a}^{3}(\xi,x)}{|x-y|^{\alpha}}dxdy\\
 &\quad+\int_{\Omega}\int_{\Omega}\frac{{U}_{\xi,\lambda}^{5-\alpha}(y)H_{a}(\xi,y){U}_{\xi,\lambda}^{4-\alpha}(x)H_{a}^{2}(\xi,x)}{|x-y|^{\alpha}}dxdy   
\end{aligned}
\end{equation*}
and
\begin{equation*}
\begin{aligned}
     G^{2}_{\xi,\lambda}&:=\int_{\Omega}\int_{\Omega}\frac{{U}_{\xi,\lambda}^{5-\alpha}(y)H_{a}(\xi,y){U}_{\xi,\lambda}^{3-\alpha}(x)H_{a}^{3}(\xi,x)}{|x-y|^{\alpha}}dxdy\\
     &\quad+\int_{\Omega}\int_{\Omega}\frac{{U}_{\xi,\lambda}^{4-\alpha}(y)H^{2}_{a}(\xi,y){U}_{\xi,\lambda}^{4-\alpha}(x)H_{a}^{2}(\xi,x)}{|x-y|^{\alpha}}dxdy.
\end{aligned}
\end{equation*}
Combining \eqref{important-identity-1} and the HLS inequality, we obtain
\begin{equation*}
    \begin{aligned}
        &\int_{\Omega}\int_{\Omega}\frac{{U}_{\xi,\lambda}^{6-\alpha}(y){U}_{\xi,\lambda}^{6-\alpha}(x)}{|x-y|^{\alpha}}dxdy\\
        &=\int_{\R^{3}}\int_{\Omega}\frac{{U}_{\xi,\lambda}^{6-\alpha}(y){U}_{\xi,\lambda}^{6-\alpha}(x)}{|x-y|^{\alpha}}dxdy-\int_{\R^{3}\setminus\Omega}\int_{\Omega}\frac{{U}_{\xi,\lambda}^{6-\alpha}(y){U}_{\xi,\lambda}^{6-\alpha}(x)}{|x-y|^{\alpha}}dxdy\\
        &=\frac{3}{\bar{C}_{\alpha}^{2(5-\alpha)}}\int_{\R^{3}}{U}_{\xi,\lambda}^{6}dx+O\left(\int_{\R^{3}\setminus\Omega}{U}_{\xi,\lambda}^{6}dx\right)
        =\frac{3^{-\frac{1}{2}}S^{\frac{3}{2}}}{\bar{C}_{\alpha}^{2(5-\alpha)}}+o(\lambda^{-2})
    \end{aligned}
\end{equation*}
\begin{equation*}
    \begin{aligned}
         G^{1}_{\xi,\lambda}&\lesssim \left(\|{U}_{\xi,\lambda}\|_{L^{4}(\Omega)}^{4}+\|{U}_{\xi,\lambda}\|_{L^{\frac{6(5-\alpha)}{6-\alpha}}(\Omega)}^{5-\alpha}\|{U}_{\xi,\lambda}\|_{L^{\frac{6(4-\alpha)}{6-\alpha}}(\Omega)}^{4-\alpha}\right)\|H_{a}(\xi,x)\|^{3}_{L^{\infty}(\Omega)}\\
         &=O(\lambda^{-1})
    \end{aligned}
\end{equation*}
\begin{equation*}
    \begin{aligned}
        G^{2}_{\xi,\lambda}&\lesssim \left(\|{U}_{\xi,\lambda}\|_{L^{\frac{6(5-\alpha)}{6-\alpha}}(\Omega)}^{5-\alpha}\|{U}_{\xi,\lambda}\|_{L^{\frac{6(3-\alpha)}{6-\alpha}}(\Omega)}^{3-\alpha}+\|{U}_{\xi,\lambda}\|_{L^{\frac{6(4-\alpha)}{6-\alpha}}(\Omega)}^{2(4-\alpha)}\right)\|H_{a}(\xi,x)\|^{4}_{L^{\infty}(\Omega)}\\
        &=O(\lambda^{-2})
    \end{aligned}
\end{equation*}
and
\begin{equation*}
    \begin{aligned}
        \int_{\Omega}\int_{\Omega}\frac{{U}_{\xi,\lambda}^{6-\alpha}(y)H_{a}^{6-\alpha}(\xi,x)}{|x-y|^{\alpha}}dxdy&\lesssim\int_{\Omega}{U}_{\xi,\lambda}^{\alpha}(x)dx\|H_{a}(\xi,x)\|^{6-\alpha}_{L^{\infty}(\Omega)}\\
        &=O\left(\lambda^{-\frac{\alpha}{2}}\right).
    \end{aligned}
\end{equation*}
On the other hand, by \eqref{important-identity-1} and \cite[Lemma B.3]{Frank2021BlowupOS}, we obtain
\begin{equation*}
    \begin{aligned}
        &\int_{\Omega}\int_{\Omega}\frac{{U}_{\xi,\lambda}^{6-\alpha}(y){U}_{\xi,\lambda}^{5-\alpha}(x)H_{a}(\xi,x)}{|x-y|^{\alpha}}dxdy\\
        &=\int_{\R^{3}}\int_{\Omega}\frac{{U}_{\xi,\lambda}^{6-\alpha}(y){U}_{\xi,\lambda}^{5-\alpha}(x)H_{a}(\xi,x)}{|x-y|^{\alpha}}dxdy-\int_{\R^{3}\setminus\Omega}\int_{\Omega}\frac{{U}_{\xi,\lambda}^{6-\alpha}(y){U}_{\xi,\lambda}^{5-\alpha}(x)H_{a}(\xi,x)}{|x-y|^{\alpha}}dxdy\\
        &=\frac{3}{\bar{C}_{\alpha}^{2(5-\alpha)}}\int_{\Omega}{U}_{\xi,\lambda}^{5}(x)H_{a}(\xi,x)dx+O\left(\|{U}_{\xi,\lambda}\|^{6-\alpha}_{L^{6}(\R^{3}\setminus\Omega)}\|{U}_{\xi,\lambda}\|^{5-\alpha}_{L^{\frac{6(5-\alpha)}{6-\alpha}}(\Omega)}\|H_{a}(\xi,x)\|_{L^{\infty}(\Omega)}\right)\\
        &=\frac{1}{\bar{C}_{\alpha}^{2(5-\alpha)}}\left(4\pi\phi_{a}(\xi)\lambda^{-\frac{1}{2}}+a(\xi)\lambda^{-\frac{3}{2}}\right)+o(\lambda^{-\frac{3}{2}})
    \end{aligned}
\end{equation*}
and
\begin{equation*}
    \begin{aligned}
        &\int_{\Omega}\int_{\Omega}\frac{{U}_{\xi,\lambda}^{6-\alpha}(y){U}_{\xi,\lambda}^{4-\alpha}(x)H_{a}^{2}(\xi,x)}{|x-y|^{\alpha}}dxdy\\
        &=\int_{\R^{3}}\int_{\Omega}\frac{{U}_{\xi,\lambda}^{6-\alpha}(y){U}_{\xi,\lambda}^{4-\alpha}(x)H_{a}^{2}(\xi,x)}{|x-y|^{\alpha}}dxdy-\int_{\R^{3}\setminus\Omega}\int_{\Omega}\frac{{U}_{\xi,\lambda}^{6-\alpha}(y){U}_{\xi,\lambda}^{4-\alpha}(x)H_{a}^{2}(\xi,x)}{|x-y|^{\alpha}}dxdy\\
        &=\frac{3}{\bar{C}_{\alpha}^{2(5-\alpha)}}\int_{\Omega}{U}_{\xi,\lambda}^{4}(x)H_{a}^{2}(\xi,x)dx+O\left(\|{U}_{\xi,\lambda}\|^{6-\alpha}_{L^{6}(\R^{3}\setminus\Omega)}\|{U}_{\xi,\lambda}\|^{4-\alpha}_{L^{\frac{6(4-\alpha)}{6-\alpha}}(\Omega)}\|H_{a}(\xi,x)\|^{2}_{L^{\infty}(\Omega)}\right)\\
        &=\frac{3\pi^{2}}{\bar{C}_{\alpha}^{2(5-\alpha)}}\phi^{2}_{a}(\xi)\lambda^{-1}+o(\lambda^{-1}).
    \end{aligned}
\end{equation*}
Let $0<\rho\leq dist(\xi,\partial\Omega)$. Then it holds that
\begin{equation*}
    \begin{aligned}
        &\int_{\Omega}\int_{\Omega}\frac{{U}_{\xi,\lambda}^{5-\alpha}(y)H_{a}(\xi,y){U}_{\xi,\lambda}^{5-\alpha}(x)H_{a}(\xi,x)}{|x-y|^{\alpha}}dxdy\\
        &=\int_{B_{\rho}(\xi)}\int_{B_{\rho}(\xi)}\frac{{U}_{\xi,\lambda}^{5-\alpha}(y)H_{a}(\xi,y){U}_{\xi,\lambda}^{5-\alpha}(x)H_{a}(\xi,x)}{|x-y|^{\alpha}}dxdy\\
        &\quad+\int_{\Omega\setminus B_{\rho}(\xi)}\int_{B_{\rho}(\xi)}\frac{{U}_{\xi,\lambda}^{5-\alpha}(y)H_{a}(\xi,y){U}_{\xi,\lambda}^{5-\alpha}(x)H_{a}(\xi,x)}{|x-y|^{\alpha}}dxdy\\
        &\quad+\int_{\Omega}\int_{\Omega\setminus B_{\rho}(\xi)}\frac{{U}_{\xi,\lambda}^{5-\alpha}(y)H_{a}(\xi,y){U}_{\xi,\lambda}^{5-\alpha}(x)H_{a}(\xi,x)}{|x-y|^{\alpha}}dxdy\\
        &=\int_{B_{\rho}(\xi)}\int_{B_{\rho}(\xi)}\frac{{U}_{\xi,\lambda}^{5-\alpha}(y)H_{a}(\xi,y){U}_{\xi,\lambda}^{5-\alpha}(x)H_{a}(\xi,x)}{|x-y|^{\alpha}}dxdy\\
        &\quad+O\left(\|{U}_{\xi,\lambda}\|^{5-\alpha}_{L^{\frac{6(5-\alpha)}{6-\alpha}}(\Omega\setminus B_{\rho}(\xi))}\|{U}_{\xi,\lambda}\|^{5-\alpha}_{L^{\frac{6(5-\alpha)}{6-\alpha}}(\Omega)}\|H_{a}(\xi,\cdot)\|^{2}_{L^{\infty}(\Omega)}\right)\\
        &=\int_{B_{\rho}(\xi)}\int_{B_{\rho}(\xi)}\frac{{U}_{\xi,\lambda}^{5-\alpha}(y)H_{a}(\xi,y){U}_{\xi,\lambda}^{5-\alpha}(x)H_{a}(\xi,x)}{|x-y|^{\alpha}}dxdy+o(\lambda^{-1}).
    \end{aligned}
\end{equation*}
Moreover, by the HLS inequality and Lemma \ref{lem-H-a}, we have
\begin{equation*}
    \begin{aligned}
          &\int_{B_{\rho}(\xi)}\int_{B_{\rho}(\xi)}\frac{{U}_{\xi,\lambda}^{5-\alpha}(y)H_{a}(\xi,y){U}_{\xi,\lambda}^{5-\alpha}(x)H_{a}(\xi,x)}{|x-y|^{\alpha}}dxdy\\
          &=\phi_{a}^{2}(\xi)\int_{B_{\rho}(\xi)}\int_{B_{\rho}(\xi)}\frac{{U}_{\xi,\lambda}^{5-\alpha}(y){U}_{\xi,\lambda}^{5-\alpha}(x)}{|x-y|^{\alpha}}dxdy\\
           &\quad+O\left\{\int_{B_{\rho}(\xi)}\int_{B_{\rho}(\xi)}\frac{{U}_{\xi,\lambda}^{5-\alpha}(y){U}_{\xi,\lambda}^{5-\alpha}(x)|x-\xi|}{|x-y|^{\alpha}}dxdy\right.\\
            &\quad+\left.\int_{B_{\rho}(\xi)}\int_{B_{\rho}(\xi)}\frac{{U}_{\xi,\lambda}^{5-\alpha}(y)|y-\xi|{U}_{\xi,\lambda}^{5-\alpha}(x)|x-\xi|}{|x-y|^{\alpha}}dxdy\right\}.
    \end{aligned}
\end{equation*}
It then follows from the HLS inequality and some direct computations that
\begin{equation*}
    \begin{aligned}
        &\int_{B_{\rho}(\xi)}\int_{B_{\rho}(\xi)}\frac{{U}_{\xi,\lambda}^{5-\alpha}(y){U}_{\xi,\lambda}^{5-\alpha}(x)|x-\xi_\varepsilon|}{|x-y|^{\alpha}}dxdy\\
        &\quad+\int_{B_{\rho}(\xi)}\int_{B_{\rho}(\xi)}\frac{{U}_{\xi,\lambda}^{5-\alpha}(y)|y-\xi|{U}_{\xi,\lambda}^{5-\alpha}(x)|x-\xi|}{|x-y|^{\alpha}}dxdy\\
        &\lesssim \left(\int_{B_{\rho}(\xi)}{U}_{\xi,\lambda}^{\frac{6(5-\alpha)}{6-\alpha}}(x)dx\right)^{\frac{6-\alpha}{6}}\left(\int_{B_{\rho}(\xi)}{U}_{\xi,\lambda}^{\frac{6(5-\alpha)}{6-\alpha}}(x)|x-\xi|^{\frac{6}{6-\alpha}}dx\right)^{\frac{6-\alpha}{6}}\\
        &\quad+\left(\int_{B_{\rho}(\xi)}{U}_{\xi,\lambda}^{\frac{6(5-\alpha)}{6-\alpha}}(x)|x-\xi|^{\frac{6}{6-\alpha}}dx\right)^{\frac{6-\alpha}{3}}\\
        &=\begin{cases}
            O(\lambda^{-2}),&\text{~if~}\alpha<2,\\
            O(\lambda^{-2}\log(\lambda\rho)^{\frac{2}{3}}),&\text{~if~}\alpha=2,\\
            O(\lambda^{-2}(\lambda\rho)^{\frac{\alpha-2}{2}}),&\text{~if~}\alpha>2,\\
        \end{cases}=o(\lambda^{-1})
    \end{aligned}
\end{equation*}
and
\begin{equation*}
    \begin{aligned}
        &\int_{B_{\rho}(\xi)}\int_{B_{\rho}(\xi)}\frac{{U}_{\xi,\lambda}^{5-\alpha}(y){U}_{\xi,\lambda}^{5-\alpha}(x)}{|x-y|^{\alpha}}dxdy\\
        &=\lambda^{-1}\int_{\R^{3}}\int_{\R^{3}}\frac{{U}_{0,1}^{5-\alpha}(y){U}_{0,1}^{5-\alpha}(x)}{|x-y|^{\alpha}}dxdy\\
        &\quad-\int_{\R^{3}\setminus B_{\rho}(\xi)}\int_{B_{\rho}(\xi)}\frac{{U}_{\xi,\lambda}^{5-\alpha}(y){U}_{\xi,\lambda}^{5-\alpha}(x)}{|x-y|^{\alpha}}dxdy-\int_{\R^{3}}\int_{\R^{3}\setminus B_{\rho}(\xi)}\frac{{U}_{\xi,\lambda}^{5-\alpha}(y){U}_{\xi,\lambda}^{5-\alpha}(x)}{|x-y|^{\alpha}}dxdy\\
        &:=\frac{{C}_{1,\alpha}}{\bar{C}_{\alpha}^{2(5-\alpha)}}\lambda^{-1}+O(\lambda^{-1}(\lambda\rho)^{-\frac{4-\alpha}{2}}).
    \end{aligned}
\end{equation*}
Combining the estimates above, we have
\begin{equation*}
    \begin{aligned}
        &\int_{\Omega}\int_{\Omega}\frac{\psi_{\xi,\lambda}^{6-\alpha}(y)\psi_{\xi,\lambda}^{6-\alpha}(x)}{|x-y|^{\alpha}}dxdy\\   
        &=S_{HL}^{\frac{6-\alpha}{5-\alpha}}+32(6-\alpha)\pi^{2}\bar{C}_{\alpha}^{2}\phi_{a}(\xi)\lambda^{-1}+8(6-\alpha)\pi\bar{C}_{\alpha}^{2}a(\xi)\lambda^{-2}\\
         &\quad+16(6-\alpha)\pi^{2}\bar{C}_{\alpha}^{2}\phi_{a}^{2}(\xi)\left((6-\alpha){C}_{1,\alpha}+(5-\alpha)3\pi^{2}\right)\lambda^{-2}+o(\lambda^{-2}).
    \end{aligned}
\end{equation*}
Thus \eqref{estimate-psi-2} holds. This together with the Taylor's formula yields that
\begin{equation}\label{proof-estimate-psi-29}
    \begin{aligned}
        &\left(\int_{\Omega}\int_{\Omega}\frac{\psi_{\xi,\lambda}^{6-\alpha}(y)\psi_{\xi,\lambda}^{6-\alpha}(x)}{|x-y|^{\alpha}}dxdy\right)^{-\frac{1}{6-\alpha}}\\
        &=S_{HL}^{-\frac{1}{5-\alpha}}-\frac{128}{3}S_{HL}^{-\frac{1}{5-\alpha}}\phi_{a}(\xi)\lambda^{-1}-\frac{32}{3}S_{HL}^{-\frac{1}{5-\alpha}}\left\{\frac{1}{\pi}a(\xi)\right.\\
         &\quad+\left.2\phi_{a}^{2}(\xi)\left((6-\alpha){C}_{1,\alpha}+(5-\alpha)3\pi^{2}-\frac{128(7-\alpha)}{3}\right)\right\}\lambda^{-2}\\
         &\quad+o(\lambda^{-2}).
    \end{aligned}
\end{equation}
On the other hand, the proof of \eqref{estimate-psi-1} is similar to (2.2) in \cite{Frank2019Energythree-dimensional}; we omit the details. Combining \eqref{estimate-psi-1} and \eqref{proof-estimate-psi-29}, we have
\begin{equation*}
    \begin{aligned}
        S_{HL}(a+\varepsilon V)[\psi_{\xi,\lambda}]=&S_{HL}-\frac{64}{3}S_{HL}\phi_{a}(\xi)\lambda^{-1}+\frac{64}{3}S_{HL}Q_{V}(\xi)\varepsilon\lambda^{-1}-\frac{8}{3}S_{HL}a(\xi)\lambda^{-2}\\
         &\quad-\frac{64}{3}S_{HL}\phi_{a}^{2}(\xi)\left((6-\alpha){C}_{1,\alpha}+(5-\alpha)3\pi^{2}-\frac{128(6-\alpha)}{3}\right)\lambda^{-2}\\
         &\quad+o(\lambda^{-2})+o(\varepsilon\lambda^{-1}).
    \end{aligned}
\end{equation*}
This establishes \eqref{estimate-psi-3} and completes the proof.

\end{proof}

Taking $\varepsilon = 0$, we obtain from \eqref{estimate-psi-3} the following corollary.
\begin{Cor}\label{corollary-phi-a}
    \begin{enumerate}[label=\upshape(\arabic*)]
        \item If $S_{HL}(a)=S_{HL}$, then $\phi_{a}(x)\leq 0$ for any $x\in\Omega$.
        \item If $S_{HL}(a)=S_{HL}$ and $\phi_{a}(x_{0})=0$ for some $x_{0}\in\Omega$, then $a(x_{0})\leq 0$.
    \end{enumerate}
\end{Cor}

\begin{proof}[Proof of Theorem \ref{thm-a}]
Suppose that $a$ is critical. Then Corollary \ref{corollary-phi-a} implies that $\phi_{a}(x)\leq 0$ for all $x\in\Omega$. Moreover, from the proof of Theorem \ref{thm-00} (see \eqref{energy-critical-1}), there exists a point $\xi_{0}\in\Omega$ such that $\phi_{a}(\xi_{0})=0$. Thus, $\max_{x\in\Omega}\phi_{a}=0$.

Conversely, suppose that $\phi_{a}(\xi_{0})=\max_{x\in\Omega}\phi_{a}=0$ for some $\xi_0 \in \Omega$. Then Theorem \ref{thm-00} implies that $S_{HL}(a)=S_{HL}$. Indeed, if $S_{HL}(a) < S_{HL}$, there would exist a point $\xi_{1}\in\Omega$ such that $\phi_{a}(\xi_{1})>0$, contradicting the assumption that $\max_{x\in\Omega}\phi_{a}=0$. Furthermore, for any function $\tilde{a}$ satisfying $\tilde{a}\leq a$ and $\tilde{a}\not\equiv a$, Lemma \ref{lem-H-a} yields $\phi_{\tilde{a}}(\xi_{0})>\phi_{a}(\xi_{0})=0$. It then follows from Theorem \ref{thm-00} that $S_{HL}(\tilde{a})<S_{HL}({a})$, and hence $a$ is a critical function. This completes the proof.
\end{proof}

\begin{Prop}\label{cor-energy-upper}
    Assume that $a(x)<0$ for any $x\in\mathcal{N}_{a}$ and $\mathcal{N}_{a}(V)\neq\emptyset$. Then $S_{HL}(a+\varepsilon V)<S_{HL}$ for any $\varepsilon>0$. Moreover, as $\varepsilon\to0$, it holds that
    \begin{equation}\label{eq-upper-bound}
         \begin{aligned}
         S_{HL}(a+\varepsilon V)\leq S_{HL}-\frac{128}{3}S_{HL}\sup_{\xi\in\mathcal{N}_{a}(V)}\frac{Q_{V}(\xi)^{2}}{|a(\xi)|}\varepsilon^{2}+o(\varepsilon^{2}).
    \end{aligned}     
    \end{equation}
\end{Prop}

\begin{proof}
Fix $\xi\in\mathcal{N}_{a}(V)$. Then by \eqref{estimate-psi-3} we have, as $\lambda\to\infty$,
    \begin{equation*}
    \begin{aligned}
         S_{HL}(a+\varepsilon V)&\leq S_{HL}(a+\varepsilon V)[\psi_{\xi,\lambda}]\\
        &=S_{HL}+\frac{8}{3}S_{HL}\left((-a(\xi)+o(1))\lambda^{-2}-8(-Q_{V}(\xi)+o(1))\varepsilon\lambda^{-1}\right).
    \end{aligned}     
    \end{equation*}
Define the function
\begin{equation*}
    f(\lambda):=\frac{A_{\xi}}{\lambda^{2}}-\varepsilon\frac{B_{\xi}}{\lambda}.
\end{equation*}
If $A_{\xi}, B_{\xi}$ are positive, then $f$ attains a unique global minimum at
\begin{equation*}
   \lambda_{0}=\left(\frac{2A_{\xi}}{B_{\xi}}\right)\varepsilon^{-1},
\end{equation*}
with the corresponding minimal value
\begin{equation*}
    \min_{\lambda>0} f(\lambda)=f(\lambda_{0})=-\frac{B_{\xi}^{2}}{4A_{\xi}}\varepsilon^{2}.
\end{equation*}
We now choose $\lambda=\frac{-a(\xi)}{-4Q_{V}(\xi)}\varepsilon^{-1}$. By the assumption and the above argument, we conclude that as $\varepsilon\to0$
\begin{equation*}
    \begin{aligned}
         S_{HL}(a+\varepsilon V) &\leq S_{HL}+\frac{8}{3}S_{HL}\left(-a(\xi)\lambda^{-2}-8(-Q_{V}(\xi))\varepsilon\lambda^{-1}\right)+o(\varepsilon^{2})\\
         &= S_{HL}-\frac{128}{3}S_{HL}\frac{Q_{V}(\xi)^{2}}{|a(\xi)|}\varepsilon^{2}+o(\varepsilon^{2}).
    \end{aligned}     
    \end{equation*}
    This establishes \eqref{eq-upper-bound}. In particular, $S_{HL}(a+\varepsilon V)<S_{HL}$ for all sufficiently small $\varepsilon>0$. Since $S_{HL}(a+\varepsilon V)$ is a concave function of $\varepsilon$ (being the infimum over $u$ of functions $S_{HL}(a+\varepsilon V)[u]$, which are linear in $\varepsilon$), it follows that $S_{HL}(a+\varepsilon V)<S_{HL}$ for all $\varepsilon>0$. This completes the proof.
\end{proof}

In what follows, we work under Assumption \ref{assumption-a}. If $\mathcal{N}_{a}(V)\neq\emptyset$, it then follows from Proposition \ref{cor-energy-upper} that $S_{HL}(a+\varepsilon V)<S_{HL}$ for any $\varepsilon>0$. This, together with Theorem \ref{thm-1}, yields that there exists $u_{\varepsilon}\in H^{1}_{0}(\Omega)$ such that $S_{HL}(a+\varepsilon V)[u_{\varepsilon}]=S_{HL}(a+\varepsilon V)$. After a suitable scaling, $u_{\varepsilon}$ satisfies the equation
\begin{equation*}
    \begin{cases}
        -\Delta u+(a+\varepsilon V)u=\displaystyle\left(\int_{\Omega}\frac{u^{6-\alpha}(y)}{|x-y|^\alpha}dy\right)u^{5-\alpha}
  \ \  &\mbox{in}\ \Omega,\\
   u>0\ \  &\mbox{in}\ \Omega,\\
  u=0\ \  &\mbox{on}\ \partial \Omega.
    \end{cases}
\end{equation*}
Similar to the argument in the proof of Theorem \ref{thm-00}, there exist sequences $\{\mu_{\varepsilon}\}\subset \R^{+}$, $\{\xi_{\varepsilon}\}\subset\Omega$, $\{\lambda_{\varepsilon}\}\subset \R^{+}$ and $\{w_{\varepsilon}\}\subset T_{\xi_{\varepsilon},\lambda_{\varepsilon}}^{\bot}$ such that 
\begin{equation*}
    u_{\varepsilon}=\mu_{\varepsilon}(P\bar{U}_{\xi_{\varepsilon},\lambda_{\varepsilon}}+w_{\varepsilon}).
\end{equation*}
Moreover, as $\varepsilon\to0$
\begin{equation}\label{eq-coefficient-convergence}
   \mu_{\varepsilon}\to 1,\quad\lambda_{\varepsilon}\to\infty,\quad\|\nabla w_{\varepsilon}\|_{L^{2}(\Omega)}=O(\lambda_{\varepsilon}^{-\frac{1}{2}}),\quad \xi_{\varepsilon}\to\xi_{0}\in\Omega,\quad \phi_{a}(\xi_{\varepsilon})\to\phi_{a}(\xi_{0})=0.
\end{equation}
In view of \eqref{convergence-w}, we further decompose the remain term $w_{\varepsilon}$ as follows 
\begin{equation*}
    w_{\varepsilon}=4\pi\bar{C}_{\alpha}\lambda_{\varepsilon}^{-\frac{1}{2}}(H_{a}(\xi_{\varepsilon},\cdot)-H_{0}(\xi_{\varepsilon},\cdot))+q_{\varepsilon},
\end{equation*}
with
\begin{equation*}
    q_{\varepsilon}=s_{\varepsilon}+r_{\varepsilon},\quad s_{\varepsilon}\in T_{\xi,\lambda_{\varepsilon}},\quad r_{\varepsilon}\in T_{\xi,\lambda_{\varepsilon}}^{\perp}.
\end{equation*}
Moreover, we have
\begin{equation}\label{eq-expansion-u}
\begin{aligned}
     u_{\varepsilon}&=\mu_{\varepsilon}(P\bar{U}_{\xi_{\varepsilon},\lambda_{\varepsilon}}+4\pi\bar{C}_{\alpha}\lambda_{\varepsilon}^{-\frac{1}{2}}(H_{a}(\xi_{\varepsilon},\cdot)-H_{0}(\xi_{\varepsilon},\cdot))+q_{\varepsilon})\\
     &=\mu_{\varepsilon}(\psi_{\xi_{\varepsilon},\lambda_{\varepsilon}}+s_{\varepsilon}+r_{\varepsilon}).
\end{aligned}
\end{equation}
Since $w_{\varepsilon}\in T_{\xi_{\varepsilon},\lambda_{\varepsilon}}^{\perp}$, it follows that
\begin{equation*}
    s_{\varepsilon}=4\pi\bar{C}_{\alpha}\lambda_{\varepsilon}^{-\frac{1}{2}}\Pi_{\xi_{\varepsilon},\lambda_{\varepsilon}}(H_{a}(\xi_{\varepsilon},\cdot)-H_{0}(\xi_{\varepsilon},\cdot)).
\end{equation*}

\begin{Lem}\label{estimate-s}
Let
\begin{equation}\label{expansion-s}
    s_{\varepsilon}:=\lambda^{-1}_{\varepsilon}\beta P\bar{U}_{\xi_{\varepsilon},\lambda_{\varepsilon}}+\gamma\partial_{\lambda}P\bar{U}_{\xi_{\varepsilon},\lambda_{\varepsilon}}+\lambda_{\varepsilon}^{-3}\sum_{i=1}^{3}\delta_{i}\partial_{\xi_{i}}P\bar{U}_{\xi_{\varepsilon},\lambda_{\varepsilon}}.
\end{equation}
As $\varepsilon\to0$, the following estimates hold
    \begin{equation*}
        \beta,\gamma,\delta_{i}=O(1)
    \end{equation*}
    \begin{equation*}
        \beta=\frac{64}{3}(\phi_{a}(\xi_{\varepsilon})-\phi_{0}(\xi_{\varepsilon}))+O(\lambda_{\varepsilon}^{-1})
    \end{equation*}
   and
    \begin{equation*}
        \|\nabla s_{\varepsilon}\|_{L^{2}(\Omega)}=O(\lambda_{\varepsilon}^{-1})\text{~~and~~}\|s_{\varepsilon}\|_{L^{2}(\Omega)}=O(\lambda_{\varepsilon}^{-\frac{3}{2}}).
    \end{equation*}
\end{Lem}

\begin{proof}
The proof is similar to those in \cite[Lemma 6.1]{Frank2019Energythree-dimensional} and \cite[Propositions 3.2 and 3.3]{Frank2021BlowupOS}, and we omit the details.
\end{proof}

For any $u,v\in H^{1}_{0}(\Omega)$, we define 
\begin{equation*}
    E_{\varepsilon}[u,v]:=\int_{\Omega}\nabla u\cdot\nabla vdx+\int_{\Omega}(a+\varepsilon V)uvdx
\end{equation*}
and
\begin{equation*}
    \begin{aligned}
        F[u,v]&:=2(6-\alpha)\int_{\Omega}\int_{\Omega}\frac{u^{6-\alpha}(y)u^{5-\alpha}(x)v(x)}{|x-y|^{\alpha}}dydx\\
        &\quad+(6-\alpha)^{2}\int_{\Omega}\int_{\Omega}\frac{u^{5-\alpha}(y)v(y)u^{5-\alpha}(x)v(x)}{|x-y|^{\alpha}}dydx\\
        &\quad+(6-\alpha)(5-\alpha)\int_{\Omega}\int_{\Omega}\frac{u^{6-\alpha}(y)u^{4-\alpha}(x)v^{2}(x)}{|x-y|^{\alpha}}dydx.
    \end{aligned}
\end{equation*}

\begin{Lem}\label{lem-energy-expansion-2}
As $\varepsilon\to0$, it holds that
    \begin{equation}\label{energy-expansion-2}
    \begin{aligned}
        &S_{HL}(a+\varepsilon V)[u_{\varepsilon}]\\
        &=S_{HL}(a+\varepsilon V)[\psi_{\xi_{\varepsilon},\lambda_{\varepsilon}}]+D_{0}^{-\frac{1}{6-\alpha}}\left(N_{1}-\frac{1}{6-\alpha}\frac{N_{0}D_{1}}
        {D_{0}}-\frac{1}{6-\alpha}\frac{N_{1}D_{1}}
        {D_{0}}+\frac{7-\alpha}{2(6-\alpha)^{2}}\frac{N_{0}D_{1}^{2}}{D_{0}^{2}}\right)\\
        &\quad+D_{0}^{-\frac{1}{6-\alpha}}\left(E_{0}[r_{\varepsilon}]-\frac{1}{6-\alpha}\frac{N_{0}}{D_{0}}I[r_{\varepsilon}]+o(\|\nabla r_{\varepsilon}\|_{L^{2}(\Omega)}^{2})\right)+o(\lambda_{\varepsilon}^{-2})+o(\varepsilon\lambda_{\varepsilon}^{-1}),
    \end{aligned}
\end{equation}
where
\begin{equation*}
    \begin{aligned}
        N_{0}:=E_{\varepsilon}[\psi_{\xi_{\varepsilon},\lambda_{\varepsilon}},\psi_{\xi_{\varepsilon},\lambda_{\varepsilon}}],\quad N_{1}:=2E_{0}[\psi_{\xi_{\varepsilon},\lambda_{\varepsilon}},s_{\varepsilon}]+\|\nabla s_{\varepsilon}\|_{L^{2}(\Omega)}^{2},\quad E_{0}[r_{\varepsilon}]:=E_{0}[r_{\varepsilon},r_{\varepsilon}]
    \end{aligned}
\end{equation*}
and
\begin{equation*}
    D_{0}:=\int_{\Omega}\int_{\Omega}\frac{\psi_{\xi_{\varepsilon},\lambda_{\varepsilon}}^{6-\alpha}(y)\psi_{\xi_{\varepsilon},\lambda_{\varepsilon}}^{6-\alpha}(x)}{|x-y|^{\alpha}}dydx,\quad D_{1}:=F[\psi_{\xi_{\varepsilon},\lambda_{\varepsilon}},s_{\varepsilon}],\quad I[r_{\varepsilon}]:=F[\psi_{\xi_{\varepsilon},\lambda_{\varepsilon}},r_{\varepsilon}].
\end{equation*}
\end{Lem}

\begin{proof}
First, it follows from \eqref{eq-expansion-u}, Lemma \ref{lem inequality 2}, Lemma \ref{estimate-s}, the HLS inequality and the Sobolev embedding theorem that
\begin{equation*}
    \begin{aligned}
        &\mu_{\varepsilon}^{-2(6-\alpha)}\int_{\Omega}\int_{\Omega}\frac{u_{\varepsilon}^{6-\alpha}(y)u_{\varepsilon}^{6-\alpha}(x)}{|x-y|^{\alpha}}dydx\\
        &=D_{0}+D_{1}+I[r_{\varepsilon}]+O\left\{\int_{\Omega}\int_{\Omega}\frac{\psi_{\xi_{\varepsilon},\lambda_{\varepsilon}}^{6-\alpha}(y)\psi_{\xi_{\varepsilon},\lambda_{\varepsilon}}^{4-\alpha}(x)r_{\varepsilon}(x)s_{\varepsilon}(x)}{|x-y|^{\alpha}}dydx\right.\\
        &\quad+\left.\int_{\Omega}\int_{\Omega}\frac{\psi_{\xi_{\varepsilon},\lambda_{\varepsilon}}^{5-\alpha}(y)r_{\varepsilon}(y)\psi_{\xi_{\varepsilon},\lambda_{\varepsilon}}^{5-\alpha}(x)s_{\varepsilon}(x)}{|x-y|^{\alpha}}dydx\right\}\\
        &\quad+o(\lambda_{\varepsilon}^{-2})+o(\|\nabla r_{\varepsilon}\|_{L^{2}(\Omega)}^{2}).
    \end{aligned}
\end{equation*}
Moreover, by \eqref{important-identity-1}, \eqref{eq-coefficient-convergence}, Lemma \ref{estimate-s}, the HLS inequality and the Sobolev embedding theorem, we have
\begin{equation*}
    \begin{aligned}
        &\int_{\Omega}\int_{\Omega}\frac{\psi_{\xi_{\varepsilon},\lambda_{\varepsilon}}^{6-\alpha}(y)\psi_{\xi_{\varepsilon},\lambda_{\varepsilon}}^{4-\alpha}(x)r_{\varepsilon}(x)s_{\varepsilon}(x)}{|x-y|^{\alpha}}dydx\\
        &=\int_{\Omega}\int_{\Omega}\frac{\bar{U}_{\xi_{\varepsilon},\lambda_{\varepsilon}}^{6-\alpha}(y)\bar{U}_{\xi_{\varepsilon},\lambda_{\varepsilon}}^{4-\alpha}(x)r_{\varepsilon}(x)s_{\varepsilon}(x)}{|x-y|^{\alpha}}dydx+O\left(\lambda_{\varepsilon}^{-\frac{1}{2}}\|\nabla r_{\varepsilon}\|_{L^{2}(\Omega)}\|\nabla s_{\varepsilon}\|_{L^{2}(\Omega)}\right)\\
        &=\int_{\Omega}\int_{\R^{3}}\frac{\bar{U}_{\xi_{\varepsilon},\lambda_{\varepsilon}}^{6-\alpha}(y)\bar{U}_{\xi_{\varepsilon},\lambda_{\varepsilon}}^{4-\alpha}(x)r_{\varepsilon}(x)s_{\varepsilon}(x)}{|x-y|^{\alpha}}dydx-\int_{\Omega}\int_{\R^{3}\setminus\Omega}\frac{\bar{U}_{\xi_{\varepsilon},\lambda_{\varepsilon}}^{6-\alpha}(y)\bar{U}_{\xi_{\varepsilon},\lambda_{\varepsilon}}^{4-\alpha}(x)r_{\varepsilon}(x)s_{\varepsilon}(x)}{|x-y|^{\alpha}}dydx\\
        &\quad+O\left(\lambda_{\varepsilon}^{-\frac{1}{2}}\|\nabla r_{\varepsilon}\|_{L^{2}(\Omega)}\|\nabla s_{\varepsilon}\|_{L^{2}(\Omega)}\right)\\
        &=3\int_{\Omega}{U}_{\xi_{\varepsilon},\lambda_{\varepsilon}}^{4}r_{\varepsilon}s_{\varepsilon}dx+o(\lambda_{\varepsilon}^{-2})+o(\|\nabla r_{\varepsilon}\|_{L^{2}(\Omega)}^{2}).
    \end{aligned}
\end{equation*}
Using the expansion \eqref{expansion-s} and Lemma \ref{estimate-s}, we obtain
\begin{equation*}
    \begin{aligned}
        \int_{\Omega}{U}_{\xi_{\varepsilon},\lambda_{\varepsilon}}^{4}r_{\varepsilon}s_{\varepsilon}dx&=\lambda^{-1}_{\varepsilon}\beta \int_{\Omega}{U}_{\xi_{\varepsilon},\lambda_{\varepsilon}}^{4}P\bar{U}_{\xi_{\varepsilon},\lambda_{\varepsilon}}r_{\varepsilon}dx+\gamma\int_{\Omega}{U}_{\xi_{\varepsilon},\lambda_{\varepsilon}}^{4}\partial_{\lambda}P\bar{U}_{\xi_{\varepsilon},\lambda_{\varepsilon}}r_{\varepsilon}dx\\
        &\quad+\lambda_{\varepsilon}^{-3}\sum_{i=1}^{3}\delta_{i}\int_{\Omega}{U}_{\xi_{\varepsilon},\lambda_{\varepsilon}}^{4}\partial_{\xi_{i}}P\bar{U}_{\xi_{\varepsilon},\lambda_{\varepsilon}}r_{\varepsilon}dx.
    \end{aligned}
\end{equation*}
It then follows from the orthogonality of $r_{\varepsilon}$, the H\"{o}lder inequality and Lemma \ref{lemma PU} that
\begin{equation*}
    \begin{aligned}
        &\int_{\Omega}{U}_{\xi_{\varepsilon},\lambda_{\varepsilon}}^{4}r_{\varepsilon}s_{\varepsilon}dx\\
        &=\lambda^{-1}_{\varepsilon}\bar{C}_{\alpha}\beta \int_{\Omega}{U}_{\xi_{\varepsilon},\lambda_{\varepsilon}}^{5}r_{\varepsilon}dx-\lambda^{-1}_{\varepsilon}\bar{C}_{\alpha}\beta \int_{\Omega}{U}_{\xi_{\varepsilon},\lambda_{\varepsilon}}^{4}\varphi_{\xi_{\varepsilon},\lambda_{\varepsilon}}r_{\varepsilon}dx\\
      &\quad+\bar{C}_{\alpha}\gamma \int_{\Omega}{U}_{\xi_{\varepsilon},\lambda_{\varepsilon}}^{4}\partial_{\lambda}{U}_{\xi_{\varepsilon},\lambda_{\varepsilon}}r_{\varepsilon}dx-\bar{C}_{\alpha}\gamma \int_{\Omega}{U}_{\xi_{\varepsilon},\lambda_{\varepsilon}}^{4}\partial_{\lambda}\varphi_{\xi_{\varepsilon},\lambda_{\varepsilon}}r_{\varepsilon}dx\\
      &\quad+\lambda_{\varepsilon}^{-3}\bar{C}_{\alpha}\sum_{i=1}^{3}\delta_{i}\int_{\Omega}{U}_{\xi_{\varepsilon},\lambda_{\varepsilon}}^{4}\partial_{\xi_{i}}{U}_{\xi_{\varepsilon},\lambda_{\varepsilon}}r_{\varepsilon}dx-\lambda_{\varepsilon}^{-3}\bar{C}_{\alpha}\sum_{i=1}^{3}\delta_{i}\int_{\Omega}{U}_{\xi_{\varepsilon},\lambda_{\varepsilon}}^{4}\partial_{\xi_{i}}{\varphi}_{\xi_{\varepsilon},\lambda_{\varepsilon}}r_{\varepsilon}dx\\
      &=o(\lambda_{\varepsilon}^{-2}).
    \end{aligned}
\end{equation*}
On the other hand, from the expansion \eqref{expansion-s}, Lemma \ref{estimate-s}, the HLS inequality and the Sobolev embedding theorem, it follows that
\begin{equation*}
    \begin{aligned}
        &\int_{\Omega}\int_{\Omega}\frac{\psi_{\xi_{\varepsilon},\lambda_{\varepsilon}}^{5-\alpha}(y)r_{\varepsilon}(y)\psi_{\xi_{\varepsilon},\lambda_{\varepsilon}}^{5-\alpha}(x)s_{\varepsilon}(x)}{|x-y|^{\alpha}}dydx\\
        &=\int_{\Omega}\int_{\Omega}\frac{\bar{U}_{\xi_{\varepsilon},\lambda_{\varepsilon}}^{5-\alpha}(y)r_{\varepsilon}(y)\bar{U}_{\xi_{\varepsilon},\lambda_{\varepsilon}}^{5-\alpha}(x)s_{\varepsilon}(x)}{|x-y|^{\alpha}}dydx+O\left(\lambda_{\varepsilon}^{-\frac{1}{2}}\|\nabla r_{\varepsilon}\|_{L^{2}(\Omega)}\|\nabla s_{\varepsilon}\|_{L^{2}(\Omega)}\right)\\
        &=\lambda^{-1}_{\varepsilon}\beta \int_{\Omega}\int_{\Omega}\frac{\bar{U}_{\xi_{\varepsilon},\lambda_{\varepsilon}}^{5-\alpha}(y)r_{\varepsilon}(y)\bar{U}_{\xi_{\varepsilon},\lambda_{\varepsilon}}^{5-\alpha}(x)P\bar{U}_{\xi_{\varepsilon},\lambda_{\varepsilon}}(x)}{|x-y|^{\alpha}}dydx\\
         &\quad+\gamma\int_{\Omega}\int_{\Omega}\frac{\bar{U}_{\xi_{\varepsilon},\lambda_{\varepsilon}}^{5-\alpha}(y)r_{\varepsilon}(y)\bar{U}_{\xi_{\varepsilon},\lambda_{\varepsilon}}^{5-\alpha}(x)\partial_{\lambda}P\bar{U}_{\xi_{\varepsilon},\lambda_{\varepsilon}}(x)}{|x-y|^{\alpha}}dydx\\
          &\quad+\lambda_{\varepsilon}^{-3}\sum_{i=1}^{3}\delta_{i}\int_{\Omega}\int_{\Omega}\frac{\bar{U}_{\xi_{\varepsilon},\lambda_{\varepsilon}}^{5-\alpha}(y)r_{\varepsilon}(y)\bar{U}_{\xi_{\varepsilon},\lambda_{\varepsilon}}^{5-\alpha}(x)\partial_{\xi_{i}}P\bar{U}_{\xi_{\varepsilon},\lambda_{\varepsilon}}(x)}{|x-y|^{\alpha}}dydx\\
          &\quad+o(\lambda_{\varepsilon}^{-2})+o(\|\nabla r_{\varepsilon}\|_{L^{2}(\Omega)}^{2}). 
    \end{aligned}
\end{equation*}
Using \eqref{important-identity-1}, Lemma \ref{lemma PU}, Lemma \ref{estimate-s} and the orthogonality of $r_{\varepsilon}$, we obtain
\begin{equation*}
    \begin{aligned}
        &\lambda^{-1}_{\varepsilon}\beta \int_{\Omega}\int_{\Omega}\frac{\bar{U}_{\xi_{\varepsilon},\lambda_{\varepsilon}}^{5-\alpha}(y)r_{\varepsilon}(y)\bar{U}_{\xi_{\varepsilon},\lambda_{\varepsilon}}^{5-\alpha}(x)P\bar{U}_{\xi_{\varepsilon},\lambda_{\varepsilon}}(x)}{|x-y|^{\alpha}}dydx\\
        &=\lambda^{-1}_{\varepsilon}\beta \int_{\R^{3}}\int_{\Omega}\frac{\bar{U}_{\xi_{\varepsilon},\lambda_{\varepsilon}}^{5-\alpha}(y)r_{\varepsilon}(y)\bar{U}_{\xi_{\varepsilon},\lambda_{\varepsilon}}^{6-\alpha}(x)}{|x-y|^{\alpha}}dydx-\lambda^{-1}_{\varepsilon}\beta \int_{\R^{3}\setminus\Omega}\int_{\Omega}\frac{\bar{U}_{\xi_{\varepsilon},\lambda_{\varepsilon}}^{5-\alpha}(y)r_{\varepsilon}(y)\bar{U}_{\xi_{\varepsilon},\lambda_{\varepsilon}}^{6-\alpha}(x)}{|x-y|^{\alpha}}dydx\\
        &\quad-\lambda^{-1}_{\varepsilon}\bar{C}_{\alpha}\beta \int_{\Omega}\int_{\Omega}\frac{\bar{U}_{\xi_{\varepsilon},\lambda_{\varepsilon}}^{5-\alpha}(y)r_{\varepsilon}(y)\bar{U}_{\xi_{\varepsilon},\lambda_{\varepsilon}}^{5-\alpha}(x){\varphi}_{\xi_{\varepsilon},\lambda_{\varepsilon}}(x)}{|x-y|^{\alpha}}dydx\\
        &\lesssim \lambda_{\varepsilon}^{-1}\int_{\Omega}{U}_{\xi_{\varepsilon},\lambda_{\varepsilon}}^{5}r_{\varepsilon}dy+\lambda_{\varepsilon}^{-1}\|\nabla r_{\varepsilon}\|_{{L^{2}}(\Omega)}\|{U}_{\xi_{\varepsilon},\lambda_{\varepsilon}}\|_{L^{6}(\R^{3}\setminus\Omega)}^{6-\alpha}\\
        &\quad+\lambda_{\varepsilon}^{-1}\|{\varphi}_{\xi_{\varepsilon},\lambda_{\varepsilon}}\|_{L^{\infty}(\Omega)}\|\nabla r_{\varepsilon}\|_{{L^{2}}(\Omega)}\|{U}_{\xi_{\varepsilon},\lambda_{\varepsilon}}\|_{L^{\frac{6(5-\alpha)}{6-\alpha}}(\Omega)}^{5-\alpha}\\
         &=o(\lambda_{\varepsilon}^{-2})
    \end{aligned}
\end{equation*}

\begin{equation*}
    \begin{aligned}
       & \gamma\int_{\Omega}\int_{\Omega}\frac{\bar{U}_{\xi_{\varepsilon},\lambda_{\varepsilon}}^{5-\alpha}(y)r_{\varepsilon}(y)\bar{U}_{\xi_{\varepsilon},\lambda_{\varepsilon}}^{5-\alpha}(x)\partial_{\lambda}P\bar{U}_{\xi_{\varepsilon},\lambda_{\varepsilon}}(x)}{|x-y|^{\alpha}}dydx\\
       &=\gamma\int_{\R^{3}}\int_{\Omega}\frac{\bar{U}_{\xi_{\varepsilon},\lambda_{\varepsilon}}^{5-\alpha}(y)r_{\varepsilon}(y)\bar{U}_{\xi_{\varepsilon},\lambda_{\varepsilon}}^{5-\alpha}(x)\partial_{\lambda}\bar{U}_{\xi_{\varepsilon},\lambda_{\varepsilon}}(x)}{|x-y|^{\alpha}}dydx\\
       &\quad-\gamma\int_{\R^{3}\setminus\Omega}\int_{\Omega}\frac{\bar{U}_{\xi_{\varepsilon},\lambda_{\varepsilon}}^{5-\alpha}(y)r_{\varepsilon}(y)\bar{U}_{\xi_{\varepsilon},\lambda_{\varepsilon}}^{5-\alpha}(x)\partial_{\lambda}\bar{U}_{\xi_{\varepsilon},\lambda_{\varepsilon}}(x)}{|x-y|^{\alpha}}dydx\\
       &\quad-\gamma\bar{C}_{\alpha}\int_{\Omega}\int_{\Omega}\frac{\bar{U}_{\xi_{\varepsilon},\lambda_{\varepsilon}}^{5-\alpha}(y)r_{\varepsilon}(y)\bar{U}_{\xi_{\varepsilon},\lambda_{\varepsilon}}^{5-\alpha}(x)\partial_{\lambda}{\varphi}_{\xi_{\varepsilon},\lambda_{\varepsilon}}(x)}{|x-y|^{\alpha}}dydx\\
       &\lesssim\lambda_{\varepsilon}^{-1}\|\nabla r_{\varepsilon}\|_{{L^{2}}(\Omega)}\|{U}_{\xi_{\varepsilon},\lambda_{\varepsilon}}\|_{L^{6}(\R^{3}\setminus\Omega)}^{6-\alpha}+\|\partial_{\lambda}{\varphi}_{\xi_{\varepsilon},\lambda_{\varepsilon}}\|_{L^{\infty}(\Omega)}\|\nabla r_{\varepsilon}\|_{{L^{2}}(\Omega)}\|{U}_{\xi_{\varepsilon},\lambda_{\varepsilon}}\|_{L^{\frac{6(5-\alpha)}{6-\alpha}}(\Omega)}^{5-\alpha}\\
         &=o(\lambda_{\varepsilon}^{-2})
    \end{aligned}
\end{equation*}
and
\begin{equation*}
    \begin{aligned}
        &\lambda_{\varepsilon}^{-3}\sum_{i=1}^{3}\delta_{i}\int_{\Omega}\int_{\Omega}\frac{\bar{U}_{\xi_{\varepsilon},\lambda_{\varepsilon}}^{5-\alpha}(y)r_{\varepsilon}(y)\bar{U}_{\xi_{\varepsilon},\lambda_{\varepsilon}}^{5-\alpha}(x)\partial_{\xi_{i}}P\bar{U}_{\xi_{\varepsilon},\lambda_{\varepsilon}}(x)}{|x-y|^{\alpha}}dydx\\
        &=\lambda_{\varepsilon}^{-3}\sum_{i=1}^{3}\delta_{i}\int_{\R^{3}}\int_{\Omega}\frac{\bar{U}_{\xi_{\varepsilon},\lambda_{\varepsilon}}^{5-\alpha}(y)r_{\varepsilon}(y)\bar{U}_{\xi_{\varepsilon},\lambda_{\varepsilon}}^{5-\alpha}(x)\partial_{\xi_{i}}\bar{U}_{\xi_{\varepsilon},\lambda_{\varepsilon}}(x)}{|x-y|^{\alpha}}dydx\\
        &\quad-\lambda_{\varepsilon}^{-3}\sum_{i=1}^{3}\delta_{i}\int_{\R^{3}\setminus\Omega}\int_{\Omega}\frac{\bar{U}_{\xi_{\varepsilon},\lambda_{\varepsilon}}^{5-\alpha}(y)r_{\varepsilon}(y)\bar{U}_{\xi_{\varepsilon},\lambda_{\varepsilon}}^{5-\alpha}(x)\partial_{\xi_{i}}\bar{U}_{\xi_{\varepsilon},\lambda_{\varepsilon}}(x)}{|x-y|^{\alpha}}dydx\\
        &\quad-\lambda_{\varepsilon}^{-3}\bar{C}_{\alpha}\sum_{i=1}^{3}\delta_{i}\int_{\Omega}\int_{\Omega}\frac{\bar{U}_{\xi_{\varepsilon},\lambda_{\varepsilon}}^{5-\alpha}(y)r_{\varepsilon}(y)\bar{U}_{\xi_{\varepsilon},\lambda_{\varepsilon}}^{5-\alpha}(x)\partial_{\xi_{i}}{\varphi}_{\xi_{\varepsilon},\lambda_{\varepsilon}}(x)}{|x-y|^{\alpha}}dydx\\
         &\lesssim\lambda_{\varepsilon}^{-2}\|\nabla r_{\varepsilon}\|_{{L^{2}}(\Omega)}\|{U}_{\xi_{\varepsilon},\lambda_{\varepsilon}}\|_{L^{6}(\R^{3}\setminus\Omega)}^{6-\alpha}+\lambda_{\varepsilon}^{-3}\|\partial_{\xi}{\varphi}_{\xi_{\varepsilon},\lambda_{\varepsilon}}\|_{L^{\infty}(\Omega)}\|\nabla r_{\varepsilon}\|_{{L^{2}}(\Omega)}\|{U}_{\xi_{\varepsilon},\lambda_{\varepsilon}}\|_{L^{\frac{6(5-\alpha)}{6-\alpha}}(\Omega)}^{5-\alpha}\\
         &=o(\lambda_{\varepsilon}^{-2}).
    \end{aligned}
\end{equation*}
Combining all the estimates above, we conclude that
\begin{equation*}
    \begin{aligned}
        F_{0}&:=\mu_{\varepsilon}^{-2(6-\alpha)}\int_{\Omega}\int_{\Omega}\frac{u_{\varepsilon}^{6-\alpha}(y)u_{\varepsilon}^{6-\alpha}(x)}{|x-y|^{\alpha}}dydx\\
        &=D_{0}+D_{1}+I[r_{\varepsilon}]+o(\lambda_{\varepsilon}^{-2})+o(\|\nabla r_{\varepsilon}\|_{L^{2}(\Omega)}^{2}).
    \end{aligned}
\end{equation*}       
This estimate, together with the Taylor's expansion, yields that
\begin{equation*}
    \begin{aligned}
         F_{0}^{-\frac{1}{6-\alpha}}&=D_{0}^{-\frac{1}{6-\alpha}}\left(1-\frac{1}{6-\alpha}\frac{D_{1}+I[r_{\varepsilon}]}{D_{0}}+\frac{7-\alpha}{2(6-\alpha)^{2}}\frac{(D_{1}+I[r_{\varepsilon}])^{2}}{D_{0}^{2}}\right)\\
         &\quad+o(\lambda_{\varepsilon}^{-2})+o(\|\nabla r_{\varepsilon}\|_{L^{2}(\Omega)}^{2}).
    \end{aligned}
\end{equation*}
Notice that by employing the HLS inequality, the H\"{o}lder inequality, and the Sobolev embedding theorem, along with the orthogonality of $r_{\varepsilon}$, we obtain
\begin{equation*}
    \begin{aligned}
        D_{1}\lesssim \|\nabla s_{\varepsilon}\|_{L^{2}(\Omega)}=O(\lambda_{\varepsilon}^{-1})
    \end{aligned}
\end{equation*}
and
\begin{equation*}
    \begin{aligned}
        I[r_{\varepsilon}]&=\int_{\Omega}\int_{\Omega}\frac{\bar{U}_{\xi_{\varepsilon},\lambda_{\varepsilon}}^{6-\alpha}(y)\bar{U}_{\xi_{\varepsilon},\lambda_{\varepsilon}}^{5-\alpha}(x)r_{\varepsilon}(x)}{|x-y|^{\alpha}}dydx+O(\lambda_{\varepsilon}^{-\frac{1}{2}}\|\nabla r_{\varepsilon}\|_{L^{2}(\Omega)})\\
        &=\int_{\Omega}\int_{\R^{3}}\frac{\bar{U}_{\xi_{\varepsilon},\lambda_{\varepsilon}}^{6-\alpha}(y)\bar{U}_{\xi_{\varepsilon},\lambda_{\varepsilon}}^{5-\alpha}(x)r_{\varepsilon}(x)}{|x-y|^{\alpha}}dydx-\int_{\Omega}\int_{\R^{3}\setminus\Omega}\frac{\bar{U}_{\xi_{\varepsilon},\lambda_{\varepsilon}}^{6-\alpha}(y)\bar{U}_{\xi_{\varepsilon},\lambda_{\varepsilon}}^{5-\alpha}(x)r_{\varepsilon}(x)}{|x-y|^{\alpha}}dydx\\
        &\quad+O(\lambda_{\varepsilon}^{-\frac{1}{2}}\|\nabla r_{\varepsilon}\|_{L^{2}(\Omega)})\\
        &=3\bar{C}_{\alpha} \int_{\Omega}{U}_{\xi_{\varepsilon},\lambda_{\varepsilon}}^{5}r_{\varepsilon}dx+O(\lambda_{\varepsilon}^{-\frac{1}{2}}\|\nabla r_{\varepsilon}\|_{L^{2}(\Omega)})+O(\|{U}_{\xi_{\varepsilon},\lambda_{\varepsilon}}\|^{6-\alpha}_{L^{6}(\R^{3}\setminus\Omega)}\|\nabla r_{\varepsilon}\|_{L^{2}(\Omega)})\\
        &=O(\lambda_{\varepsilon}^{-\frac{1}{2}}\|\nabla r_{\varepsilon}\|_{L^{2}(\Omega)}).
    \end{aligned}
\end{equation*}
It then follows that
\begin{equation}\label{proof-expan-2-63}
    \begin{aligned}
          F_{0}^{-\frac{1}{6-\alpha}}&=D_{0}^{-\frac{1}{6-\alpha}}\left(1-\frac{1}{6-\alpha}\frac{D_{1}+I[r_{\varepsilon}]}{D_{0}}+\frac{7-\alpha}{2(6-\alpha)^{2}}\frac{D_{1}^{2}}{D_{0}^{2}}\right)\\
         &\quad+o(\lambda_{\varepsilon}^{-2})+o(\|\nabla r_{\varepsilon}\|_{L^{2}(\Omega)}^{2}).
    \end{aligned}
\end{equation}
On the other hand, by an argument similar to that in \cite[Lemma 6.4]{Frank2019Energythree-dimensional}, we have
\begin{equation}\label{proof-expan-2-64}
    \begin{aligned}
        &\mu_{\varepsilon}^{-2}E_{\varepsilon}[u_{\varepsilon},u_{\varepsilon}]\\
        &=E_{\varepsilon}[\psi_{\xi_{\varepsilon},\lambda_{\varepsilon}},\psi_{\xi_{\varepsilon},\lambda_{\varepsilon}}]+(2E_{0}[\psi_{\xi_{\varepsilon},\lambda_{\varepsilon}},s_{\varepsilon}]+\|\nabla s_{\varepsilon}\|_{L^{2}(\Omega)}^{2})+E_{0}[r_{\varepsilon},r_{\varepsilon}]+o(\lambda_{\varepsilon}^{-2})+o(\varepsilon\lambda_{\varepsilon}^{-1})\\
        &=N_{0}+N_{1}+E_{0}[r_{\varepsilon}]+o(\lambda_{\varepsilon}^{-2})+o(\varepsilon\lambda_{\varepsilon}^{-1}).
    \end{aligned}
\end{equation}
Combining \eqref{proof-expan-2-63} and \eqref{proof-expan-2-64}, we conclude that
\begin{equation*}
    \begin{aligned}
        &S_{HL}(a+\varepsilon V)[u_{\varepsilon}]\\
        &=N_{0}D_{0}^{-\frac{1}{6-\alpha}}+D_{0}^{-\frac{1}{6-\alpha}}\left(N_{1}-\frac{1}{6-\alpha}\frac{N_{0}D_{1}}
        {D_{0}}-\frac{1}{6-\alpha}\frac{N_{1}D_{1}}
        {D_{0}}+\frac{7-\alpha}{2(6-\alpha)^{2}}\frac{N_{0}D_{1}^{2}}{D_{0}^{2}}\right)\\
        &\quad+D_{0}^{-\frac{1}{6-\alpha}}\left(E_{0}[r_{\varepsilon}]-\frac{1}{6-\alpha}\frac{N_{0}}{D_{0}}I[r_{\varepsilon}]+o(\|\nabla r_{\varepsilon}\|_{L^{2}(\Omega)}^{2})\right)+o(\lambda_{\varepsilon}^{-2})+o(\varepsilon\lambda_{\varepsilon}^{-1}).
    \end{aligned}
\end{equation*}
Notice that $N_{0}D_{0}^{-\frac{1}{6-\alpha}}=S_{HL}(a+\varepsilon V)[\psi_{\xi_{\varepsilon},\lambda_{\varepsilon}}]$. Therefore, the expansion \eqref{energy-expansion-2} holds.
\end{proof}

\begin{Prop}\label{prop-lower-bound}
As $\varepsilon\to0$, it holds that 
   \begin{equation*}
    \begin{aligned}
        &S_{HL}(a+\varepsilon V)[u_{\varepsilon}]\\
        &=S_{HL}(a+\varepsilon V)[\psi_{\xi_{\varepsilon},\lambda_{\varepsilon}}]+S_{HL}^{-\frac{1}{5-\alpha}}\left(E_{0}[r_{\varepsilon}]-\frac{1}{6-\alpha}\frac{N_{0}}{D_{0}}I[r_{\varepsilon}]+o(\|\nabla r_{\varepsilon}\|_{L^{2}(\Omega)}^{2})\right)\\
        &\quad+o(\lambda_{\varepsilon}^{-2})+o(\varepsilon\lambda_{\varepsilon}^{-1}).
    \end{aligned}
\end{equation*}
\end{Prop}

\begin{proof}
First, we need to establish a refined estimate for the term $D_{1}$
\begin{equation*}
    \begin{aligned}
        D_{1}&=2(6-\alpha)\int_{\Omega}\int_{\Omega}\frac{\psi_{\xi_{\varepsilon},\lambda_{\varepsilon}}^{6-\alpha}(y)\psi_{\xi_{\varepsilon},\lambda_{\varepsilon}}^{5-\alpha}(x)s_{\varepsilon}(x)}{|x-y|^{\alpha}}dydx\\
        &\quad+(6-\alpha)^{2}\int_{\Omega}\int_{\Omega}\frac{\psi_{\xi_{\varepsilon},\lambda_{\varepsilon}}^{5-\alpha}(y)s_{\varepsilon}(y)\psi_{\xi_{\varepsilon},\lambda_{\varepsilon}}^{5-\alpha}(x)s_{\varepsilon}(x)}{|x-y|^{\alpha}}dydx\\
        &\quad+(6-\alpha)(5-\alpha)\int_{\Omega}\int_{\Omega}\frac{\psi_{\xi_{\varepsilon},\lambda_{\varepsilon}}^{6-\alpha}(y)\psi_{\xi_{\varepsilon},\lambda_{\varepsilon}}^{4-\alpha}(x)s^{2}_{\varepsilon}(x)}{|x-y|^{\alpha}}dydx.
    \end{aligned}
\end{equation*}
Using the expansion of $s_{\varepsilon}$ (see \eqref{expansion-s}), we divide the first term in $D_{1}$ into three parts
\begin{equation}\label{eq-proof-expansion-2-2}
    \begin{aligned}
        &\int_{\Omega}\int_{\Omega}\frac{\psi_{\xi_{\varepsilon},\lambda_{\varepsilon}}^{6-\alpha}(y)\psi_{\xi_{\varepsilon},\lambda_{\varepsilon}}^{5-\alpha}(x)s_{\varepsilon}(x)}{|x-y|^{\alpha}}dydx\\
        &=\lambda^{-1}_{\varepsilon}\beta\int_{\Omega}\int_{\Omega}\frac{\psi_{\xi_{\varepsilon},\lambda_{\varepsilon}}^{6-\alpha}(y)\psi_{\xi_{\varepsilon},\lambda_{\varepsilon}}^{5-\alpha}(x)P\bar{U}_{\xi_{\varepsilon},\lambda_{\varepsilon}}(x)}{|x-y|^{\alpha}}dydx+\gamma\int_{\Omega}\int_{\Omega}\frac{\psi_{\xi_{\varepsilon},\lambda_{\varepsilon}}^{6-\alpha}(y)\psi_{\xi_{\varepsilon},\lambda_{\varepsilon}}^{5-\alpha}(x)\partial_{\lambda}P\bar{U}_{\xi_{\varepsilon},\lambda_{\varepsilon}}(x)}{|x-y|^{\alpha}}dydx\\
        &\quad+\lambda_{\varepsilon}^{-3}\sum_{i=1}^{3}\delta_{i}\int_{\Omega}\int_{\Omega}\frac{\psi_{\xi_{\varepsilon},\lambda_{\varepsilon}}^{6-\alpha}(y)\psi_{\xi_{\varepsilon},\lambda_{\varepsilon}}^{5-\alpha}(x)\partial_{\xi_{i}}P\bar{U}_{\xi_{\varepsilon},\lambda_{\varepsilon}}(x)}{|x-y|^{\alpha}}dydx
    \end{aligned}
\end{equation}
We now estimate the first term in \eqref{eq-proof-expansion-2-2}. First, Lemma \ref{lemma PU} gives that
\begin{equation*}
    \begin{aligned}
        &\lambda^{-1}_{\varepsilon}\beta\int_{\Omega}\int_{\Omega}\frac{\psi_{\xi_{\varepsilon},\lambda_{\varepsilon}}^{6-\alpha}(y)\psi_{\xi_{\varepsilon},\lambda_{\varepsilon}}^{5-\alpha}(x)P\bar{U}_{\xi_{\varepsilon},\lambda_{\varepsilon}}(x)}{|x-y|^{\alpha}}dydx\\
        &=\lambda^{-1}_{\varepsilon}\beta\int_{\Omega}\int_{\Omega}\frac{\psi_{\xi_{\varepsilon},\lambda_{\varepsilon}}^{6-\alpha}(y)\psi_{\xi_{\varepsilon},\lambda_{\varepsilon}}^{5-\alpha}(x)\bar{U}_{\xi_{\varepsilon},\lambda_{\varepsilon}}(x)}{|x-y|^{\alpha}}dydx-\lambda^{-1}_{\varepsilon}\bar{C}_{\alpha}\beta\int_{\Omega}\int_{\Omega}\frac{\psi_{\xi_{\varepsilon},\lambda_{\varepsilon}}^{6-\alpha}(y)\psi_{\xi_{\varepsilon},\lambda_{\varepsilon}}^{5-\alpha}(x){\varphi}_{\xi_{\varepsilon},\lambda_{\varepsilon}}(x)}{|x-y|^{\alpha}}dydx.
    \end{aligned}
\end{equation*}
Moreover, by \eqref{eq-psi-1}, Lemma \ref{lemma PU} and \cite[Lemma 2.5]{Frank2019Energythree-dimensional}, we obtain
\begin{equation*}
    \begin{aligned}
        &\lambda^{-1}_{\varepsilon}\beta\int_{\Omega}\int_{\Omega}\frac{\psi_{\xi_{\varepsilon},\lambda_{\varepsilon}}^{6-\alpha}(y)\psi_{\xi_{\varepsilon},\lambda_{\varepsilon}}^{5-\alpha}(x)\bar{U}_{\xi_{\varepsilon}\lambda_{\varepsilon}}(x)}{|x-y|^{\alpha}}dydx\\
        &=\lambda^{-1}_{\varepsilon}\beta\int_{\Omega}\int_{\Omega}\frac{\bar{U}_{\xi_{\varepsilon},\lambda_{\varepsilon}}^{6-\alpha}(y)\bar{U}_{\xi_{\varepsilon},\lambda_{\varepsilon}}^{6-\alpha}(x)}{|x-y|^{\alpha}}dydx+o(\lambda_{\varepsilon}^{-2})\\
        &\quad+(11-2\alpha)4\pi\bar{C}_{\alpha}\beta\lambda^{-\frac{3}{2}}_{\varepsilon}\int_{\Omega}\int_{\Omega}\frac{\bar{U}_{\xi_{\varepsilon},\lambda_{\varepsilon}}^{6-\alpha}(y)\bar{U}^{5-\alpha}_{\xi_{\varepsilon},\lambda_{\varepsilon}}(x)H_{a}(\xi_{\varepsilon},x)}{|x-y|^{\alpha}}dydx\\
        &=3\bar{C}_{\alpha}^{2}\beta\lambda^{-1}_{\varepsilon}\int_{\R^{3}}{U}_{\xi_{\varepsilon},\lambda_{\varepsilon}}^{6}(x)dx+o(\lambda_{\varepsilon}^{-2})\\
        &\quad+3(11-2\alpha)4\pi\bar{C}_{\alpha}^{2}\beta\lambda^{-\frac{3}{2}}_{\varepsilon}\int_{\Omega}{U}^{5}_{\xi_{\varepsilon},\lambda_{\varepsilon}}(x)H_{a}(\xi_{\varepsilon},x)dx\\
        &=S_{HL}^{\frac{6-\alpha}{5-\alpha}}\beta\lambda^{-1}_{\varepsilon}+\frac{64}{3}(11-2\alpha)S_{HL}^{\frac{6-\alpha}{5-\alpha}}\beta\phi_{a}(\xi_{\varepsilon})\lambda^{-2}_{\varepsilon}+o(\lambda_{\varepsilon}^{-2})
    \end{aligned}
\end{equation*}
and
\begin{equation*}
    \begin{aligned}
        &-\lambda^{-1}_{\varepsilon}\bar{C}_{\alpha}\beta\int_{\Omega}\int_{\Omega}\frac{\psi_{\xi_{\varepsilon},\lambda_{\varepsilon}}^{6-\alpha}(y)\psi_{\xi_{\varepsilon},\lambda_{\varepsilon}}^{5-\alpha}(x){\varphi}_{\xi_{\varepsilon},\lambda_{\varepsilon}}(x)}{|x-y|^{\alpha}}dydx\\
        &=4\pi\bar{C}_{\alpha}\beta\lambda_{\varepsilon}^{-\frac{3}{2}}\int_{\Omega}\int_{\Omega}\frac{\psi_{\xi_{\varepsilon},\lambda_{\varepsilon}}^{6-\alpha}(y)\psi_{\xi_{\varepsilon},\lambda_{\varepsilon}}^{5-\alpha}(x)H_{0}(\xi_{\varepsilon},x)}{|x-y|^{\alpha}}dydx+o(\lambda_{\varepsilon}^{-2})\\
        &=4\pi\bar{C}_{\alpha}\beta\lambda_{\varepsilon}^{-\frac{3}{2}}\int_{\Omega}\int_{\Omega}\frac{\bar{U}_{\xi_{\varepsilon},\lambda_{\varepsilon}}^{6-\alpha}(y)\bar{U}_{\xi_{\varepsilon},\lambda_{\varepsilon}}^{5-\alpha}(x)H_{0}(\xi_{\varepsilon},x)}{|x-y|^{\alpha}}dydx+o(\lambda_{\varepsilon}^{-2})\\
        &=4\pi\bar{C}_{\alpha}\beta\lambda_{\varepsilon}^{-\frac{3}{2}}\int_{\Omega}\int_{\R^{3}}\frac{\bar{U}_{\xi_{\varepsilon},\lambda_{\varepsilon}}^{6-\alpha}(y)\bar{U}_{\xi_{\varepsilon},\lambda_{\varepsilon}}^{5-\alpha}(x)H_{0}(\xi_{\varepsilon},x)}{|x-y|^{\alpha}}dydx+o(\lambda_{\varepsilon}^{-2})\\
        &=12\pi\bar{C}^{2}_{\alpha}\beta\lambda_{\varepsilon}^{-\frac{3}{2}}\int_{\Omega}{U}_{\xi_{\varepsilon},\lambda_{\varepsilon}}^{5}(x)H_{0}(\xi_{\varepsilon},x)dx+o(\lambda_{\varepsilon}^{-2})\\
        &=\frac{64}{3}S_{HL}^{\frac{6-\alpha}{5-\alpha}}\beta\phi_{0}(\xi_{\varepsilon})\lambda_{\varepsilon}^{-2}+o(\lambda_{\varepsilon}^{-2}).
    \end{aligned}
\end{equation*}
Thus the first term in \eqref{eq-proof-expansion-2-2} satisfies
\begin{equation}\label{eq-proof-expansion-2-76}
    \begin{aligned}
        &\lambda^{-1}_{\varepsilon}\beta\int_{\Omega}\int_{\Omega}\frac{\psi_{\xi_{\varepsilon},\lambda_{\varepsilon}}^{6-\alpha}(y)\psi_{\xi_{\varepsilon},\lambda_{\varepsilon}}^{5-\alpha}(x)P\bar{U}_{\xi_{\varepsilon},\lambda_{\varepsilon}}(x)}{|x-y|^{\alpha}}dydx\\
        &=S_{HL}^{\frac{6-\alpha}{5-\alpha}}\beta\lambda^{-1}_{\varepsilon}+\frac{64}{3}(11-2\alpha)S_{HL}^{\frac{6-\alpha}{5-\alpha}}\beta\phi_{a}(\xi_{\varepsilon})\lambda^{-2}_{\varepsilon}+\frac{64}{3}S_{HL}^{\frac{6-\alpha}{5-\alpha}}\beta\phi_{0}(\xi_{\varepsilon})\lambda_{\varepsilon}^{-2}+o(\lambda_{\varepsilon}^{-2}).
    \end{aligned}
\end{equation}
For the second term in \eqref{eq-proof-expansion-2-2}, it follows that
\begin{equation}\label{proof-eq-expansion-2-73}
    \begin{aligned}
        &\gamma\int_{\Omega}\int_{\Omega}\frac{\psi_{\xi_{\varepsilon},\lambda_{\varepsilon}}^{6-\alpha}(y)\psi_{\xi_{\varepsilon},\lambda_{\varepsilon}}^{5-\alpha}(x)\partial_{\lambda}P\bar{U}_{\xi_{\varepsilon},\lambda_{\varepsilon}}(x)}{|x-y|^{\alpha}}dydx\\
        &=\gamma\int_{\Omega}\int_{\Omega}\frac{\psi_{\xi_{\varepsilon},\lambda_{\varepsilon}}^{6-\alpha}(y)\psi_{\xi_{\varepsilon},\lambda_{\varepsilon}}^{5-\alpha}(x)\partial_{\lambda}\bar{U}_{\xi_{\varepsilon},\lambda_{\varepsilon}}(x)}{|x-y|^{\alpha}}dydx\\
        &\quad-\gamma\bar{C}_{\alpha}\int_{\Omega}\int_{\Omega}\frac{\psi_{\xi_{\varepsilon},\lambda_{\varepsilon}}^{6-\alpha}(y)\psi_{\xi_{\varepsilon},\lambda_{\varepsilon}}^{5-\alpha}(x)\partial_{\lambda}{\varphi}_{\xi_{\varepsilon},\lambda_{\varepsilon}}(x)}{|x-y|^{\alpha}}dydx.       
    \end{aligned}
\end{equation}
Using \eqref{important-identity-1}, \eqref{eq-psi-1} and \cite[Lemma B.3]{Frank2021BlowupOS}, the first term in \eqref{proof-eq-expansion-2-73} satisfies
\begin{equation*}
    \begin{aligned}
        &\gamma\int_{\Omega}\int_{\Omega}\frac{\psi_{\xi_{\varepsilon},\lambda_{\varepsilon}}^{6-\alpha}(y)\psi_{\xi_{\varepsilon},\lambda_{\varepsilon}}^{5-\alpha}(x)\partial_{\lambda}\bar{U}_{\xi_{\varepsilon},\lambda_{\varepsilon}}(x)}{|x-y|^{\alpha}}dydx\\
        &=\gamma\int_{\Omega}\int_{\Omega}\frac{\bar{U}_{\xi_{\varepsilon},\lambda_{\varepsilon}}^{6-\alpha}(y)\bar{U}_{\xi_{\varepsilon},\lambda_{\varepsilon}}^{5-\alpha}(x)\partial_{\lambda}\bar{U}_{\xi_{\varepsilon},\lambda_{\varepsilon}}(x)}{|x-y|^{\alpha}}dydx+o(\lambda_{\varepsilon}^{-2})\\
         &\quad+4(5-\alpha)\pi\bar{C}_{\alpha}\gamma\lambda_{\varepsilon}^{-\frac{1}{2}}\int_{\Omega}\int_{\Omega}\frac{\bar{U}_{\xi_{\varepsilon},\lambda_{\varepsilon}}^{6-\alpha}(y)\bar{U}_{\xi_{\varepsilon},\lambda_{\varepsilon}}^{4-\alpha}(x)H_{a}(\xi_{\varepsilon},x)\partial_{\lambda}\bar{U}_{\xi_{\varepsilon},\lambda_{\varepsilon}}(x)}{|x-y|^{\alpha}}dydx\\
          &\quad+4(6-\alpha)\pi\bar{C}_{\alpha}\gamma\lambda_{\varepsilon}^{-\frac{1}{2}}\int_{\Omega}\int_{\Omega}\frac{\bar{U}_{\xi_{\varepsilon},\lambda_{\varepsilon}}^{5-\alpha}(y)H_{a}(\xi_{\varepsilon},y)\bar{U}_{\xi_{\varepsilon},\lambda_{\varepsilon}}^{5-\alpha}(x)\partial_{\lambda}\bar{U}_{\xi_{\varepsilon},\lambda_{\varepsilon}}(x)}{|x-y|^{\alpha}}dydx\\
          &=3\bar{C}_{\alpha}^{2}\gamma\int_{\R^{3}}{U}_{\xi_{\varepsilon},\lambda_{\varepsilon}}^{5}(x)\partial_{\lambda}{U}_{\xi_{\varepsilon},\lambda_{\varepsilon}}(x)dx+o(\lambda_{\varepsilon}^{-2})\\
          &\quad+12(5-\alpha)\pi\bar{C}_{\alpha}^{2}\gamma\lambda_{\varepsilon}^{-\frac{1}{2}}\int_{\Omega}{U}_{\xi_{\varepsilon},\lambda_{\varepsilon}}^{4}(x)\partial_{\lambda}{U}_{\xi_{\varepsilon},\lambda_{\varepsilon}}(x)H_{a}(\xi_{\varepsilon},x)dx\\
          &\quad+4(6-\alpha)\pi\bar{C}_{\alpha}^{12-2\alpha}\gamma\lambda_{\varepsilon}^{-\frac{1}{2}}\int_{\Omega}\int_{\Omega}\frac{{U}_{\xi_{\varepsilon},\lambda_{\varepsilon}}^{5-\alpha}(y)H_{a}(\xi_{\varepsilon},y){U}_{\xi_{\varepsilon},\lambda_{\varepsilon}}^{5-\alpha}(x)\partial_{\lambda}{U}_{\xi_{\varepsilon},\lambda_{\varepsilon}}(x)}{|x-y|^{\alpha}}dydx\\
          &=-\frac{32}{15}(5-\alpha)S_{HL}^{\frac{6-\alpha}{5-\alpha}}\gamma\phi_{a}(\xi_{\varepsilon})\lambda_{\varepsilon}^{-2}+o(\lambda_{\varepsilon}^{-2})\\
          &\quad+4(6-\alpha)\pi\bar{C}_{\alpha}^{12-2\alpha}\gamma\lambda_{\varepsilon}^{-\frac{1}{2}}\int_{\Omega}\int_{\Omega}\frac{{U}_{\xi_{\varepsilon},\lambda_{\varepsilon}}^{5-\alpha}(y)H_{a}(\xi_{\varepsilon},y){U}_{\xi_{\varepsilon},\lambda_{\varepsilon}}^{5-\alpha}(x)\partial_{\lambda}{U}_{\xi_{\varepsilon},\lambda_{\varepsilon}}(x)}{|x-y|^{\alpha}}dydx.
    \end{aligned}
\end{equation*}
Moreover, it follows from \eqref{important-identity-1} and Lemma \ref{lem-H-a} that
\begin{equation*}
    \begin{aligned}
        &\int_{\Omega}\int_{\Omega}\frac{{U}_{\xi_{\varepsilon},\lambda_{\varepsilon}}^{5-\alpha}(y)H_{a}(\xi_{\varepsilon},y){U}_{\xi_{\varepsilon},\lambda_{\varepsilon}}^{5-\alpha}(x)\partial_{\lambda}{U}_{\xi_{\varepsilon},\lambda_{\varepsilon}}(x)}{|x-y|^{\alpha}}dydx\\
        &=\int_{B_{d_{\varepsilon}}(\xi_{\varepsilon})}\int_{B_{d_{\varepsilon}}(\xi_{\varepsilon})}\frac{{U}_{\xi_{\varepsilon},\lambda_{\varepsilon}}^{5-\alpha}(y)H_{a}(\xi_{\varepsilon},y){U}_{\xi_{\varepsilon},\lambda_{\varepsilon}}^{5-\alpha}(x)\partial_{\lambda}{U}_{\xi_{\varepsilon},\lambda_{\varepsilon}}(x)}{|x-y|^{\alpha}}dydx\\
         &\quad+\int_{\Omega\setminus B_{d_{\varepsilon}}(\xi_{\varepsilon})}\int_{B_{d_{\varepsilon}}(\xi_{\varepsilon})}\frac{{U}_{\xi_{\varepsilon},\lambda_{\varepsilon}}^{5-\alpha}(y)H_{a}(\xi_{\varepsilon},y){U}_{\xi_{\varepsilon},\lambda_{\varepsilon}}^{5-\alpha}(x)\partial_{\lambda}{U}_{\xi_{\varepsilon},\lambda_{\varepsilon}}(x)}{|x-y|^{\alpha}}dydx\\
         &\quad+\int_{\Omega}\int_{\Omega\setminus B_{d_{\varepsilon}}(\xi_{\varepsilon})}\frac{{U}_{\xi_{\varepsilon},\lambda_{\varepsilon}}^{5-\alpha}(y)H_{a}(\xi_{\varepsilon},y){U}_{\xi_{\varepsilon},\lambda_{\varepsilon}}^{5-\alpha}(x)\partial_{\lambda}{U}_{\xi_{\varepsilon},\lambda_{\varepsilon}}(x)}{|x-y|^{\alpha}}dydx\\
         &=\int_{B_{d_{\varepsilon}}(\xi_{\varepsilon})}\int_{B_{d_{\varepsilon}}(\xi_{\varepsilon})}\frac{{U}_{\xi_{\varepsilon},\lambda_{\varepsilon}}^{5-\alpha}(y)H_{a}(\xi_{\varepsilon},y){U}_{\xi_{\varepsilon},\lambda_{\varepsilon}}^{5-\alpha}(x)\partial_{\lambda}{U}_{\xi_{\varepsilon},\lambda_{\varepsilon}}(x)}{|x-y|^{\alpha}}dydx+o(\lambda_{\varepsilon}^{-\frac{3}{2}})\\
          &=\phi_{a}(\xi_{\varepsilon})\lambda_{\varepsilon}^{-\frac{3}{2}}\int_{\R^{3}}\int_{\R^{3}}\frac{{U}_{0,1}^{5-\alpha}(y){U}_{0,1}^{5-\alpha}(x)\partial_{\lambda}{U}_{0,\lambda}|_{\lambda=1}(x)}{|x-y|^{\alpha}}dydx+o(\lambda_{\varepsilon}^{-\frac{3}{2}})\\
          &=\frac{3\alpha}{(6-\alpha)\bar{C}_{\alpha}^{2(5-\alpha)}}\phi_{a}(\xi_{\varepsilon})\lambda_{\varepsilon}^{-\frac{3}{2}}\int_{\R^{3}}{U}_{0,1}^{4}\partial_{\lambda}{U}_{0,\lambda}|_{\lambda=1}dx+o(\lambda_{\varepsilon}^{-\frac{3}{2}})\\
          &=-\frac{2\alpha\pi}{5(6-\alpha)\bar{C}_{\alpha}^{2(5-\alpha)}}\phi_{a}(\xi_{\varepsilon})\lambda_{\varepsilon}^{-\frac{3}{2}}+o(\lambda_{\varepsilon}^{-\frac{3}{2}}).
    \end{aligned}
\end{equation*}
Thus the first term in \eqref{proof-eq-expansion-2-73} satisfies
\begin{equation*}
    \begin{aligned}
         \gamma\int_{\Omega}\int_{\Omega}\frac{\psi_{\xi_{\varepsilon},\lambda_{\varepsilon}}^{6-\alpha}(y)\psi_{\xi_{\varepsilon},\lambda_{\varepsilon}}^{5-\alpha}(x)\partial_{\lambda}\bar{U}_{\xi_{\varepsilon},\lambda_{\varepsilon}}(x)}{|x-y|^{\alpha}}dydx
         &=-\frac{32}{3}S_{HL}^{\frac{6-\alpha}{5-\alpha}}\gamma\phi_{a}(\xi_{\varepsilon})\lambda_{\varepsilon}^{-2}+o(\lambda_{\varepsilon}^{-2}).
    \end{aligned}
\end{equation*}
On the other hand, the second term in \eqref{proof-eq-expansion-2-73} satisfies
\begin{equation*}
    \begin{aligned}
        &-\gamma\bar{C}_{\alpha}\int_{\Omega}\int_{\Omega}\frac{\psi_{\xi_{\varepsilon},\lambda_{\varepsilon}}^{6-\alpha}(y)\psi_{\xi_{\varepsilon},\lambda_{\varepsilon}}^{5-\alpha}(x)\partial_{\lambda}{\varphi}_{\xi_{\varepsilon},\lambda_{\varepsilon}}(x)}{|x-y|^{\alpha}}dydx\\
        &=-2\pi\gamma\bar{C}_{\alpha}\lambda_{\varepsilon}^{-\frac{3}{2}}\int_{\Omega}\int_{\Omega}\frac{\bar{U}_{\xi_{\varepsilon},\lambda_{\varepsilon}}^{6-\alpha}(y)\bar{U}_{\xi_{\varepsilon},\lambda_{\varepsilon}}^{5-\alpha}(x)H_{0}(\xi_{\varepsilon},x)}{|x-y|^{\alpha}}dydx+o(\lambda_{\varepsilon}^{-2})\\
        &=-2\pi\gamma\bar{C}_{\alpha}\lambda_{\varepsilon}^{-\frac{3}{2}}\int_{\Omega}\int_{\R^{3}}\frac{\bar{U}_{\xi_{\varepsilon},\lambda_{\varepsilon}}^{6-\alpha}(y)\bar{U}_{\xi_{\varepsilon},\lambda_{\varepsilon}}^{5-\alpha}(x)H_{0}(\xi_{\varepsilon},x)}{|x-y|^{\alpha}}dydx+o(\lambda_{\varepsilon}^{-2})\\
         &=-6\pi\gamma\bar{C}_{\alpha}^{2}\lambda_{\varepsilon}^{-\frac{3}{2}}\int_{\Omega}{U}_{\xi_{\varepsilon},\lambda_{\varepsilon}}^{5}(x)H_{0}(\xi_{\varepsilon},x)dx+o(\lambda_{\varepsilon}^{-2})\\
          &=-\frac{32}{3}S_{HL}^{\frac{6-\alpha}{5-\alpha}}\phi_{0}(\xi_{\varepsilon})\gamma\lambda_{\varepsilon}^{-2}+o(\lambda_{\varepsilon}^{-2}).
    \end{aligned}
\end{equation*}
Therefore, the second term in \eqref{eq-proof-expansion-2-2} becomes
\begin{equation}\label{eq-proof-expansion-2-78}
    \begin{aligned}
         \gamma\int_{\Omega}\int_{\Omega}\frac{\psi_{\xi_{\varepsilon},\lambda_{\varepsilon}}^{6-\alpha}(y)\psi_{\xi_{\varepsilon},\lambda_{\varepsilon}}^{5-\alpha}(x)\partial_{\lambda}P\bar{U}_{\xi_{\varepsilon},\lambda_{\varepsilon}}(x)}{|x-y|^{\alpha}}dydx=-\frac{32}{3}S_{HL}^{\frac{6-\alpha}{5-\alpha}}(\phi_{a}(\xi_{\varepsilon})+\phi_{0}(\xi_{\varepsilon}))\gamma\lambda_{\varepsilon}^{-2}+o(\lambda_{\varepsilon}^{-2}).
    \end{aligned}
\end{equation}
For the third term in \eqref{eq-proof-expansion-2-2}, it follows from Lemma \ref{lemma PU} that
\begin{equation}\label{eq-proof-expansion-2-79}
    \begin{aligned}
        &\lambda_{\varepsilon}^{-3}\sum_{i=1}^{3}\delta_{i}\int_{\Omega}\int_{\Omega}\frac{\psi_{\xi_{\varepsilon},\lambda_{\varepsilon}}^{6-\alpha}(y)\psi_{\xi_{\varepsilon},\lambda_{\varepsilon}}^{5-\alpha}(x)\partial_{\xi_{i}}P\bar{U}_{\xi_{\varepsilon},\lambda_{\varepsilon}}(x)}{|x-y|^{\alpha}}dydx\\
        &=\lambda_{\varepsilon}^{-3}\sum_{i=1}^{3}\delta_{i}\int_{\Omega}\int_{\Omega}\frac{\psi_{\xi_{\varepsilon},\lambda_{\varepsilon}}^{6-\alpha}(y)\psi_{\xi_{\varepsilon},\lambda_{\varepsilon}}^{5-\alpha}(x)\partial_{\xi_{i}}\bar{U}_{\xi,\lambda_{\varepsilon}}(x)}{|x-y|^{\alpha}}dydx\\
        &\quad-\lambda_{\varepsilon}^{-3}\bar{C}_{\alpha}\sum_{i=1}^{3}\delta_{i}\int_{\Omega}\int_{\Omega}\frac{\psi_{\xi_{\varepsilon},\lambda_{\varepsilon}}^{6-\alpha}(y)\psi_{\xi_{\varepsilon},\lambda_{\varepsilon}}^{5-\alpha}(x)\partial_{\xi_{i}}{\varphi}_{\xi_{\varepsilon},\lambda_{\varepsilon}}(x)}{|x-y|^{\alpha}}dydx\\
        &=o(\lambda_{\varepsilon}^{-2}).
    \end{aligned}
\end{equation}
Combining \eqref{eq-proof-expansion-2-76}, \eqref{eq-proof-expansion-2-78} and \eqref{eq-proof-expansion-2-79}, we obtain
\begin{equation*}
    \begin{aligned}
        &\int_{\Omega}\int_{\Omega}\frac{\psi_{\xi_{\varepsilon},\lambda_{\varepsilon}}^{6-\alpha}(y)\psi_{\xi_{\varepsilon},\lambda_{\varepsilon}}^{5-\alpha}(x)s_{\varepsilon}(x)}{|x-y|^{\alpha}}dydx\\
        &=S_{HL}^{\frac{6-\alpha}{5-\alpha}}\beta\lambda^{-1}_{\varepsilon}+\frac{64}{3}(11-2\alpha)S_{HL}^{\frac{6-\alpha}{5-\alpha}}\beta\lambda^{-2}_{\varepsilon}\phi_{a}(\xi_{\varepsilon})+\frac{64}{3}S_{HL}^{\frac{6-\alpha}{5-\alpha}}\beta\lambda_{\varepsilon}^{-2}\phi_{0}(\xi_{\varepsilon})\\
         &\quad-\frac{32}{3}S_{HL}^{\frac{6-\alpha}{5-\alpha}}(\phi_{a}(\xi_{\varepsilon})+\phi_{0}(\xi_{\varepsilon}))\gamma\lambda_{\varepsilon}^{-2}+o(\lambda_{\varepsilon}^{-2}).
    \end{aligned}
\end{equation*}
From the expansion of $s_{\varepsilon}$ (see \eqref{expansion-s}), the second term in $D_{1}$ satisfies
\begin{equation*}
    \begin{aligned}
        &\int_{\Omega}\int_{\Omega}\frac{\psi_{\xi_{\varepsilon},\lambda_{\varepsilon}}^{5-\alpha}(y)s_{\varepsilon}(y)\psi_{\xi_{\varepsilon},\lambda_{\varepsilon}}^{5-\alpha}(x)s_{\varepsilon}(x)}{|x-y|^{\alpha}}dydx\\
        &=\lambda^{-2}_{\varepsilon}\beta^{2}\int_{\Omega}\int_{\Omega}\frac{\psi_{\xi_{\varepsilon},\lambda_{\varepsilon}}^{5-\alpha}(y)P\bar{U}_{\xi_{\varepsilon},\lambda_{\varepsilon}}(y)\psi_{\xi_{\varepsilon},\lambda_{\varepsilon}}^{5-\alpha}(x)P\bar{U}_{\xi_{\varepsilon},\lambda_{\varepsilon}}(x)}{|x-y|^{\alpha}}dydx\\
         &\quad+\gamma^{2}\int_{\Omega}\int_{\Omega}\frac{\psi_{\xi_{\varepsilon},\lambda_{\varepsilon}}^{5-\alpha}(y)\partial_{\lambda}P\bar{U}_{\xi_{\varepsilon},\lambda_{\varepsilon}}(y)\psi_{\xi_{\varepsilon},\lambda_{\varepsilon}}^{5-\alpha}(x)\partial_{\lambda}P\bar{U}_{\xi_{\varepsilon},\lambda_{\varepsilon}}(x)}{|x-y|^{\alpha}}dydx\\
        &\quad+2\lambda^{-1}_{\varepsilon}\beta\gamma\int_{\Omega}\int_{\Omega}\frac{\psi_{\xi_{\varepsilon},\lambda_{\varepsilon}}^{5-\alpha}(y)P\bar{U}_{\xi_{\varepsilon},\lambda_{\varepsilon}}(y)\psi_{\xi_{\varepsilon},\lambda_{\varepsilon}}^{5-\alpha}(x)\partial_{\lambda}P\bar{U}_{\xi_{\varepsilon},\lambda_{\varepsilon}}(x)}{|x-y|^{\alpha}}dydx+o(\lambda_{\varepsilon}^{-2}).
    \end{aligned}
\end{equation*}
By \eqref{important-identity-1} and some direct computations, we obtain
\begin{equation*}
    \begin{aligned}
        &\lambda^{-2}_{\varepsilon}\beta^{2}\int_{\Omega}\int_{\Omega}\frac{\psi_{\xi_{\varepsilon},\lambda_{\varepsilon}}^{5-\alpha}(y)P\bar{U}_{\xi_{\varepsilon},\lambda_{\varepsilon}}(y)\psi_{\xi_{\varepsilon},\lambda_{\varepsilon}}^{5-\alpha}(x)P\bar{U}_{\xi_{\varepsilon},\lambda_{\varepsilon}}(x)}{|x-y|^{\alpha}}dydx\\
        &=\lambda^{-2}_{\varepsilon}\beta^{2}\int_{\R^{3}}\int_{\R^{3}}\frac{\bar{U}^{6-\alpha}_{\xi,\lambda_{\varepsilon}}(y)\bar{U}^{6-\alpha}_{\xi_{\varepsilon},\lambda_{\varepsilon}}(x)}{|x-y|^{\alpha}}dydx+o(\lambda_{\varepsilon}^{-2})\\
        &=S_{HL}^{\frac{6-\alpha}{5-\alpha}}\beta^{2}\lambda^{-2}_{\varepsilon}+o(\lambda_{\varepsilon}^{-2})
    \end{aligned}
\end{equation*}
\begin{equation*}
    \begin{aligned}
        &\gamma^{2}\int_{\Omega}\int_{\Omega}\frac{\psi_{\xi_{\varepsilon},\lambda_{\varepsilon}}^{5-\alpha}(y)\partial_{\lambda}P\bar{U}_{\xi_{\varepsilon},\lambda_{\varepsilon}}(y)\psi_{\xi_{\varepsilon},\lambda_{\varepsilon}}^{5-\alpha}(x)\partial_{\lambda}P\bar{U}_{\xi_{\varepsilon},\lambda_{\varepsilon}}(x)}{|x-y|^{\alpha}}dydx\\
        &=\gamma^{2}\int_{\Omega}\int_{\Omega}\frac{\bar{U}_{\xi_{\varepsilon},\lambda_{\varepsilon}}^{5-\alpha}(y)\partial_{\lambda}\bar{U}_{\xi_{\varepsilon},\lambda_{\varepsilon}}(y)\bar{U}_{\xi_{\varepsilon},\lambda_{\varepsilon}}^{5-\alpha}(x)\partial_{\lambda}\bar{U}_{\xi_{\varepsilon},\lambda_{\varepsilon}}(x)}{|x-y|^{\alpha}}dydx+o(\lambda_{\varepsilon}^{-2})\\
        &=\gamma^{2}\bar{C}_{\alpha}^{2(6-\alpha)}\lambda_{\varepsilon}^{-2}\int_{\R^{3}}\int_{\R^{3}}\frac{{U}_{0,1}^{5-\alpha}(y)\partial_{\lambda}{U}_{0,\lambda}|_{\lambda=1}(y){U}_{0,1}^{5-\alpha}(x)\partial_{\lambda}{U}_{0,\lambda}|_{\lambda=1}(x)}{|x-y|^{\alpha}}dydx+o(\lambda_{\varepsilon}^{-2})\\
        &=\frac{3\alpha}{(6-\alpha)}\bar{C}_{\alpha}^{2}\gamma^{2}\lambda_{\varepsilon}^{-2}\int_{\R^{3}}{U}_{0,1}^{4}(\partial_{\lambda}{U}_{0,\lambda}|_{\lambda=1})^{2}dx+o(\lambda_{\varepsilon}^{-2})\\
        &=\frac{\alpha}{16(6-\alpha)}S_{HL}^{\frac{6-\alpha}{5-\alpha}}\gamma^{2}\lambda_{\varepsilon}^{-2}+o(\lambda_{\varepsilon}^{-2})\\
    \end{aligned}
\end{equation*}
and
\begin{equation*}
    \begin{aligned}
        &2\lambda^{-1}_{\varepsilon}\beta\gamma\int_{\Omega}\int_{\Omega}\frac{\psi_{\xi_{\varepsilon},\lambda_{\varepsilon}}^{5-\alpha}(y)P\bar{U}_{\xi_{\varepsilon},\lambda_{\varepsilon}}(y)\psi_{\xi_{\varepsilon},\lambda_{\varepsilon}}^{5-\alpha}(x)\partial_{\lambda}P\bar{U}_{\xi_{\varepsilon},\lambda_{\varepsilon}}(x)}{|x-y|^{\alpha}}dydx\\
        &=2\lambda^{-1}_{\varepsilon}\beta\gamma\int_{\Omega}\int_{\Omega}\frac{\bar{U}_{\xi_{\varepsilon},\lambda_{\varepsilon}}^{6-\alpha}(y)(y)\bar{U}_{\xi_{\varepsilon},\lambda_{\varepsilon}}^{5-\alpha}(x)\partial_{\lambda}\bar{U}_{\xi_{\varepsilon},\lambda_{\varepsilon}}(x)}{|x-y|^{\alpha}}dydx+o(\lambda_{\varepsilon}^{-2})\\
        &=2\lambda^{-1}_{\varepsilon}\beta\gamma\int_{\Omega}\int_{\R^{3}}\frac{\bar{U}_{\xi_{\varepsilon},\lambda_{\varepsilon}}^{6-\alpha}(y)(y)\bar{U}_{\xi_{\varepsilon},\lambda_{\varepsilon}}^{5-\alpha}(x)\partial_{\lambda}\bar{U}_{\xi_{\varepsilon},\lambda_{\varepsilon}}(x)}{|x-y|^{\alpha}}dydx+o(\lambda_{\varepsilon}^{-2})\\
        &=6\bar{C}_{\alpha}^{2}\beta\gamma\lambda^{-1}_{\varepsilon}\int_{\R^{3}}{U}_{\xi_{\varepsilon},\lambda_{\varepsilon}}^{5}\partial_{\lambda}{U}_{\xi_{\varepsilon},\lambda_{\varepsilon}}dx-6\bar{C}_{\alpha}^{2}\beta\gamma\lambda^{-1}_{\varepsilon}\int_{\R^{3}\setminus\Omega}{U}_{\xi_{\varepsilon},\lambda_{\varepsilon}}^{5}\partial_{\lambda}{U}_{\xi_{\varepsilon},\lambda_{\varepsilon}}dx+o(\lambda_{\varepsilon}^{-2})\\
        &=o(\lambda_{\varepsilon}^{-2}).
    \end{aligned}
\end{equation*}
It then follows that
\begin{equation*}
    \begin{aligned}
        \int_{\Omega}\int_{\Omega}\frac{\psi_{\xi_{\varepsilon},\lambda_{\varepsilon}}^{5-\alpha}(y)s_{\varepsilon}(y)\psi_{\xi_{\varepsilon},\lambda_{\varepsilon}}^{5-\alpha}(x)s_{\varepsilon}(x)}{|x-y|^{\alpha}}dydx=S_{HL}^{\frac{6-\alpha}{5-\alpha}}\beta^{2}\lambda^{-2}_{\varepsilon}+\frac{\alpha}{16(6-\alpha)}S_{HL}^{\frac{6-\alpha}{5-\alpha}}\gamma^{2}\lambda_{\varepsilon}^{-2}+o(\lambda_{\varepsilon}^{-2}).
    \end{aligned}
\end{equation*}
Finally, the third term in $D_{1}$ satisfies
\begin{equation*}
    \begin{aligned}
        &\int_{\Omega}\int_{\Omega}\frac{\psi_{\xi_{\varepsilon},\lambda_{\varepsilon}}^{6-\alpha}(y)\psi_{\xi_{\varepsilon},\lambda_{\varepsilon}}^{4-\alpha}(x)s^{2}_{\varepsilon}(x)}{|x-y|^{\alpha}}dydx\\
        &=\lambda^{-2}_{\varepsilon}\beta^{2}\int_{\R^{3}}\int_{\R^{3}}\frac{\bar{U}_{\xi_{\varepsilon},\lambda_{\varepsilon}}^{6-\alpha}(y)\bar{U}_{\xi_{\varepsilon},\lambda_{\varepsilon}}^{6-\alpha}(x)}{|x-y|^{\alpha}}dydx+o(\lambda_{\varepsilon}^{-2})\\ 
       &\quad+\gamma^{2}\bar{C}_{\alpha}^{2(6-\alpha)}\lambda_{\varepsilon}^{-2}\int_{\R^{3}}\int_{\R^{3}}\frac{{U}_{0,1}^{6-\alpha}(y){U}_{0,1}^{4-\alpha}(x)(\partial_{\lambda}{U}_{0,\lambda})|_{\lambda=1}^{2}(x)}{|x-y|^{\alpha}}dydx\\
       &=S_{HL}^{\frac{6-\alpha}{5-\alpha}}\lambda^{-2}_{\varepsilon}\beta^{2}+3\bar{C}_{\alpha}^{2}\gamma^{2}\lambda_{\varepsilon}^{-2}\int_{\R^{3}}{U}_{0,1}^{4}(\partial_{\lambda}{U}_{0,\lambda}|_{\lambda=1})^{2}dx+o(\lambda_{\varepsilon}^{-2})\\
       &=S_{HL}^{\frac{6-\alpha}{5-\alpha}}\lambda^{-2}_{\varepsilon}\beta^{2}+\frac{1}{16}S_{HL}^{\frac{6-\alpha}{5-\alpha}}\gamma^{2}\lambda_{\varepsilon}^{-2}+o(\lambda_{\varepsilon}^{-2}).
    \end{aligned}
\end{equation*}
Combining all the estimates above and using \eqref{eq-coefficient-convergence}, we conclude that
\begin{equation*}
    \begin{aligned}
        D_{1}
         &=2(6-\alpha)S_{HL}^{\frac{6-\alpha}{5-\alpha}}\beta\lambda^{-1}_{\varepsilon}+\frac{128}{3}(6-\alpha)S_{HL}^{\frac{6-\alpha}{5-\alpha}}\phi_{0}(\xi_{\varepsilon})\beta\lambda_{\varepsilon}^{-2}-\frac{64}{3}(6-\alpha)S_{HL}^{\frac{6-\alpha}{5-\alpha}}\phi_{0}(\xi_{\varepsilon})\gamma\lambda_{\varepsilon}^{-2}\\
         &\quad+(6-\alpha)(11-2\alpha)S_{HL}^{\frac{6-\alpha}{5-\alpha}}\beta^{2}\lambda^{-2}_{\varepsilon}+\frac{5}{16}(6-\alpha)S_{HL}^{\frac{6-\alpha}{5-\alpha}}\gamma^{2}\lambda_{\varepsilon}^{-2}\\
         &\quad+o(\lambda_{\varepsilon}^{-2}).
    \end{aligned}
\end{equation*}
On the other hand, arguing as in \cite[Appendix]{Frank2019Energythree-dimensional}, we have
\begin{equation*}
    \begin{aligned}
        N_{1}&=2S_{HL}^{\frac{6-\alpha}{5-\alpha}}\beta \lambda_{\varepsilon}^{-1}+\frac{128}{3}S_{HL}^{\frac{6-\alpha}{5-\alpha}}\phi_{0}(\xi_{\varepsilon})\beta\lambda_{\varepsilon}^{-2}-\frac{64}{3}S_{HL}^{\frac{6-\alpha}{5-\alpha}}\phi_{0}(\xi_{\varepsilon})\gamma\lambda_{\varepsilon}^{-2}\\
        &\quad+S_{HL}^{\frac{6-\alpha}{5-\alpha}}\beta^{2}\lambda_{\varepsilon}^{-2}+\frac{5}{16}S_{HL}^{\frac{6-\alpha}{5-\alpha}}\gamma^{2}\lambda_{\varepsilon}^{-2}+o(\lambda_{\varepsilon}^{-2}).
    \end{aligned}
\end{equation*}
Notice that by Proposition \ref{prop-upper-estimate}, we have
\begin{equation*}
    \left(1-\frac{N_{0}}{D_{0}}\right)\frac{D_{1}}{6-\alpha}\left(1-\frac{(7-\alpha)D_{1}}{2(6-\alpha)D_{0}}\right)=o(\lambda_{\varepsilon}^{-2})
\end{equation*}
and so
\begin{equation*}
    \begin{aligned}
        &N_{1}-\frac{1}{6-\alpha}\frac{D_{1}}
        {D_{0}}N_{1}-\frac{1}{6-\alpha}\frac{D_{1}}
        {D_{0}}N_{0}+\frac{7-\alpha}{2(6-\alpha)^{2}}\frac{D_{1}^{2}}{D_{0}^{2}}N_{0}\\
        &=\left(N_{1}-\frac{D_{1}}{6-\alpha}\right)\left(1-\frac{D_{1}}{(6-\alpha)D_{0}}\right)+\frac{(5-\alpha)D_{1}^{2}}{2(6-\alpha)^{2}D_{0}}+\left(1-\frac{N_{0}}{D_{0}}\right)\frac{D_{1}}{6-\alpha}\left(1-\frac{(7-\alpha)D_{1}}{2(6-\alpha)D_{0}}\right)\\
         &=\left(N_{1}-\frac{D_{1}}{6-\alpha}\right)\left(1-\frac{D_{1}}{(6-\alpha)D_{0}}\right)+\frac{(5-\alpha)D_{1}^{2}}{2(6-\alpha)^{2}D_{0}}+o(\lambda_{\varepsilon}^{-2}).
    \end{aligned}
\end{equation*}
Moreover, we see that
\begin{equation*}
    \begin{aligned}
        N_{1}-\frac{D_{1}}{6-\alpha}&=S_{HL}^{\frac{6-\alpha}{5-\alpha}}\beta^{2}\lambda_{\varepsilon}^{-2}
        -(11-2\alpha)S_{HL}^{\frac{6-\alpha}{5-\alpha}}\beta^{2}\lambda^{-2}_{\varepsilon}+o(\lambda_{\varepsilon}^{-2})
    \end{aligned}
\end{equation*}
and
\begin{equation*}
    \begin{aligned}
        \frac{(5-\alpha)D_{1}^{2}}{2(6-\alpha)^{2}D_{0}}&=\frac{5-\alpha}{2(6-\alpha)^{2}}\frac{4(6-\alpha)^{2}S_{HL}^{\frac{2(6-\alpha)}{5-\alpha}}\beta^{2}}{S_{HL}^{\frac{6-\alpha}{5-\alpha}}}\lambda^{-2}_{\varepsilon}+o(\lambda_{\varepsilon}^{-2})=2(5-\alpha)S_{HL}^{\frac{6-\alpha}{5-\alpha}}\beta^{2}\lambda^{-2}_{\varepsilon}+o(\lambda_{\varepsilon}^{-2}).
    \end{aligned}
\end{equation*}
It then follows that
\begin{equation*}
    N_{1}-\frac{1}{6-\alpha}\frac{N_{0}D_{1}}
        {D_{0}}-\frac{1}{6-\alpha}\frac{N_{1}D_{1}}
        {D_{0}}+\frac{7-\alpha}{2(6-\alpha)^{2}}\frac{N_{0}D_{1}^{2}}{D_{0}^{2}}=o(\lambda_{\varepsilon}^{-2}).
\end{equation*}
This, together with Lemma \ref{lem-energy-expansion-2} and Proposition \ref{prop-upper-estimate}, yields that
\begin{equation*}
    \begin{aligned}
        &S_{HL}(a+\varepsilon V)[u_{\varepsilon}]\\
        &=S_{HL}(a+\varepsilon V)[\psi_{\xi_{\varepsilon},\lambda_{\varepsilon}}]+S_{HL}^{-\frac{1}{5-\alpha}}\left(E_{0}[r_{\varepsilon}]-\frac{1}{6-\alpha}\frac{N_{0}}{D_{0}}I[r_{\varepsilon}]+o(\|\nabla r_{\varepsilon}\|_{L^{2}(\Omega)}^{2})\right)\\
        &\quad+o(\lambda_{\varepsilon}^{-2})+o(\varepsilon\lambda_{\varepsilon}^{-1}).
    \end{aligned}
\end{equation*}
This completes the proof.
\end{proof}

\begin{Lem}\label{lem-corecivity-2}
    There exists $\rho>0$ such that, as $\varepsilon\to0$,
    \begin{equation*}
    \begin{aligned}
        E_{0}[r_{\varepsilon}]-\frac{1}{6-\alpha}\frac{N_{0}}{D_{0}}I[r_{\varepsilon}]\geq \rho\|\nabla r_{\varepsilon}\|^{2}_{L^{2}(\Omega)}+o(\lambda_{\varepsilon}^{-2}).
    \end{aligned}
\end{equation*}
\end{Lem}

\begin{proof}
First, it follows from Proposition \ref{prop-upper-estimate} and Lemma \ref{lem-energy-expansion-2} that
\begin{equation*}
    \begin{aligned}
        E_{0}[r_{\varepsilon}]-\frac{1}{6-\alpha}\frac{N_{0}}{D_{0}}I[r_{\varepsilon}]&=\int_{\Omega}|\nabla r_{\varepsilon}|^{2}dx+\int_{\Omega}a r_{\varepsilon}^{2}dx\\
        &\quad-(6-\alpha)\int_{\Omega}\int_{\Omega}\frac{P\bar{U}_{\xi_{\varepsilon},\lambda_{\varepsilon}}^{5-\alpha}(y)r_{\varepsilon}(y)P\bar{U}_{\xi_{\varepsilon},\lambda_{\varepsilon}}^{5-\alpha}(x)r_{\varepsilon}(x)}{|x-y|^{\alpha}}dydx\\
        &\quad-(5-\alpha)\int_{\Omega}\int_{\Omega}\frac{P\bar{U}_{\xi_{\varepsilon},\lambda_{\varepsilon}}^{6-\alpha}(y)P\bar{U}_{\xi_{\varepsilon},\lambda_{\varepsilon}}^{4-\alpha}(x)r^{2}_{\varepsilon}(x)}{|x-y|^{\alpha}}dydx\\
        &\quad-2\int_{\Omega}\int_{\Omega}\frac{\psi_{\xi_{\varepsilon},\lambda_{\varepsilon}}^{6-\alpha}(y)\psi_{\xi_{\varepsilon},\lambda_{\varepsilon}}^{5-\alpha}(x)r_{\varepsilon}(x)}{|x-y|^{\alpha}}dydx+o(\|\nabla r_{\varepsilon}\|_{L^{2}(\Omega)}^{2}).
    \end{aligned}
\end{equation*}
Moreover, Lemma \ref{lemma coercivity} yields that
\begin{equation}\label{corrvity-proof-2}
    \begin{aligned}
         E_{0}[r_{\varepsilon}]-\frac{1}{6-\alpha}\frac{N_{0}}{D_{0}}I[r_{\varepsilon}]
         &\geq\rho\|\nabla r_{\varepsilon}\|^{2}_{L^{2}(\Omega)} -2\int_{\Omega}\int_{\Omega}\frac{\psi_{\xi_{\varepsilon},\lambda_{\varepsilon}}^{6-\alpha}(y)\psi_{\xi_{\varepsilon},\lambda_{\varepsilon}}^{5-\alpha}(x)r_{\varepsilon}(x)}{|x-y|^{\alpha}}dydx,
    \end{aligned}
\end{equation}
for some $\rho>0$. Notice that by \eqref{eq-psi-1}, we obtain
\begin{equation*}
    \begin{aligned}
        &\int_{\Omega}\int_{\Omega}\frac{\psi_{\xi_{\varepsilon},\lambda_{\varepsilon}}^{6-\alpha}(y)\psi_{\xi_{\varepsilon},\lambda_{\varepsilon}}^{5-\alpha}(x)r_{\varepsilon}(x)}{|x-y|^{\alpha}}dydx\\
        &=\int_{\Omega}\int_{\Omega}\frac{\bar{U}_{\xi_{\varepsilon},\lambda_{\varepsilon}}^{6-\alpha}(y)\bar{U}_{\xi_{\varepsilon},\lambda_{\varepsilon}}^{5-\alpha}(x)r_{\varepsilon}(x)}{|x-y|^{\alpha}}dydx+O\left\{\lambda_{\varepsilon}^{-\frac{1}{2}}\int_{\Omega}\int_{\Omega}\frac{\bar{U}_{\xi_{\varepsilon},\lambda_{\varepsilon}}^{6-\alpha}(y)\bar{U}_{\xi_{\varepsilon},\lambda_{\varepsilon}}^{4-\alpha}(x)H_{a}(\xi_{\varepsilon},x)r_{\varepsilon}(x)}{|x-y|^{\alpha}}dydx\right.\\
        &\quad+\left.\lambda_{\varepsilon}^{-\frac{1}{2}}\int_{\Omega}\int_{\Omega}\frac{\bar{U}_{\xi_{\varepsilon},\lambda_{\varepsilon}}^{5-\alpha}(y)H_{a}(\xi_{\varepsilon},y)\bar{U}_{\xi_{\varepsilon},\lambda_{\varepsilon}}^{5-\alpha}(x)r_{\varepsilon}(x)}{|x-y|^{\alpha}}dydx\right\}+o(\lambda_{\varepsilon}^{-2}).
    \end{aligned}
\end{equation*}
By \eqref{important-identity-1}, the HLS inequality, the H\"{o}lder inequality and the orthogonality of $r_{\varepsilon}$, we have
\begin{equation*}
    \begin{aligned}
        &\int_{\Omega}\int_{\Omega}\frac{\bar{U}_{\xi_{\varepsilon},\lambda_{\varepsilon}}^{6-\alpha}(y)\bar{U}_{\xi_{\varepsilon},\lambda_{\varepsilon}}^{5-\alpha}(x)r_{\varepsilon}(x)}{|x-y|^{\alpha}}dydx\\
        &=\int_{\Omega}\int_{\R^{3}}\frac{\bar{U}_{\xi_{\varepsilon},\lambda_{\varepsilon}}^{6-\alpha}(y)\bar{U}_{\xi_{\varepsilon},\lambda_{\varepsilon}}^{5-\alpha}(x)r_{\varepsilon}(x)}{|x-y|^{\alpha}}dydx-\int_{\Omega}\int_{\R^{3}\setminus\Omega}\frac{\bar{U}_{\xi_{\varepsilon},\lambda_{\varepsilon}}^{6-\alpha}(y)\bar{U}_{\xi_{\varepsilon},\lambda_{\varepsilon}}^{5-\alpha}(x)r_{\varepsilon}(x)}{|x-y|^{\alpha}}dydx\\
        &=3\bar{C}_{\alpha}\int_{\Omega}{U}^{5}_{\xi_{\varepsilon},\lambda_{\varepsilon}}r_{\varepsilon}dx+O\left(\|U_{\xi_{\varepsilon},\lambda_{\varepsilon}}\|_{L^{6}(\R^{3}\setminus\Omega)}^{6-\alpha}\|\nabla r_{\varepsilon}\|_{L^{2}(\Omega)}\right)\\
        &=o(\lambda_{\varepsilon}^{-2}).
    \end{aligned}
\end{equation*}
Moreover, it follows from \eqref{eq-coefficient-convergence}, the Young inequality, Lemma \ref{lem-H-a} and \cite[Lemma B.3]{Frank2021BlowupOS} that for any sufficiently small $\delta>0$, 
\begin{equation*}
    \begin{aligned}
        &\lambda_{\varepsilon}^{-\frac{1}{2}}\int_{\Omega}\int_{\Omega}\frac{\bar{U}_{\xi_{\varepsilon},\lambda_{\varepsilon}}^{6-\alpha}(y)\bar{U}_{\xi_{\varepsilon},\lambda_{\varepsilon}}^{4-\alpha}(x)H_{a}(\xi_{\varepsilon},x)r_{\varepsilon}(x)}{|x-y|^{\alpha}}dydx\\
        &=3\lambda_{\varepsilon}^{-\frac{1}{2}}\int_{\Omega}U_{\xi_{\varepsilon},\lambda_{\varepsilon}}^{4}(x)H_{a}(\xi_{\varepsilon},x)r_{\varepsilon}(x)dx\\
        &\leq \delta \int_{\Omega}U_{\xi_{\varepsilon},\lambda_{\varepsilon}}^{4}(x)r^{2}_{\varepsilon}(x)dx+C\lambda_{\varepsilon}^{-1}\int_{\Omega}U_{\xi_{\varepsilon},\lambda_{\varepsilon}}^{4}H_{a}^{2}(\xi_{\varepsilon},x)dx\\
        &\lesssim \delta\|\nabla r_{\varepsilon}\|_{L^{2}}^{2}+o(\lambda_{\varepsilon}^{-2})
    \end{aligned}
\end{equation*}
and
\begin{equation*}
    \begin{aligned}
        &\lambda_{\varepsilon}^{-\frac{1}{2}}\int_{\Omega}\int_{\Omega}\frac{\bar{U}_{\xi_{\varepsilon},\lambda_{\varepsilon}}^{5-\alpha}(y)H_{a}(\xi_{\varepsilon},y)\bar{U}_{\xi_{\varepsilon},\lambda_{\varepsilon}}^{5-\alpha}(x)r_{\varepsilon}(x)}{|x-y|^{\alpha}}dydx\\
        &=\lambda_{\varepsilon}^{-\frac{1}{2}}\int_{\Omega}\int_{B_{d_{\varepsilon}}(\xi_{\varepsilon})}\frac{\bar{U}_{\xi_{\varepsilon},\lambda_{\varepsilon}}^{5-\alpha}(y)H_{a}(\xi_{\varepsilon},y)\bar{U}_{\xi_{\varepsilon},\lambda_{\varepsilon}}^{5-\alpha}(x)r_{\varepsilon}(x)}{|x-y|^{\alpha}}dydx\\
        &\quad+\lambda_{\varepsilon}^{-\frac{1}{2}}\int_{\Omega}\int_{\Omega\setminus B_{d_{\varepsilon}}(\xi_{\varepsilon})}\frac{\bar{U}_{\xi_{\varepsilon},\lambda_{\varepsilon}}^{5-\alpha}(y)H_{a}(\xi_{\varepsilon},y)\bar{U}_{\xi_{\varepsilon},\lambda_{\varepsilon}}^{5-\alpha}(x)r_{\varepsilon}(x)}{|x-y|^{\alpha}}dydx\\
        &=\lambda_{\varepsilon}^{-\frac{1}{2}}\phi_{a}(\xi_{\varepsilon})\int_{\Omega}\int_{B_{d_{\varepsilon}}(\xi_{\varepsilon})}\frac{\bar{U}_{\xi_{\varepsilon},\lambda_{\varepsilon}}^{5-\alpha}(y)\bar{U}_{\xi_{\varepsilon},\lambda_{\varepsilon}}^{5-\alpha}(x)r_{\varepsilon}(x)}{|x-y|^{\alpha}}dydx\\
        &\quad+O\left(\lambda_{\varepsilon}^{-\frac{1}{2}}\int_{\Omega}\int_{B_{d_{\varepsilon}}(\xi_{\varepsilon})}\frac{\bar{U}_{\xi_{\varepsilon},\lambda_{\varepsilon}}^{5-\alpha}(y)|y-\xi_{\varepsilon}|\bar{U}_{\xi_{\varepsilon},\lambda_{\varepsilon}}^{5-\alpha}(x)r_{\varepsilon}(x)}{|x-y|^{\alpha}}dydx\right)+o(\lambda_{\varepsilon}^{-2})\\
        &\lesssim \delta\|\nabla r_{\varepsilon}\|_{L^{2}(\Omega)}^{2}+o(\lambda_{\varepsilon}^{-2}).
    \end{aligned}
\end{equation*}
Therefore, for any sufficiently small $\delta>0$, it holds that
\begin{equation*}
    \begin{aligned}
        \int_{\Omega}\int_{\Omega}\frac{\psi_{\xi_{\varepsilon},\lambda_{\varepsilon}}^{6-\alpha}(y)\psi_{\xi_{\varepsilon},\lambda_{\varepsilon}}^{5-\alpha}(x)r_{\varepsilon}(x)}{|x-y|^{\alpha}}dydx\leq\delta\|\nabla r_{\varepsilon}\|_{L^{2}(\Omega)}^{2}+o(\lambda_{\varepsilon}^{-2}).
    \end{aligned}
\end{equation*}
This, together with \eqref{corrvity-proof-2}, completes the proof.
\end{proof}

Combining Proposition \ref{prop-upper-estimate}, Proposition \ref{prop-lower-bound} and Lemma \ref{lem-corecivity-2}, we obtain
\begin{equation}\label{eq-energy-low-bound}
    \begin{aligned}
        &S_{HL}(a+\varepsilon V)\\
        &=S_{HL}+\frac{64}{3}S_{HL}Q_{V}(\xi_{0})\varepsilon\lambda_{\varepsilon}^{-1}-\frac{8}{3}S_{HL}a(\xi_{0})\lambda_{\varepsilon}^{-2}\\
        &\quad-\frac{64}{3}S_{HL}\phi_{a}(\xi_{\varepsilon})\lambda_{\varepsilon}^{-1}+S_{HL}^{-\frac{1}{5-\alpha}}\left(E_{0}[r_{\varepsilon}]-\frac{1}{6-\alpha}\frac{N_{0}}{D_{0}}I[r_{\varepsilon}]+o(\|\nabla r_{\varepsilon}\|_{L^{2}(\Omega)}^{2})\right)\\
         &\quad+o(\lambda_{\varepsilon}^{-2})+o(\varepsilon\lambda_{\varepsilon}^{-1})\\
         &\geq S_{HL}+\frac{64}{3}S_{HL}(Q_{V}(\xi_{0})+o(1))\varepsilon\lambda_{\varepsilon}^{-1}+\frac{8}{3}S_{HL}(-a(\xi_{0})+o(1))\lambda_{\varepsilon}^{-2}\\
         &\quad-\frac{64}{3}S_{HL}\phi_{a}(\xi_{\varepsilon})\lambda_{\varepsilon}^{-1}+\rho S_{HL}^{-\frac{1}{5-\alpha}}\|\nabla r_{\varepsilon}\|_{L^{2}(\Omega)}^{2}.
    \end{aligned}
\end{equation}
Let us define
\begin{equation*}
    R:=-\frac{64}{3}S_{HL}\phi_{a}(\xi_{\varepsilon})\lambda_{\varepsilon}^{-1}+\rho  S_{HL}^{-\frac{1}{5-\alpha}}\|\nabla r_{\varepsilon}\|_{L^{2}(\Omega)}^{2}.
\end{equation*}
Corollary \ref{corollary-phi-a} implies that $R\geq0$, and hence
\begin{equation*}
    \begin{aligned}
        0&\geq \frac{64}{3}S_{HL}(Q_{V}(\xi_{0})+o(1))\varepsilon\lambda_{\varepsilon}^{-1}+\frac{8}{3}S_{HL}(-a(\xi_{0})+o(1))\lambda_{\varepsilon}^{-2}.
    \end{aligned}
\end{equation*}

\begin{Lem}
    If $\mathcal{N}_{a}(V)\neq\emptyset$, then $\xi_{0}\in \mathcal{N}_{a}(V)$.
\end{Lem}

\begin{proof}
First, it follows from Proposition \ref{cor-energy-upper} that
\begin{equation*}
\begin{aligned}
         S_{HL}(a+\varepsilon V)\leq S_{HL}-\frac{128}{3}S_{HL}\sup_{\xi\in\mathcal{N}_{a}(V)}\frac{Q_{V}(\xi)^{2}}{|a(\xi)|}\varepsilon^{2}+o(\varepsilon^{2}).
\end{aligned}     
\end{equation*}
This, together with \eqref{eq-energy-low-bound} and the fact $R\geq0$, yields that
    \begin{equation*}
        \begin{aligned}
            \frac{64}{3}S_{HL}(-Q_{V}(\xi_{0})+o(1))\varepsilon\lambda_{\varepsilon}^{-1}\geq&\frac{8}{3}S_{HL}(-a(\xi_{0})+o(1))\lambda_{\varepsilon}^{-2}\\
            &+\frac{128}{3}S_{HL}\left(\sup_{\xi\in\mathcal{N}_{a}(V)}\frac{Q_{V}(\xi)^{2}}{|a(\xi)|}+o(1)\right)\varepsilon^{2}\\
            &:=A_{\varepsilon}\lambda_{\varepsilon}^{-2}+B_{\varepsilon}\varepsilon^{2}.
            \end{aligned}
    \end{equation*}
From \eqref{eq-coefficient-convergence} and Assumption \ref{assumption-a}, we find that $\phi_{a}(\xi_{0})=0$ and $a(\xi_{0})<0$. Thus $A_{\varepsilon}$ and $B_{\varepsilon}$ tend to some positive quantities as $\varepsilon\to 0$. It then follows that there exists a constant $C>0$ such that
\begin{equation*}
    \begin{aligned}
         \frac{64}{3}S_{HL}(-Q_{V}(\xi_{0})+o(1))\varepsilon\lambda_{\varepsilon}^{-1}\geq A_{\varepsilon}\lambda_{\varepsilon}^{-2}+B_{\varepsilon}\varepsilon^{2}\geq 2\sqrt {A_{\varepsilon}B_{\varepsilon}}\varepsilon \lambda_{\varepsilon}^{-1}\geq C\varepsilon \lambda_{\varepsilon}^{-1}.
    \end{aligned}
\end{equation*}
Thus $Q_{V}(\xi_{0})<0$ and the proof is completed.
\end{proof}

\begin{proof}[Proof of Theorems \ref{thm-3} and \ref{thm-6}]
Assume that $\mathcal{N}_{a}(V)\neq\emptyset$. Notice that
    \begin{equation*}
        \begin{aligned}
            &\frac{8}{3}S_{HL}\left(8Q_{V}(\xi_{0})\varepsilon\lambda_{\varepsilon}^{-1}-a(\xi_{0})\lambda_{\varepsilon}^{-2}\right)+o(\lambda_{\varepsilon}^{-2})+o(\varepsilon\lambda_{\varepsilon}^{-1})\\
            &=-\frac{8}{3}S_{HL}\frac{(4Q_{V}(\xi_{0})+o(1))^{2}}{(|a(\xi_{0})|+o(1))}\varepsilon^{2}\\
            &\quad+\frac{8}{3}S_{HL}\left(\frac{(4Q_{V}(\xi_{0})+o(1))}{(|a(\xi_{0})|+o(1))^{\frac{1}{2}}}\varepsilon+(|a(\xi_{0})|+o(1))^{\frac{1}{2}}\lambda_{\varepsilon}^{-1}\right)^{2}.
        \end{aligned}
    \end{equation*}
It then follows from \eqref{eq-energy-low-bound} and $R\geq0$ that
    \begin{equation}\label{eq-proof-energy-asym-0}
    \begin{aligned}
        S_{HL}(a+\varepsilon V)&= S_{HL}-\frac{8}{3}S_{HL}\frac{(4Q_{V}(\xi_{0})+o(1))^{2}}{(|a(\xi_{0})|+o(1))}\varepsilon^{2}\\
            &\quad+\frac{8}{3}S_{HL}\left(\frac{(4Q_{V}(\xi_{0})+o(1))}{(|a(\xi_{0})|+o(1))^{\frac{1}{2}}}\varepsilon+(|a(\xi_{0})|+o(1))^{\frac{1}{2}}\lambda_{\varepsilon}^{-1}\right)^{2}+R\\
            &\geq S_{HL}-\frac{8}{3}S_{HL}\frac{(4Q_{V}(\xi_{0})+o(1))^{2}}{(|a(\xi_{0})|+o(1))}\varepsilon^{2}.
    \end{aligned}
\end{equation}
Thus we have
\begin{equation}\label{eq-proof-energy-asym-1}
    \begin{aligned}
    S_{HL}-S_{HL}(a+\varepsilon V)
    &\leq\frac{128}{3}S_{HL}\frac{Q_{V}(\xi_{0})^{2}}{|a(\xi_{0})|}\varepsilon^{2}+o(\varepsilon^{2})\\
    &\leq \frac{128}{3}S_{HL}\sup_{\xi\in\mathcal{N}_{a}(V)}\frac{Q_{V}(\xi)^{2}}{|a(\xi)|}\varepsilon^{2}+o(\varepsilon^{2}).
    \end{aligned}
\end{equation}
On the other hand, Proposition \ref{cor-energy-upper} gives that
\begin{equation}\label{eq-proof-energy-asym-2}
\begin{aligned}
          S_{HL}-S_{HL}(a+\varepsilon V)\geq\frac{128}{3}S_{HL}\sup_{\xi\in\mathcal{N}_{a}(V)}\frac{Q_{V}(\xi)^{2}}{|a(\xi)|}\varepsilon^{2}+o(\varepsilon^{2}).
\end{aligned}     
\end{equation}
Combining \eqref{eq-proof-energy-asym-1} and \eqref{eq-proof-energy-asym-2}, we obtain \eqref{eq-energy-asym} and
    \begin{equation}\label{eq-location}
        \begin{aligned}
            \frac{Q_{V}(\xi_{0})^{2}}{|a(\xi_{0})|}=\sup_{\xi\in\mathcal{N}_{a}(V)}\frac{Q_{V}(\xi)^{2}}{|a(\xi)|}.
        \end{aligned}
    \end{equation}
Moreover, by \eqref{eq-energy-asym} and \eqref{eq-proof-energy-asym-0}, we obtain
 \begin{equation*}
    \begin{aligned}
        o(\varepsilon^{2})=\frac{8}{3}S_{HL}\left(\frac{(4Q_{V}(\xi_{0})+o(1))}{(|a(\xi_{0})|+o(1))^{\frac{1}{2}}}\varepsilon+(|a(\xi_{0})|+o(1))^{\frac{1}{2}}\lambda_{\varepsilon}^{-1}\right)^{2}+R
    \end{aligned}
\end{equation*}
It follows that
    \begin{equation*}
        \begin{aligned}
             R=o(\varepsilon^{2}),\quad \lambda_{\varepsilon}^{-1}=-\frac{4Q_{V}(\xi_{0})}{|a(\xi_{0})|}\varepsilon+o(\varepsilon)
        \end{aligned}
    \end{equation*}
and thus
    \begin{equation*}
      \|\nabla r_{\varepsilon}\|_{L^{2}}^{2}=o(\varepsilon^{2}),\quad \phi_{a}(\xi_{\varepsilon})=o(\varepsilon).
    \end{equation*}
Finally, by Lemma \ref{lem-energy-expansion-2} and Proposition \ref{prop-lower-bound}, we obtain
    \begin{equation}\label{eq-mu-var-1}
        \begin{aligned}
            S_{HL}^{\frac{6-\alpha}{5-\alpha}}(a+\varepsilon V)&=\int_{\Omega}\int_{\Omega}\frac{u_{\varepsilon}^{6-\alpha}(y)u_{\varepsilon}^{6-\alpha}(x)}{|x-y|^{\alpha}}dydx\\
            &=\mu_{\varepsilon}^{2(6-\alpha)}\left\{S_{HL}^{\frac{6-\alpha}{5-\alpha}}+2(6-\alpha)S_{HL}^{\frac{6-\alpha}{5-\alpha}}\beta\lambda^{-1}_{\varepsilon}\right\}+o(\varepsilon)\\
            &=\mu_{\varepsilon}^{2(6-\alpha)}\left\{S_{HL}^{\frac{6-\alpha}{5-\alpha}}+8(6-\alpha)S_{HL}^{\frac{6-\alpha}{5-\alpha}}\beta\frac{|Q_{V}(\xi_{0})|}{|a(\xi_{0})|}\varepsilon\right\}+o(\varepsilon).
        \end{aligned}
    \end{equation}
On the other hand, \eqref{eq-energy-asym} and \eqref{eq-location} imply that
\begin{equation}\label{eq-mu-var-2}
    \begin{aligned}
        S_{HL}^{\frac{6-\alpha}{5-\alpha}}(a+\varepsilon V) &=\left(S_{HL}-\frac{128}{3}S_{HL}\frac{Q_{V}(\xi_{0})^{2}}{|a(\xi_{0})|}\varepsilon^{2}+o(\varepsilon^{2})\right)^{\frac{6-\alpha}{5-\alpha}}\\
            &=S_{HL}^{\frac{6-\alpha}{5-\alpha}}-\frac{128}{3}\frac{6-\alpha}{5-\alpha}S_{HL}^{\frac{6-\alpha}{5-\alpha}}\frac{Q_{V}(\xi_{0})^{2}}{|a(\xi_{0})|}\varepsilon^{2}+o(\varepsilon^{2}).
    \end{aligned}
\end{equation}
Combining \eqref{eq-mu-var-1} and \eqref{eq-mu-var-2}, we conclude that
    \begin{equation*}
        \mu_{\varepsilon}=1-4\beta\frac{|Q_{V}(\xi_{0})|}{|a(\xi_{0})|}\varepsilon+o(\varepsilon)=1+\frac{256}{3}\phi_{0}(\xi_{0})\frac{|Q_{V}(\xi_{0})|}{|a(\xi_{0})|}\varepsilon+o(\varepsilon).
    \end{equation*}
This completes the proof.
\end{proof}

\begin{proof}[Proof of Theorem \ref{thm-4}]
Proposition \ref{prop of SHL and SHLa} shows that
\begin{equation*}
   S_{HL}(a+\varepsilon V)\leq S_{HL}.
\end{equation*}
\textbf{Case 1.} $ S_{HL}(a+\varepsilon V)<S_{HL}$ for any sufficiently small $\varepsilon>0$. Given $Q_{V}(\xi_{0})\geq0$, $a(\xi_{0})<0$, $R\geq0$, together with \eqref{eq-energy-low-bound} and the Young inequality, we have
\begin{equation}\label{eq-degeneracy-case-1}
    \begin{aligned}
        &S_{HL}(a+\varepsilon V)\\
         &\geq S_{HL}+\frac{64}{3}S_{HL}(Q_{V}(\xi_{0})+o(1))\varepsilon\lambda_{\varepsilon}^{-1}+\frac{8}{3}S_{HL}(-a(\xi_{0})+o(1))\lambda_{\varepsilon}^{-2}+R\\
         &\geq S_{HL}+\frac{8}{3}S_{HL}(|a(\xi_{0})|+o(1))\lambda_{\varepsilon}^{-2}+o(\varepsilon\lambda_{\varepsilon}^{-1})\\
         &\geq S_{HL}+o(\varepsilon^{2}).
    \end{aligned}
\end{equation} 
Thus we obtain $S_{HL}=S_{HL}(a+\varepsilon V)+o(\varepsilon^{2})$. \textbf{Case 2.} There exists $\{\varepsilon_{n}\}\subset(0,\infty)$ such that $ S_{HL}(a+\varepsilon V)=S_{HL}$ and $\varepsilon_{n}\to0$ as $n\to\infty$. Since $S_{HL}(a+\varepsilon V)$ is concave in $\varepsilon$, it follows that $S_{HL}(a+\varepsilon V)=S_{HL}$ for any sufficiently small $\varepsilon>0$. Moreover, when $Q_{V}(\xi_{0})>0$, if $S_{HL}(a+\varepsilon V)<S_{HL}$, then \eqref{eq-degeneracy-case-1} implies that $S_{HL}(a+\varepsilon V)\geq S_{HL}$, which is a contradiction. Therefore, {Case 1} cannot occur. This completes the proof.
\end{proof}

\noindent\textbf{Data Availability} This manuscript has no associated data

\noindent\textbf{Conflicts of Interest} The author has no conflicts of interest to declare that are relevant to this article.


\end{document}